\documentclass[a4paper, 10pt]{article}

\usepackage[a4paper,left=3cm,right=3cm,top=2.5cm,bottom=2.5cm]{geometry}

\usepackage{subfig}
\usepackage{pgfplots}
\usepackage{pgfplotstable}
\usepackage{mathtools}
\usepackage{multicol}
\usepackage{comment}
\usepackage{booktabs}
\pgfplotsset{compat=1.5}
\usepackage{amssymb}
\usepackage{url}
\usepackage{bm}

\usepackage{pdfsync}
\usepackage{float}
\usepackage{tabularx}
\usepackage{enumerate}
\usepackage{array}
\usepackage{xspace}
\usepackage{tikz}
\usepackage{tikz-cd}
\usepackage{tikzsymbols}
\usetikzlibrary{calc,trees,positioning,arrows,chains,shapes.geometric,%
    decorations.pathreplacing,decorations.pathmorphing,shapes,%
    matrix,shapes.symbols, decorations.markings, patterns,fit}

\usepackage{authblk}
\usepackage{siunitx}
\usepackage{algorithm}
\usepackage{algpseudocode}

\usepackage{hyperref}
\usepackage[draft,inline,marginclue]{fixme}
\usepackage{mathrsfs}
\usepackage{color, colortbl}
\usepackage{multirow}

\usepackage{amsthm}
\usepackage{stmaryrd}
\usepackage{graphicx}
\usepackage{booktabs}

\FXRegisterAuthor{mt}{amt}{MT}
\FXRegisterAuthor{fr}{afr}{FR}
\FXRegisterAuthor{al}{aal}{AL}

\newtheorem{theorem}{Theorem}

\newtheorem{definition}{Definition}
\newtheorem{proposition}{Proposition}

\theoremstyle{remark}

\newcommand{\x}{\ensuremath{\mathbf{x}}}
\newcommand{\ub}{\ensuremath{\mathbf{u}}}

\newcommand{\y}{\ensuremath{\mathbf{y}}}
\newcommand{\z}{\ensuremath{\mathbf{z}}}
\newcommand{\X}{\ensuremath{\mathbf{X}}}

\newcommand{\Z}{\ensuremath{\mathbf{Z}}}
\newcommand{\R}{\ensuremath{\mathbb{R}}}

\DeclareMathOperator*{\argmin}{\arg\!\min}
\newcolumntype{C}[1]{>{\centering\arraybackslash}m{#1}}
\definecolor{Gray}{gray}{0.9}

\newcommand{\RA}[1]{{\color{black}#1}}
\newcommand{\RB}[1]{{\color{black}#1}}
\newcommand{\RC}[1]{{\color{black}#1}}

\begin{document}

\title{Kernel-based Active Subspaces with application to CFD problems
  using Discontinuous Galerkin method} 

\author[]{Francesco~Romor\footnote{francesco.romor@sissa.it}}
\author[]{Marco~Tezzele\footnote{marco.tezzele@sissa.it}}
\author[]{Andrea~Lario\footnote{andrea.lario@sissa.it}}
\author[]{Gianluigi~Rozza\footnote{gianluigi.rozza@sissa.it}}

\affil{Mathematics Area, mathLab, SISSA, via Bonomea 265, I-34136
  Trieste, Italy}

\maketitle

\begin{abstract}
Nonlinear extensions to the active subspaces method have brought remarkable
results for dimension reduction in the parameter space and response
surface design. We further develop a kernel-based nonlinear method. In particular
we introduce it in a broader mathematical framework that contemplates also the
reduction in parameter space of multivariate objective functions. The
implementation is thoroughly discussed and tested on more challenging
benchmarks than the ones already present in the literature, for which dimension
reduction with active subspaces produces already good results. Finally, we show
a whole pipeline for the design of response surfaces with the new methodology in
the context of a parametric CFD application solved with the Discontinuous
Galerkin method.
\end{abstract}

\tableofcontents

\section{Introduction}
\label{sec:intro}
Nowadays, in many industrial settings the simulation of complex
systems requires a huge amount of computational power. Problems
involving high-fidelity simulations are usually large-scale, moreover
the number of solutions required increases with the number of
parameters. In this context, we mention optimization tasks, inverse problems,
optimal control problems, and uncertainty
quantification; they all suffer from the curse of dimensionality, that
is, in this case, the computational time grows exponentially with the
dimension of the input parameter space. Data-driven reduced order
methods (ROM)~\cite{brunton2019data, rozza2018advances,
salmoiraghi2016advances} have been developed to deal with such costly outer loop
applications for parametric PDEs, but the limit for high dimensional parameter spaces remains.

One approach to alleviate the curse of dimensionality is to identify
and exploit some notion of low-dimensional structure of the
model or objective function that maps the inputs to the outputs of interest. A possible linear input coordinate transformation technique
is the Sliced Inverse Regression (SIR)~\cite{li1991sliced} approach
and its extensions~\cite{cook2005sufficient, li2007sparse,
wu2009localized}. Sharing some characteristics with
SIR, there is the active subspaces (AS)
property\footnote{Some authors refer to the
active subspaces method, we prefer to employ the term active subspaces property,
as suggested by Constantine.}~\cite{russi2010uncertainty, constantine2015active,
constantine2017global, zahm2020gradient} which, in the last years, has
emerged as a powerful linear data-driven technique to construct ridge
approximations using gradients of the model function. AS has been
successfully applied to quantify uncertainty in the numerical
simulation of the HyShot II scramjet~\cite{constantine2015exploiting},
and for sensitivity analysis of an integrated hydrologic
model~\cite{jefferson2015active}.
Reduction in parameter space has been coupled with model order
reduction techniques~\cite{hesthaven2016certified,
quarteroni2014reduced, morhandbook2020} to enable more complex
numerical studies without increasing the computational load. We
mention the use of AS in cardiovascular applications with
POD-Galerkin~\cite{tezzele2018combined}, in nonlinear structural
analysis~\cite{guo2018reduced}, in nautical and naval
engineering~\cite{tezzele2018dimension, tezzele2018ecmi,
tezzele2019marine, tezzele2018model}, coupled with POD with interpolation for structural
and computational fluid dynamic (CFD) analysis~\cite{demo2019cras, tezzele2021multi}, and with Dynamic Mode
Decomposition in~\cite{tezzele2020enhancing}. Applications in
automotive engineering within a
multi-fidelity setting can be found in~\cite{romor2021multi}, for
turbomachinery see~\cite{seshadri2018turbomachinery}, while
for results in chemistry see~\cite{ji2019quantifying, vohra2019active}.
Advances in efficient global design optimization with surrogate
modeling are presented in~\cite{lukaczyk2014active,
lukaczyk2015surrogate} and applied to the shape design of the $N+2$
Supersonic Passenger Jet. Applications to enhance optimization methods
have been developed in~\cite{tripathy2016gaussian,
ghoreishi2019adaptive, demo2020asga, demo2021hull}. AS has also been successfully
used to reduced the memory consumption of highly parametrized systems
such as artificial neural networks~\cite{cui2020active, meneghetti2021dimensionality}.

Possible extensions and variants of the active subspaces property are
the Local Active Subspace method~\cite{romor2021las}, the Active
Manifold method~\cite{Bridges2019ActiveMA} which reduces the problem to the
analysis of a 1D manifold by traversing the level sets of the model
function at the expense of high online costs, the shared Active Subspace
method~\cite{ji2018shared}, the active subspaces property for
multivariate functions~\cite{zahm2020gradient}, and more recently an
extension of AS to dynamical systems~\cite{aguiar2018dynamic}. Another
method is Nonlinear Level set Learning
(NLL)~\cite{zhang2019learning} which exploits RevNets to reduce the
input parameter space with a nonlinear transformation.

The search for low dimensional structures is also investigated in machine
learning with manifold learning algorithms. In this context the Active
Subspaces methodology can be seen as a supervised dimension reduction
technique along with Kernel Principal Component Analysis
(KPCA)~\cite{sriperumbudur2017approximate} and Supervised Kernel
Principal Component Analysis
(SKPCA)~\cite{barshan2011supervised}. Other methods in the context of
kernel-based ROMs are~\cite{heas2020generalized, kevrekidis2015kernel,
mika1999kernel}.  In~\cite{palaci2018gaussian} a
non-linear extension of the active subspaces property based on
Random Fourier Features~\cite{rahimi2008random, li2019a} is
introduced and compared with machine learning manifold learning
algorithms for the construction of Gaussian process regressions
(GPR)~\cite{williams2006gaussian}.

From the preliminary work~\cite{palaci2018gaussian} in the context of supervised dimension
reduction algorithms in machine learning, we develop the kernel-based active subspaces
(KAS) method. The novelties of our contribution are the following:
\begin{itemize}
    \item regarding the AS theoretical background, we provide an upper bound of the ridge
    approximation error~\eqref{def:ridge_pb} for
    vector-valued objective functions and for a wide collection of probability
    distributions (see Assumption~\ref{ass:pdf}).
    \item we extend kernel-based AS to
    vector-valued model functions and develop a detailed algorithmic procedure
    for the optimization of the feature map. We also test different spectral
    measures (see Equation~\eqref{eq:Bochner's theorem} for the definition), differently from~\cite{palaci2018gaussian} where only the Gaussian
    measure is employed.
    \item the application to several test problems of increasing
    complexity. In particular, we mainly test KAS on problems where the active
    subspace is not present or the behaviour is not linear, differently
    from~\cite{palaci2018gaussian}, where the comparison is made with KPCA and
    its variants on datasets with linear trends in the reduced parameter space,
    apart from the hyperparaboloid test case that we have also included among
    our toy problems.
    \item the KAS method is finally applied to a computational fluid
    dynamics problem and compared with the standard AS technique. We study
    the evolution of fluid flow past a NACA~0012 airfoil in a duct
    composed by an initialization channel and a chamber. The motion is
    modelled with the unsteady incompressible Navier-Stokes
    equations, and discretized with the Discontinuous Galerkin
    method (DG)~\cite{hesthaven2007nodal}. Physical and geometrical parameters
    are introduced and sensitivity analysis of the lift and drag coefficients with respect to
    these parameters is provided.
\end{itemize}

The work is divided as follows: in \autoref{sec:as} we briefly
present the active subspaces property of a model function with a focus
on the construction of Gaussian process response surfaces.
Then, \autoref{sec:nas} illustrates the novel method called kernel-based
active subspaces for both scalar and vector-valued model
functions.
Several tests to compare AS and KAS are provided in \autoref{sec:test}
where we start from scalar functions with radial symmetry, we analyze an epidemiology model and a
vector-valued output generated from a stochastic elliptic PDE. A
parametric CFD test case for the study of the flow past a NACA airfoil
using the Discontinuous Galerkin method is presented in
\autoref{sec:results}. Finally, we outline some perspectives and future
studies in \autoref{sec:the_end}.

\section{Active Subspaces for parameter space reduction}
\label{sec:as}

Active Subspaces (AS) approach proposed in~\cite{russi2010uncertainty}
and developed in~\cite{constantine2015active} is a technique for
dimension reduction in parameter space. In brief AS are defined as the leading
eigenspaces of the second moment matrix of the model function's
gradient (for scalar model functions) and constitutes a global
sensitivity index~\cite{zahm2020gradient}. In the context of ridge
approximation, the choice of the active subspace corresponds to the
minimizer of an upper bound of the mean square error obtained through
Poincaré-type inequalities~\cite{zahm2020gradient}. After
performing dimension reduction in the parameter space through AS, the
method can be applied to reduce the computational costs of different
parameter studies such as inverse problems, optimization tasks and
numerical integration. In this work we
are going to focus on the construction of response surfaces with
Gaussian process regression.

\begin{definition}[Hypothesis on input and output spaces]
The quantities related to the input space are:
\label{def:hp}
\begin{itemize}
\item $m\in\mathbb{N}$ the dimension of the input space,
\item $(\Omega, \mathcal{F}, P)$ the probability space,
\item $\X:(\Omega, \mathcal{F},P)\rightarrow \mathbb{R}^m$, the absolutely continuous random vector representing the parameters,
\item $\rho:\mathbb{R}^{m}\rightarrow \mathbb{R}$, the probability density of $\mathbf{X}$ with support $\mathcal{X}\subset\mathbb{R}^{m}$.
\end{itemize}
The quantities related to the output are:
\begin{itemize}
\item $d\in\mathbb{N}$ the dimension of the output space,
\item $V=(\mathbb{R}^{d}, R_{V})$ the Euclidean space with metric
$R_{V}\in\mathcal{M}(d\times d)$\footnote{\RB{In this work with $\mathcal{M}(m \times n)$ we denote the set of real matrices with
$m$ rows and $n$ columns.}} and norm
$$\lVert \mathbf{x}\rVert^{2}_{R_{V}}=\mathbf{x}^{T}R_{V}\mathbf{x},$$
\item $f:\mathcal{X}\subset\mathbb{R}^{m}\rightarrow V$, the quantity/function
of interest, also called objective function in optimization tasks.
\end{itemize}
\end{definition}

Let $\mathcal{B} ( \mathbb{R}^m )$ be the Borel $\sigma$-algebra of
$\mathbb{R}^m$. We will consider the Hilbert space $L^2 (\mathbb{R}^m, \mathcal{B} (\mathbb{R}^m), \rho \,; \,V)$, of the measurable functions $f:(\mathbb{R}^m,
\mathcal{B}(\mathbb{R}^m),\rho)\rightarrow(\mathbb{R}^{d},R_V)$
such that
\begin{equation*}
\lVert f\rVert^{2}_{L^{2}}:=\int_{\mathcal{X}}\lVert f(\x)\rVert^2_{R_V}\,d\rho(\x)\leq \infty;
\end{equation*}
and the Sobolev space $H^{1}(\mathbb{R}^m, \mathcal{B} (\mathbb{R}^m), \rho \,; \,V)$ of measurable functions $f:(\mathbb{R}^m,
\mathcal{B}(\mathbb{R}^m),\rho)\rightarrow(\mathbb{R}^{d},R_V)$ such that
\begin{equation}
\lVert f\rVert^{2}_{H^{1}}:=\lVert f\rVert^{2}_{L^{2}}+\lVert\nabla f\rVert^{2}_{L^{2}}=\lVert f\rVert^{2}_{L^{2}}+|f|^{2}_{H^{1}}\leq \infty
\end{equation}
where $\nabla f$ is the weak derivative of $f$, and $\lVert\nabla
f\rVert_{L^{2}}=:|f|_{H^1}$.

We briefly recall how dimension reduction in parameter space is
achieved in the construction of response surfaces. The first step
involves the approximation of the model function with ridge
approximation. We will follow~\cite{zahm2020gradient, parente2020generalized}
for a review of the method.

The ridge approximation problem can be stated in the following way:
\begin{definition}[Ridge approximation]
\label{def:ridge_pb}
Let $\mathcal{B}(\mathbb{R}^m)$ be the
Borel $\sigma$-algebra of $\mathbb{R}^m$. Given
$r\in\mathbb{N},\,r\ll d$ and a tolerance $\epsilon\geq 0$, find the
profile $h:(\mathbb{R}^m, \mathcal{B}(\mathbb{R}^m),
\rho)\rightarrow V$ and the $r$-rank projection
$P_r:\mathbb{R}^m\rightarrow\mathbb{R}^m$ such that
\RA{
\begin{equation}
\label{eq:ridge_ineq}
\mathbb{E}_{P}[\lVert
f(\X)-h(P_r\X)\rVert^2_{R_V}]\leq \epsilon^2.
\end{equation}}
\end{definition}

In particular we are interested in the minimization problem
\RA{
\begin{equation}
\label{eq:min_pb}
  \argmin_{P_r \in \mathcal{M}(m\times m)} \, \mathbb{E}_{P} \left [ \lVert
f(\X)-\Tilde{h}(P_r\X) \rVert^2_{R_V} \right ],
\end{equation}}
where $\Tilde{h}=\mathbb{E}_{\rho}[ f |\sigma(P_r) ]$ is the conditional
expectation of $f$ under the distribution $\rho$ given the
$\sigma$-algebra $\sigma(P_r)$. The range of the projector $P_r$,
$\mathbb{R}^{r}\sim\text{Im}(P_{r})\subset\mathbb{R}^{m}$, is the reduced
parameter space. The kernel of the projector $P_r$,
$\mathbb{R}^{m-r}\sim\text{Im}(P_{r})\subset\mathbb{R}^{m}$, is the inactive subspace. The
existence of $\Tilde{h}$ is guaranteed by the Doob-Dynkin
lemma~\cite{bobrowski2005functional}. The function $\Tilde{h}$ is
proven to be the optimal profile for each fixed $P_r$, as a
consequence of the definition of the conditional expectation of a
random variable with respect to a $\sigma$-algebra.

Dimension reduction is effective if the
inequality~\eqref{eq:ridge_ineq} is satisfied for a specific
tolerance. The choice of $r$ is certainly of central
importance. The dimension of the reduced parameter space can be chosen a
priori for a specific parameter study (for example $r$-dimensional
regression), it can be chosen in order to satisfy the
inequality~\eqref{eq:ridge_ineq} or it is determined to guarantee a
good accuracy of the numerical method used to evaluate it [Corollary
3.10,~\cite{constantine2014active}].

Dividing the left term of the inequality \eqref{eq:ridge_ineq} with $\mathbb{E}_{\rho}[\lVert
f(\X)-\mathbb{E}_{\rho}[f(\X])\rVert^2_{R_V}]$ we obtain the Relative Root Mean
Square Error (RRMSE) and since it is a normalized quantity, we will use it to
make comparisons between different models
\RA{
\begin{equation}
\label{eq:RRMSE}
\text{RRMSE} = \sqrt{\frac{\mathbb{E}_{P}[\lVert
    f(\X)-h(P_r\X)\rVert^2_{R_V}]}{\mathbb{E}_{P}[\lVert
    f(\X)-\mathbb{E}_{P}[f(\X)]\rVert^2_{R_V}]}}.
\end{equation}}
We remark that $P_{r}$ is not unique.
It can be shown that if $\Tilde{h}$ is the optimal profile, then $P_r$ is not
uniquely defined and can be chosen arbitrarily from the set
$\{ Q_{r} : \mathbb{R}^m \rightarrow \mathbb{R}^m | \, \ker Q_{r}
 = \ker  P_{r}  \}$, see [Proposition 2.2, \cite{zahm2020gradient}].

The following lemma is the key ingredient in the proof of the existence of an
active subspace. It is inherently linked to probability Poincaré inequalities of
the kind
\RA{
\begin{equation}
\int_{\mathcal{X}}\lVert h(\mathbf{x})\rVert_{L^{2}}^2\,d\rho(\x)\leq\,C_{P}(\mathcal{X}, \rho)\,\int_{\mathcal{X}}\lVert \nabla h(\mathbf{x})\rVert_{L^{2}}^2\,d\rho(\x),
\end{equation}}
for zero-mean functions in the Sobolev space $h\in
H^{1}(\mathcal{X})$, where $C_{P}(\mathcal{X}, \rho)$ is the Poincaré
constant dependent on the domain $\mathcal{X}$ and on the probability
density functions (p.d.f.), $\rho$. We need to make the following
assumption to prove the next lemma and the next theorem.
\begin{definition}
\label{ass:pdf}
The probability density function $\rho:\mathcal{X}\rightarrow\mathbb{R}$ belongs to one of the following classes:
\begin{enumerate}
\item $\mathcal{X}$ is convex and bounded, $\exists \delta, D>0:\,0<\delta\leq\lVert\rho(\mathbf{x})\rVert_{L^{\infty}}\leq D<\infty\,\forall \x\in\mathcal{X}$,
\item $\rho(\x)\sim\exp(-V(\x))$ where $V:\mathbb{R}^{m}\rightarrow (-\infty,\infty]\,,V\in\mathcal{C}^{2}$ is $\alpha$-uniformly convex,
\begin{gather}
\mathbf{u}^{T}\text{Hess}(V(\x))\mathbf{u}\geq\,\alpha \lVert\mathbf{u}\rVert^{2}_{2},\quad\forall\x, \mathbf{u}\in\mathbb{R}^{m}
\end{gather}
where $\text{Hess}(V(\x))$ is the Hessian of $V(\x)$.
\item $\rho(\x)\sim\exp(-V(\x))$ where $V$ is a convex function. In this case we require also $f$ Lipschitz continuous.
\end{enumerate}
\end{definition}
In particular the uniform distribution belongs to the first class, the
multivariate Gaussian distribution $\mathcal{N}(m, \Sigma)$ to the
second with $\alpha=1/(\sigma_{max}(\Sigma))$ and the exponential and
Laplace distributions to the third. A complete analysis of the various
cases is done in~\cite{parente2020generalized}.
\begin{proposition}
\label{lemma:subspace_P_inequality}
Let $(\Omega, \mathcal{F}, P)$ be a probability space, $\X:(\Omega,
\mathcal{F},P)\rightarrow \mathbb{R}^m$ an absolutely continuous
random vector with probability density function $\rho$ belonging to
one of the classes from the Assumption~\ref{ass:pdf}. Then the
following inequality is satisfied
\begin{gather}
\mathbb{E}_{\rho}\left[\left(h-\mathbb{E}_{\rho}[h|\sigma (P_r)]\right)^{2}|\sigma(P_r)\right]\leq C_{P}(P_{r}, \rho)\,\mathbb{E}_{\rho}\left[\lVert(I-P_{r}^{T})\nabla h\rVert^{2}_{2}|\sigma (P_r)\right]
\end{gather}
for all scalar functions $h\in H^{1}(\mathcal{X})$ and for all $r$-rank
orthogonal projectors, $P_r$, where $C_{P}(P_{r}, \rho)$ is the Poincaré constant depending on $P_{r}$ and on the p.d.f. $\rho$.
\end{proposition}
A summary of the values of the Poincar\'e constant in relationship with the choice of the probability density function $\rho$ is reported in~\cite{parente2020generalized}.

In the next theorem the projection $P_{r}$ will depend on the output function $f$, so also the Poincar\'e constant $C_{P}(P_{r}, \rho)$ will depend in fact on $f$.

We introduce the following notation for the matrix that substitutes
the uncentered covariance matrix of the gradient $\nabla f$ in the case of the
application of AS to scalar model
functions~\cite{constantine2014active}
\begin{equation*}
H=\int_{\mathcal{X}} (D_{\x}f(\x))^T R_V(\rho) (D_{\x} f(\x)) \, d\rho (\x).
\end{equation*}
where $D_{\x} f(\x) \in \mathcal{M}(d\times m)$ is the
Jacobian matrix of $f$. The matrix $R_V(\rho)$ depends on the class which $\rho$ belongs to, see Appendix~\ref{sec:appendix}.

\begin{theorem}[Existence of an active subspace]
\label{theo:existence}
Under the hypothesis \ref{def:hp}, let $f\in H^1
(\mathbb{R}^m, \mathcal{B} (\mathbb{R}^m), \rho \,; \,V)$ and let the
p.d.f. $\rho$ satisfy Lemma~\ref{lemma:subspace_P_inequality} and
Assumption~\ref{ass:pdf}. Then the solution $\Tilde{P}_{r}$ of the ridge
approximation problem~\ref{def:ridge_pb} is the orthogonal
projector to the eigenspace of the first $r$-eigenvalues of $H$ ordered by magnitude
\begin{equation*}
Hv_i=\lambda_i v_i\qquad\forall i\in\{1,\dots,m\},\qquad
\Tilde{P}_{r}=\sum_{j=1}^{r}v_{j}\otimes v_{j},
\end{equation*}
with $r\in\mathbb{N}$ chosen such that
\RA{
\begin{equation}
\mathbb{E}_{\rho}\left[\lVert
f-\Tilde{h}\rVert^2_{R_V}\right]\leq\,C(C_{P}, \tau)\,\left(\sum_{i=r+1}^{m}\lambda_{i}\right)^{\frac{1}{1+\tau}}\leq \epsilon^2.
\end{equation}}
with $C(C_{P}, \tau)$ a constant depending on $\tau>0$ related to the choice of
$\rho$ and on the Poincar\'e constant from
lemma~\ref{lemma:subspace_P_inequality}, and $\Tilde{h}=\mathbb{E}_{\rho}[ f
|\sigma(P_r) ]$ is the conditional expectation of $f$ given the $\sigma$-algebra
generated by the random variable $P_r\circ\X$.
.

\end{theorem}
\begin{proof}
This theorem summarizes the results from Proposition 2.5 and Proposition 2.6
of~\cite{zahm2020gradient}, and from Lemma 3.1, Lemma 4.2, Lemma 4.3, Lemma 4.4
and Theorem 4.5 of~\cite{parente2020generalized}. The proof is expanded in
Appendix~\ref{sec:appendix}.
\end{proof}
The eigenspace $\text{span}\{v_1,\dots, v_r\}\subset\R^m$
is the active subspace and the remaining eigenvectors generate the
inactive subspace
$\text{span}\{v_{r+1},\dots,v_m\}\subset\R^m$. The
condition $f\in L^2 (\R^m, \mathcal{B} (\R^m), \rho \,; \,V)$ is
necessary for $f$ to satisfy the error bound~\eqref{eq:ridge_ineq}.

For the explicit procedure to compute the active subspace given its dimension
$r$ see Algorithm~\ref{algo:AS}: from $W_1$ and $W_2$ we define the approximations of
the projector $P_r$ with $\Hat{P}_r = W_1W_1^T$.

\begin{algorithm}
\caption{Active subspace computation.}
\label{algo:AS}
\begin{algorithmic}[1]
\Require gradients dataset $dY = (dy_1, \dots, dy_M)^T,\ dy_i\in\mathcal{M}(d\times
    m)$, that is $dY$ is a $3$-rank tensor
\Require symmetric positive definite metric matrix $R_V\in \mathcal{M}(d\times d)$
\Require active subspace dimension $r$

\State Compute the uncentered covariance matrix with Monte Carlo:
\begin{equation*}
\Tilde{H}=\frac{1}{M} \sum^M_{j=1} dY[j, :, :]^T R_V dY[j, :, :].
\end{equation*}
\State Solve the eigenvalue problem:
\begin{gather*}
\Tilde{H}\mathbf{v}_i = \lambda_i\mathbf{v}_i\qquad\forall i\in\{1,\dots,m\},\\
W_1 = (\mathbf{v}_1, \dots, \mathbf{v}_r),\quad W_2 = (\mathbf{v}_{r+1},\dots, \mathbf{v}_m) \Rightarrow \Hat{P}_r:=W_1W_1^T.
\end{gather*}
\State \Return active eigenvectors $W_1=(\mathbf{v}_1,\dots, \mathbf{v}_r)$ and
inactive eigenvectors $W_2=(\mathbf{v}_{r+1},\dots, \mathbf{v}_m)$ with $\mathbf{v}_i\in\R^m$,
and ordered eigenvalues $(\lambda_1,\dots,\lambda_m)$
\end{algorithmic}
\end{algorithm}

\subsection{Response surfaces}
\label{ssec:response_surfaces}
The term response surface refers to the general procedure of finding
the values of a model function $f$ for new inputs without directly
computing it but exploiting regression or interpolation from a
training set $\{\x_i,\ f(\x_i)\}$. The procedure for
constructing a Gaussian process response is reported in
Algorithm~\ref{algo:response_surface}, while in Algorithm~\ref{algo:gp_pred} we
show how to exploit it o predict the model function
at new input parameters.

Directly applying the simple Monte Carlo method with $N$ samples we get a
reduced approximation of $f$ as
\RA{
\begin{equation}
    (\Tilde{h}_{\epsilon}\circ P_r) (\X) = \mathbb{E}_{\rho}\left[ f |
      \sigma(P_r) \right] \approx \frac{1}{N} \sum^N_{i=1} f(\hat{P}_r
    \X + (I_d - \hat{P}_r ) \mathbf{Y}_i ) =: \hat{h}_{\epsilon, N}
    (\hat{P}_r \X),
\end{equation}}
where we have made explicit the dependence of the optimal profile
$\Tilde{h}_{\epsilon}$ on $\epsilon$,
$\mathbf{Y}_1,\dots,\mathbf{Y}_N$ are independent and
identically distributed samples of $\mathbf{Y} \sim \rho$, and
$\hat{P}_r$ is an approximation of $P_r$ obtained with the simple Monte Carlo
method from $H$, see Algorithm~\ref{algo:AS}. An intermediate approximation error is obtained
employing the Poincaré inequality and the central limit theorem for
the Monte Carlo approximation
\RA{
\begin{equation}
    \mathbb{E}_{P}\left[ (f(\X)-\hat{h}_{\epsilon, N} (\hat{P}_r\X))^2
    \right] \leq\ C_1 \left(1+N^{-1/2}\right)^2
    (\lambda_{n+1}+ \dots + \lambda_m ),
\end{equation}}
where $C_1$ is a constant, and $\lambda_{n+1}, \dots, \lambda_m$ are
the eigenvalues of the inactive subspace of $H$ [Theorem 4.4, \cite{constantine2014active}].

In practice $\hat{h}_{\epsilon, N}(\hat{P}_r\X)$ is approximated with a
regression or an interpolation such that a response surface
$\mathscr{R}$ satisfying
$\mathbb{E}_{\rho}\left[(\hat{h}_{\epsilon, N} (\hat{P}_r\x)-\mathscr{R}(\hat{P}_r\x)^2)
\right] \leq\ C_2 \delta$ is built, where $C_2$ is a constant, and
$\delta$ depends on the chosen method. An estimate for the successive
approximations
\begin{equation}
    f(\X) \approx \Tilde{h}_{\epsilon}(P_r\X)\approx
    \hat{h}_{\epsilon, N}(\hat{P}_r\X)\approx
    \mathscr{R}_{\epsilon,N,\delta}(\hat{P}_r\X),
\end{equation}
is given by
\RA{
\begin{align*}
    \mathbb{E}_{P}&\left[(f(\X)-\mathscr{R}(\hat{P}_r\X))^2\right] \\
    &\leq \ C_1(1+N^{-1/2})^2\left(
      \tau(\lambda_1+\dots+\lambda_n)^{1/2}+(\lambda_{n+1}+\dots+\lambda_m)^{1/2}\right)^2+C_2\lambda
\end{align*}}
where $\text{dist}(\text{Im}(P_r), \text{Im}(\hat{P}_r))
\leq\tau$, and $\lambda_i$ are the eigenvalues of $H$ [Theorem 4.8, \cite{constantine2014active}].

In our numerical simulations we will build the response surface
$\mathscr{R}$ with Gaussian process
regression (GPR)~\cite{williams2006gaussian}.

\begin{algorithm}
  \caption{Response surface construction with Gaussian process
    regression over the active subspace.}\label{algo:response_surface}

\begin{algorithmic}[1]
\Require normalized input dataset $X=(\x_1, \dots, \x_M)^T,\
  \x_i\in\R^m$
\Require output dataset $Y=(y_1, \dots, y_M)^T,\ y_i\in\R$
\Require active eigenvectors $W_1=(\mathbf{v}_1,\dots, \mathbf{v}_r),\
  \mathbf{v}_i\in\R^m$
  \Require kernel $k:\R^m\times\R^m \rightarrow \R$, \RC{with} hyper-parameters
  $\theta$ and the variance $\epsilon$ of the Gaussian noise

\State Project the inputs in the active subspace:
$XW_1 = \Tilde{X}\in\mathcal{M}(M\times r)$.
\State Evaluate the Gram matrix:
$K_{ij}(\theta)=k(\Tilde{\x}_i, \Tilde{\x}_j;\ \theta),\quad 1 \leq i,j \leq r$.
\State Tune the hyperparameters minimizing the negative log-likelihood:
\begin{gather*}
\bar{\theta} = \arg \min_{\theta} \,\, -\log{p(y|x, \theta)}\propto \frac{1}{2}\log{|K(\theta)+\sigma I_M|}+\frac{1}{2}Y^{T}(K(\theta)+\sigma I_M)^{-1}Y\; ,
\end{gather*}

\State \Return trained Gaussian process
\end{algorithmic}
\end{algorithm}
\begin{algorithm}
  \caption{Prediction phase using the Gaussian process response surface
  over the active subspace.}\label{algo:gp_pred}

\begin{algorithmic}[1]
\Require trained response surface $\y(\x)$
\Require active eigenvectors $W_1=(\mathbf{v}_1,\dots, \mathbf{v}_r),\
  \mathbf{v}_i\in\R^m$
\Require test samples $\overline{\x}\in \R^m$

\State Map the test samples $\overline{\x}$ onto the active subspace:
$\Tilde{\x} = W_1 \overline{\x}$.
\State Evaluate the Gaussian process on $\Tilde{\x}$ and return the
  prediction $t\sim\mathcal{N}(\mathbb{E}[t], \sigma^{2}(t))$:
\begin{equation*}
\mathbb{E}[t] =k(\Tilde{\x}, X)K^{-1}Y, \qquad
\sigma^2 (t) = k(\Tilde{\x}, \Tilde{\x}) - k(\Tilde{\x}, X) K^{-1}k(X,
\Tilde{\x}).
\end{equation*}
%
%
\end{algorithmic}
\end{algorithm}

\section{Kernel-based Active Subspaces extension}
\label{sec:nas}
\begin{figure}[ht!]
\centering
\includegraphics[width=.8\textwidth]{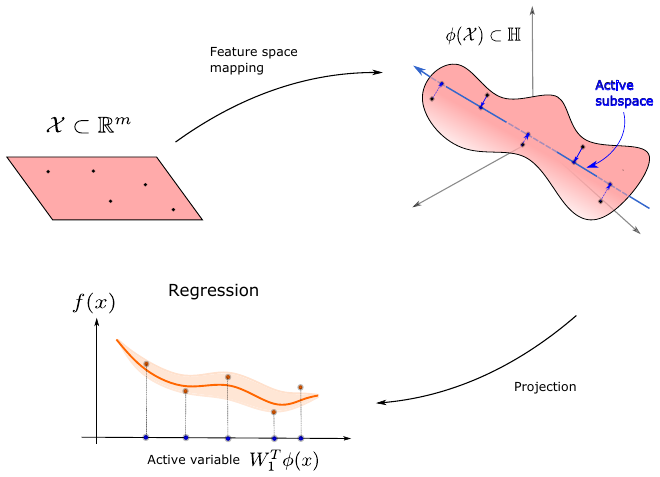}
\caption{Illustration of the construction of a one-dimensional response surface with
  kernel-based active subspaces and Gaussian process regression.}
\label{fig:math_graph}
\end{figure}

Keeping the notations of section~\ref{def:hp}, $\X:(\Omega,
\mathcal{F}, P) \to \R^m$ is the absolutely continuous
random vector representing the $m$-dimensional inputs with density
$\rho:\mathcal{X} \subset \R^m \to \R$, and
$f:\mathcal{X} \subset \R^m \to (V, R_V)$ is the
model function that we assume to be continuously differentiable and
Lipschitz continuous.

One drawback of sufficient dimension reduction with AS applied to
ridge approximation is that if a clear linear trend is missing,
projecting the inputs as $P_r \X$ represents a loss of accuracy on the
approximation of the model $f$ that may not be compensated even by
the choice of the optimal profile
$\Tilde{h}\circ P_{r}=\mathbb{E}_{\rho} [f | \sigma(P_r)]$.
In order to overcome this,
non-linear dimension reduction to one-dimensional parameter space
could be achieved discovering a curve in the space of parameters that
cuts transversely the level sets of $f$, this variation is presented
in~\cite{Bridges2019ActiveMA} as Active Manifold. Another approach could
consist in finding a diffeomorphism $\phi$ that reshapes the level
sets such that subsequently applying AS dimension reduction to the new
model function~$\Tilde{f}\circ\phi=f$
could be more profitable:
\[
  \begin{tikzcd}
    \mathcal{X}\subset\mathbb{R}^m \arrow{r}{\phi}
    \arrow[swap]{dr}{f} & \phi(\mathcal{X})\subset\mathbb{R}^m
    \arrow{d}{\Tilde{f}} \\
     & V
  \end{tikzcd}
\]

Unfortunately constructing the Active Manifold or finding the right
diffeomorphism $\phi$ could be a complicated matter. If we renounce to
have a backward map and we weaken the bond of the method with the
model, we can consider an immersion $\phi$ from the space of
parameters $\mathcal{X}$ to an infinite-dimensional Hilbert space
$\mathbb{H}$ obtaining
\[
  \begin{tikzcd}
    \mathcal{X}\subset\mathbb{R}^m \arrow{r}{\phi}
    \arrow[swap]{dr}{f} & \phi(\mathcal{X})\subset\mathbb{H}
    \arrow{d}{\Tilde{f}} \\
     & V
  \end{tikzcd}
\]
This is a common procedure in machine learning in order to increase
the number of features~\cite{williams2006gaussian}. Then AS is applied
to the new model function
$\Tilde{f}:\phi(\mathcal{X})\subset\mathbb{H}\rightarrow V$ with
parameter space $\phi(\mathcal{X})\subset\mathbb{H}$. A response surface can be built with \autoref{algo:response_surface} remembering to replace every occurrence of the inputs $\x$ with their images $\phi(\x)$. A synthetic scheme of the procedure is represented in \autoref{fig:math_graph}.


In practice we consider a discretization of the infinite-dimensional
Hilbert space $\R^D \simeq \mathbb{H}$ with $D>m$. Dimension reduction with AS results in the choice of a $r$-rank
projection in the much broader set of $r$-rank projections in
$\mathbb{H}$.

Since for AS only the samples of the
Jacobian matrix of the model function are employed, we can ignore the
definition of the new map
$\Tilde{f}: \phi(\mathcal{X})\subset\mathbb{H}\rightarrow (V, R_V)$
and focus only on the computation of the Jacobian matrix of
$\Tilde{f}$ with respect to the new input variable
$\z:=\phi(\x)$. The uncentered covariance
matrix becomes
\RA{
\begin{align*}
H &= \int_{\phi(\mathcal{X})} \left [
                      (D_{\z}\Tilde{f})^T(\z) \right
                      ] R_V \left [
                      (D_{\z}\Tilde{f})(\z) \right ] \,
                      d\mu(\z)\\
&= \int_{\mathcal{X}} \left [
                                            (D_{\z}\Tilde{f})^T(\phi(\x))
                                            \right ] R_V \left [
                                            (D_{\z}\Tilde{f})(\phi(\x))
                                            \right ] \,d\mathcal{L}_{\X}(\x),
\end{align*}}
where $\mu:=\phi_{\#}(\mathcal{L}_{\X})$ is the pushforward
probability measure of $\mathcal{L}_{\X}$ (the law of
probability of $\X$) with respect to the map $\phi$. Simple
Monte Carlo can be applied sampling from the distribution $\rho$ in
the input space $\mathcal{X}$
\RA{
\begin{align*}
H &= \int_{\mathcal{X}} \left [
                      (D_{\z}\Tilde{f})^T(\phi(\x))
                      \right ] R_V \left [
                      (D_{\z}\Tilde{f})(\phi(\x))
                      \right ] \,d\mathcal{L}_{\X}(\x)\\
&\approx\frac{1}{M}\sum_{i=1}^M \left [(D_{\z}\Tilde{f})^T(\phi(\x_i))
  \right ]  R_V \left [ (D_{\z}\Tilde{f})(\phi(\x_i)) \right ] .
\end{align*}}
The gradients of $\Tilde{f}$ with respect to the new input variable
$\Z$ are computed from the known values $D_{\x}f$ with
the chain rule.

The application of the chain rule to the composition of
functions~$\Tilde{f}\circ\phi:\mathbb{R}^m\rightarrow\mathbb{H}\rightarrow
V$ is applicable if $\Tilde{f}$ is defined in an open set
$U\supset\phi(\mathcal{X})$.
If $\phi$ is non singular and also injective the new input space is a
$m$-dimensional submanifold of $\mathbb{H}$. If $\phi$ is also smooth
there exists a smooth extension of
$\Tilde{f}:\phi(\mathcal{X})\subset\mathbb{H}\rightarrow V$ onto the whole
domain $\mathbb{H}$,
see Proposition 1.36 from~\cite{warner1983foundations}.

If the Hilbert space $\mathbb{H}$ has finite dimension $\mathbb{H}\sim\mathbb{R}^{D}$ this procedure leaves us with an underdetermined linear system to solve
for~$D_{\z}\Tilde{f}$
\begin{align}
\label{eq:pseudo_jac}
&D_{\z}\Tilde{f}(\phi(\x))D\phi(\x)=D_{\x}f(\x),
  \nonumber\\
&D_{\z}\Tilde{f}(\phi(\x))=D_{\x}f(\x)(D\phi(\x))^{\dagger},
\end{align}
where $^{\dagger}$ stands for the right Moore-Penrose inverse of the matrix $D\phi(\x)$ with rank $r$, that is
\begin{equation*}
(D\phi(\x))^{\dagger}=V \Sigma^{\dagger}U^{T},
\end{equation*}
with the usual notation for the singular value decomposition (SVD) of $D\phi(\x)$
\begin{equation}
D\phi(\x) = U \Sigma V^{T},
\end{equation}
and $\Sigma^{\dagger}\in\mathcal{M}(r\times r)$ equal to the diagonal matrix
with the inverse of the singular values as diagonal elements. \RC{As anticipated if
$f$ is smooth enough and $\phi$ is an embedding, so that $D\phi$ has full rank, the previous system has an
unique solution. The most crucial part is the evaluation of the gradients
$D_{\x}f(\x)$ from the input output couples, when they are not available
analytically or from the adjoint method applied to PDEs models:
different approaches are present in the literature, like local polynomial regressions and
Gaussian process regression on the whole domain to approximate the gradients;
both are available in the ATHENA package~\cite{romor2020athena}. For an estimate
of the ridge approximation error due to inaccurate gradients see~\cite{constantine2015active}.}

Finally, we remark that in the AS method we approximate the random variable $\X$ as
\begin{equation}
P_r \X = \mathbf{v}_1(\mathbf{v}_1\cdot \X)+\dots+\mathbf{v}_r
(\mathbf{v}_r\cdot \X),
\end{equation}
with $\{\mathbf{v}_i\} \subset \R^m$ the active
eigenvectors, whereas with KAS the reduced input space is contained in
$\mathcal{H}$
\begin{equation}
P_r \X = \mathbf{v}_1 (\mathbf{v}_1 \cdot \phi(\X))+\dots+\mathbf{v}_r
(\mathbf{v}_r \cdot \phi(\X)),
\end{equation}
with $\{\mathbf{v}_i\} \subset \mathcal{H}$ the active eigenvectors of
KAS. In this case the model is enriched by the non-linear feature map
$\phi$.

\begin{algorithm}
  \caption{Kernel-based active subspace computation.}\label{algo:KAS}

  \begin{algorithmic}[1]
\Require gradients dataset $dY=(dy_1, \dots, dy_M)^T,\
  dy_i\in\mathcal{M}(d\times m)$
\Require symmetric positive definite metric matrix $R_V \in \mathcal{M} (d\times d)$
\Require feature subspace dimension $D$
\Require feature map $\phi:\R^m \rightarrow \R^D$
\Require active subspace dimension $r$

\State Evaluate gradients solving an overdetermined linear system:
\begin{equation*}
\forall j\in\{1,\dots,M\},\quad dY[j,:, :] (D\phi)^{\dagger} =
d\Tilde{Y}[j, :, :] \in\mathcal{M}(d, D).
\end{equation*}
\State Compute the uncentered covariance matrix with Monte Carlo:
\begin{equation*}
\Tilde{H}=\frac{1}{M}\sum^M_{k=1} d\Tilde{Y}[k, :, :]^T R_V d\Tilde{Y}[k, :, :].
\end{equation*}
\State Solve the eigenvalue problem:
$\Tilde{H}\mathbf{v}_i=\lambda_i\mathbf{v}_i,\quad \forall i\in\{1,\dots,D\}$.

\State \Return active eigenvectors $W_1 = (\mathbf{v}_1,\dots,
  \mathbf{v}_r)$ and inactive eigenvectors $W_2 = (\mathbf{v}_{r+1},\dots,
  \mathbf{v}_D)$ with $\mathbf{v}_i \in \R^D$, and ordered eigenvalues $(\lambda_1,\dots,\lambda_D)$
\end{algorithmic}
\end{algorithm}

\subsection{Choice of the Feature Map}
The choice for the map $\phi$ is linked to the theory of Reproducing
Kernel Hilbert Spaces (RKHS)~\cite{berlinet2011reproducing}, and it is
defined as
\begin{gather}
\label{eq:feature_map}
\z=\phi(\x)=\sqrt{\frac{2}{D}}\,\sigma_f\,\cos(W\x+\mathbf{b}),\\
\cos(W\x+\mathbf{b}):=\frac{1}{\sqrt{D}}(\cos(W[1, :] \cdot \x +b_1),\dots, \cos(W[D, :] \cdot \x +b_D))^T
\end{gather}
where $\sigma_f$ is an hyperparameter corresponding to the empirical
variance of the model, \RB{$W\in \mathcal{M}(D\times m)$} is the
projection matrix whose rows are sampled from a probability
distribution $\mu$ on $\R^m$ and $\mathbf{b}\in \R^D$ is a bias term whose components are
sampled independently and uniformly in the interval $[0, 2\pi]$. We
remark that its Jacobian can be computed analytically as
\begin{equation}
\label{eq:feature_map_grad}
\frac{\partial z^j}{\partial
  x^i}=-\sqrt{\frac{2}{D}}\,\sigma_f\,\sin\left(\sum_{k=1}^D W_{ik}\x_k+\mathbf{b}_k\right)\,W_{ij},
\end{equation}
for all $i \in \{1, \dots, m\}$, and for all $ j \in \{1, \dots, D\}$.

We remark that in order to guarantee the correctness of the procedure for evaluating
the gradients we have to prove that the feature map is injective and
non singular. In general however the feature map~\eqref{eq:feature_map} cannot
not be injective due to the periodicity of the cosine but at least it is almost
surely non singular if the dimension of the feature space is high enough.

The feature map~\eqref{eq:feature_map} is not the only effective
immersion that provides a kernel-based extension of the active
subspaces. For example an alternative is the following composition of
a linear map with a sigmoid
\begin{equation*}
\phi(\z)=\frac{C}{1+\alpha \, e^{ -W\z }},
\end{equation*}
where $C$ is a constant, $\alpha$ is an hyperparameter to be tuned, and
$W \in \mathcal{M} (D, m)$ is, as before, a matrix whose rows are
sampled from a probability distribution on $\R^m$.

Other choices involve the use of Deep Neural Networks to learn the
profile $h$ and the projection function $P_r$ of the ridge approximation
problem~\cite{tripathy2019deep}.

The tuning of the hyperparameters of the spectral measure consists
in a global optimization problem where the dimension of the domain can
vary between $1$ and the dimension of the input space $m$. The object
function to optimize is the relative root mean square error (RRMSE)
\begin{equation}
\label{eq:RRMSE_discrete}
\text{RRMSE}(Y_{\text{test}},
T_{\text{test}}) = \sqrt{\frac{\sum^{N}_{i=1}(t_{i}-y_{i})^{2}}{\sum^{N}_{i=1}(t_{i}-\Bar{y})^{2}}},
\end{equation}
where $T_{\text{test}}=(t_{i})_{i\in \{1, \dots, N\}}$ are the
predictions obtained from the response surface built
with KAS and associated
to the test set, $Y_{\text{test}}=(y_{i})_{i\in \{1, \dots, N\}}$ are the targets
associated to the test set, and $\Bar{y}$ is the mean value of the
targets. We implemented a logarithmic grid-search\RC{, see Algorithm~\ref{algo: tuning},} making use of
the SciPy library~\cite{2020SciPy-NMeth}. Another choice could be
Bayesian stochastic optimization implemented in the open-source library GPyOpt~\cite{gpyopt2016}.

\RB{The tuning of the hyperparameters of the spectral measure chosen is
the most computationally expensive part of the procedure. We report the
computational complexity of the  algorithms introduced to have a better understanding of
the additional cost implied by the implementation of response
surface design with KAS. Let us assume that the number of random Fourier features $D$, the
number of input, output, and gradient samples $M$, and the dimension of the
parameter space $m$, are ordered in this manner $D>M>m$, as is usually the case,
and that the quantity of interest $f$
is a scalar function. The cost of computing an active subspace is $O(Mm^2)$, that is the cost of the
SVD of the gradients matrix
$dY$ used to get the active and inactive eigenvectors in
Algorithm~\ref{algo:AS}. The cost of the training of a response surface with
Gaussian process regression in Algorithm~\ref{algo:response_surface} depends on
the cost of minimization of the log-likelihood: each evaluation of the
log-likelihood involves the
computation of the determinant and the inverse of the regularized Gram matrix $K(\theta)+\sigma
I_M$, that is $O(M^3)$. Finally, the cost for the evaluation of the
kernel-based active subspace is associated to the SVD of $d\Tilde{Y}$ that is $O(DM^2)$ in
Algorithm~\ref{algo:KAS}, and to the resolution of the overdetermined linear system to
obtain the gradients $d\Tilde{Y}$, that is $M$ times $O(Dm^2)$ since it is related to the evaluation of the pseudo-inverse of $D\phi$. So, the
computational complexity for the response surface design with AS and GPR is
$O(n_{\text{GPR}}M^3)$, while for the response surface design with KAS and GPR is
$O\left(n_{\text{grid-search}}
n\left(D\frac{M^2}{n^2}+\frac{M}{n}Dm^2+n_{\text{GPR}}\frac{M^3}{n^3}\right)\right)$, where
$n_{\text{GPR}}$ is the maximum number of steps of the
optimization algorithm used to minimize the log-likelihood,
$n_{\text{grid-search}}$ is the number of hyperparameter instances $\gamma\in
G$ to try in Algorithm~\ref{algo: tuning}, and $n$ is the number of batches in the
$n$-fold cross validation procedure. In particular, for each grid search
hyperparameter the main cost is associated to the GPR training since
$n_{\text{GPR}}$ usually satisfy $Dn<n_{\text{GPR}}M$, when the optimizer chosen
is L-BFGS-B from SciPy~\cite{2020SciPy-NMeth}, accounting also for the number of restarts of the
optimizer: in the numerical tests we performed the number of restarts of the training of
the GPR is problem-dependent but always less than $10$. In
general, the number $n_{\text{grid-search}}$ depends on the chosen application,
and the multiplicative factor between the computational complexity of the
response surface design procedure with KAS or AS is lower than $3n_{\text{grid-search}}n$.}

\begin{algorithm}
\caption{Tuning the feature map with logarithmic grid-search.}\label{algo: tuning}

  \begin{algorithmic}[1]
\Require normalized input dataset $X=(\mathbf{x}_{1}, \dots,
  \x_{M})^{T},\ \x_{i}\in\mathbb{R}^{m}$
\Require output dataset  $Y=(y_1,  \dots, y_M)^T,\ y_i\in\R$
\Require gradients dataset $dY=(dy_{1}, \dots, dy_{M})^{T},\
  dy_{i}\in\mathcal{M}(d\times D)$
\Require spd metric matrix
  $R_{V}\in \mathcal{M}(d\times d)$
\Require feature subspace dimension $D$
\Require feature map $\phi:\mathbb{R}^{m}\rightarrow\mathbb{R}^{D}$
\Require spectral density with hyperparameter $\alpha$, $\mu=\mu(\alpha)$
\Require active subspace dimension $r$
\Require tolerance for the tuning procedure $\text{tol}\approx 0.8$.

\State Create the grid $G$ and set the variable BEST to $1$.
\For{$\gamma\in G$}
\State Compute the feature map projection matrix $W$ associated to $\gamma$: $W[i, :] \text{ sampled from } \mu(\gamma),\quad \forall
i\in\{1,\dots, D\}$.
\State Compute the uniformly sampled bias $b$: $b[i] \sim \mathcal{U}(0, 2\pi)$.
\State Compute score with $n$-fold cross validation:
\For{$i= 1$ to $n$}
\State Divide input, output, and gradients in train and test datasets.
\State Compute $(W_1, W_2, (\lambda_1,\dots, \lambda_D))$ with KAS method in
\autoref{algo:KAS} with inputs ($dY_{\text{train}}, R_V, D, \phi,
r)$.
\State Build GPR response surface with inputs ($X_{\text{train}},
  Y_{\text{train}}, W_1, k$) using \autoref{algo:response_surface}.
\State Predict the values $T_{\text{test}}$ using \autoref{algo:gp_pred} with input $X_{\text{test}}$.
\State Evaluate the score as $\text{score}[n]=\text{RRMSE}(Y_{\text{test}},
T_{\text{test}})$.

\If{$\text{score}[n] > \text{tol}$}
\State Stop cross validation and pass to the next value of $\alpha$.
\EndIf

\EndFor

\If{$\text{mean(score)} <$ BEST}
\State Save $W$ and $b$, and set BEST to $\text{mean(score)}$.
\EndIf

\EndFor

\State \Return projection matrix $W$, and bias $b$
\end{algorithmic}
\end{algorithm}

\subsection{Random Fourier Features}
The motivation behind the choice for this map from Equation~\eqref{eq:feature_map}
comes from the theory on Reproducing Kernel Hilbert Spaces. The infinite-dimensional Hilbert
space $(\mathbb{H}, \langle\cdot,\cdot\rangle)$ is assumed to be a
RKHS with real shift-invariant kernel
$k:\mathcal{X}\times\mathcal{X} \to \mathbb{R}$ with $k(0)=1$
and feature map $\phi$.

In order to get a discrete approximation of
$\phi:\mathcal{X} \subset \mathbb{R}^m \to \mathbb{H}$, random
Fourier features are employed~\cite{rahimi2008random,
  li2019a}. Bochner's theorem \cite{mohri2018foundations} guarantees the existence of a
spectral probability measure $\mu$ such that
\begin{equation*}
  \label{eq:Bochner's theorem}
k(\mathbf{x}, \mathbf{y}) = \int_{\mathbb{R}^m} e^{i\boldsymbol{\omega}\cdot(\mathbf{x}-\mathbf{y})}\,d\mu(\boldsymbol{\omega}).
\end{equation*}
From this identity we can get a discrete approximation of the scalar
product $\langle\cdot,\cdot\rangle$ with Monte Carlo method,
exploiting the fact that the kernel is real
\begin{gather}
\label{eq:rff_approx}
\langle\phi(\x),\phi(\y) \rangle=k(\x,
\y)\approx
\frac{1}{D}\sum_{i=1}^D\cos(\boldsymbol{\omega}_i\cdot\x + b_i)
\cos(\boldsymbol{\omega}_i \cdot \y +b_i)=\z^T\z, \\
\z=\frac{1}{\sqrt{D}}(\cos(\boldsymbol{\omega}_1 \cdot \x +b_1),\dots, \cos(\boldsymbol{\omega}_D \cdot \x +b_D)),
\end{gather}
and from this relation we obtain the approximation
$\phi\approx\mathbf{z}$. The sampled vectors
$\{\boldsymbol{\omega}_i\}_{i=1,\dots,D}$ are called random Fourier
features. The scalars $\{b_i\}_{i=1,\dots,D}$ are bias terms
introduced since in the approximation we have excluded some trigonometric terms from the following initial expression
\begin{equation*}
\frac{1}{D}\sum_{i=1}^D\left( \cos(\boldsymbol{\omega}_i \cdot \mathbf{x})
\cos(\boldsymbol{\omega}_i \cdot \mathbf{y})-\sin(\boldsymbol{\omega}_i \cdot \mathbf{x})
\sin(\boldsymbol{\omega}_i \cdot \mathbf{y})\right).
\end{equation*}

Random Fourier features are frequently used to approximate
kernels. We consider only spectral probability measures which have a
probability density, usually named spectral density. In the approximation of the
kernel with random Fourier features, under some
regularity conditions on the kernel, an explicit
probabilistic bound depending on the dimension of the feature space
$D$ can be proved~\cite{mohri2018foundations}.  This technique is used to scale up Kernel Principal
Component Analysis~\cite{scholkopf2002learning, scholkopf1998nonlinear} and Supervised Kernel Principal Component
Analysis~\cite{barshan2011supervised}, but in the case of kernel-based AS
the resulting overdetermined linear system employed to compute the
Jacobian matrix of the new model function increases in dimension
instead.

The most famous kernel is the squared exponential kernel also called Radial Basis Function kernel (RBF)
    \begin{equation}
        k_{\text{RBF}}(\x, \mathbf{y})=\exp\left( -\frac{\lVert\x-\mathbf{y}\rVert^2}{2l^2}\right),
    \end{equation}
    where $l$ is the characteristic length-scale. The spectral density
    is Gaussian $\mathcal{N}(0,\ 1/4\pi^2l^2)$:
    \begin{equation}
        S(\boldsymbol{\omega})=(2\pi l^2)^{D/2}\text{exp}(-2\pi^2l^2\boldsymbol{\omega}^2).
    \end{equation}

Thanks to Bochner's theorem to every probability distribution that
admits a probability density function corresponds a stationary
positive definite kernel. So having in mind the
definition of the feature map $\phi$ from
Equation~\eqref{eq:feature_map}, we can choose any probability
distribution for sampling the
random projection matrix \RB{$W\in\mathcal{M}(D\times m)$}
without focusing on the corresponding kernel since it is not needed by
the numerical procedure.

After the choice of the spectral measure the corresponding
hyperparameters have to be tuned. This is linked to the choice of the
hypothesis models in machine learning and it is usually carried out for
the hyperparameters of the employed kernel. From the choice of the
kernel and the corresponding hyperparameters some regularity
properties of the model are implicitly
assumed~\cite{williams2006gaussian}.

\section{Benchmark test problems}
\label{sec:test}
In this section we are going to present some benchmarks to prove
the potential gain of KAS over standard linear AS, for both scalar and
vectorial model functions. In particular we test KAS on radial
symmetric functions, with $2$-dimensional and $8$-dimensional
parameter spaces, on the approximation of the reproduction number
$R_0$ of the SEIR model, and finally on a vectorial output function
that is the solution of a Poisson problem.

One dimensional response surfaces are built following the algorithm described in
\autoref{ssec:response_surfaces}. The tuning of the hyperparameters
of the feature map is carried out with a logarithmic grid-search and
$5$-fold cross validation for the Ebola test case, while for the other
cases we employed Bayesian stochastic optimization implemented
in~\cite{gpy2014} with $3$-fold cross validation. The score function
chosen is the Relative Root Mean Square Error (RRMSE). \RC{The spectral measure
for each test case is chosen by brute force among the Laplace, Gaussian, Beta
and multivariate Gaussian distributions. The number of Fourier features is not
established based on a criterion but we have seen experimentally that above a certain
threshold the number of features is high enough to at least reproduce the accuracy of
the AS method. Since the most sensitive part to the final accuracy of the
response surface is the tuning
of the hyperparameters of the spectral measures, we suggest to choose an
affordable number of features between 1000 and 2000, and focus on the tuning of
said hyperparameters instead.}\RB{We remark that the number of samples employed
is problem dependent: some heuristics to determine it can be found
in~\cite{constantine2015active}, but the crucial point is that additional
training samples with respect to the ones used for the AS method are not needed for the
novel KAS method.} \RC{Moreover, the CPU time for the hyperparameters tuning
procedure is usually negligible with respect to the time required to obtain
input-output pairs from the numerical simulation of PDEs models: in our
applications the tuning procedure's computational time is in the order of
minutes (usually around 10-15 minutes for most testcases), while for
the CFD application of Section~\ref{sec:results} it is in the order of days and
for the stochastic elliptic partial differential equation of
Subsection~\ref{ssec:spde} it is in the order of hours. We also remark that the
tuning Algorithm~\ref{algo: tuning}, the GPR training restarts, and the choice
of the spectral measure can be easily
parallelized.}

For the radial
symmetric and Ebola test cases the inputs are sampled from a uniform
distribution with problem dependent ranges.For the stochastic
elliptic partial differential case the inputs are the coefficients of
a Karhunen-Lo\`eve expansion and are sampled from a normal
distribution. All the computations regarding AS and KAS are done using
the open source Python package called ATHENA~\cite{romor2020athena}.

\subsection{Radial symmetric functions}
Radial symmetric functions represent a class of model functions for
which AS is not able to unveil any low dimensional behaviour. In fact
for these functions any rotation of the parameter space produce the
same model representation. Instead Kernel-based AS is able to overcome
this problem thanks to the mapping onto the feature space.

We present two benchmarks: an $8$-dimensional hyperparaboloid defined as
\begin{equation}
\label{eq:paraboloid}
f:[-1, 1]^8 \subset \R^8 \rightarrow \R , \qquad \qquad
f(\x) = \frac{1}{2} \lVert \x \rVert^2 ,
\end{equation}
and the surface of revolution in $\R^3$ with generatrix $g(x)=
\sin(x^2)$
\begin{equation}
\label{eq:sine}
f:[-3, 3]^2 \subset \R^2 \rightarrow \R , \qquad \qquad
f(\x) = g(\lVert \x \rVert)=\sin(\lVert \x \rVert^2) .
\end{equation}
The gradients are computed analytically.

For the hyperparaboloid we use $N_s = 500$ independent, uniformly
distributed training samples in $[-1, 1]^8$, while for the sine case the
training samples are $N_s = 800$ in $[-3, 3]^2$. In both cases the test samples are $500$. The feature space has dimension
$1000$ for both the first and the second case. The spectral
distribution chosen is the multivariate normal with
hyperparameter a uniform variance $\lambda I_{d}$, and a product of
Laplace distributions with $\gamma$ and $b$ as
hyperparameters, respectively. The tuning is carried out with $3$-fold
cross validation. The results are summarized in \autoref{tab:res_tests}.

\begin{table}[htp!]
\centering
\caption{Performance results for AS and KAS methods. For each case we
  report the parameter space dimension, the number of samples $N_s$
  used for the training, the chosen distribution, the dimension of
  the feature space, and the RRMSE mean and standard deviation for AS
  and KAS. In bold the best results.\label{tab:res_tests}}
\begin{tabular}{ l c c c c c c }
\hline
\hline
\multirow{2}{*}{Case} & \multirow{2}{*}{Dim} & \multirow{2}{*}{$N_s$}
  & Spectral & Feature & \multirow{2}{*}{RRMSE AS} & \multirow{2}{*}{RRMSE KAS} \\
  &  &  & distribution & space dim &  &  \\
\hline
\hline
\rowcolor{Gray}
Hyperparaboloid  & 8 & 500 & $\mathcal{N}(\mathbf{0}, \lambda I_{d})$ & 1000 & 0.98 $\pm$ 0.03 & \textbf{0.23} $\pm$ 0.02 \\
Sine & 2 & 800 & $\text{Laplace}(\gamma, b)$ & 1000 & 1.011 $\pm$ 0.01 & \textbf{0.31} $\pm$ 0.06 \\
\rowcolor{Gray}
Ebola & 8 & 800 & $\text{Beta}(\alpha, \beta)$ & 1000 & 0.46 $\pm$ 0.31 & \textbf{0.31} $\pm$ 0.03 \\
SPDE~\eqref{eq:spde_1}  & 10 & 1000 & $\mathcal{N}(\mathbf{0}, \Sigma)$ & 1500 & 0.611 $\pm$ 0.001 & \textbf{0.515} $\pm$ 0.013 \\
\hline
\hline
\end{tabular}
\end{table}

Looking at the eigenvalues of the uncentered covariance matrix of the
gradients $\Tilde{H}$ for the hyperparaboloid case
in \autoref{fig:evals_hyperparaboloid},
we can clearly see how the decay for AS is almost absent, while using
KAS the decay after the first eigenvalue is pronounced, suggesting the
presence of a kernel-based active subspace of dimension $1$.

\begin{figure}[ht!]
\centering
\includegraphics[width=.49\textwidth]{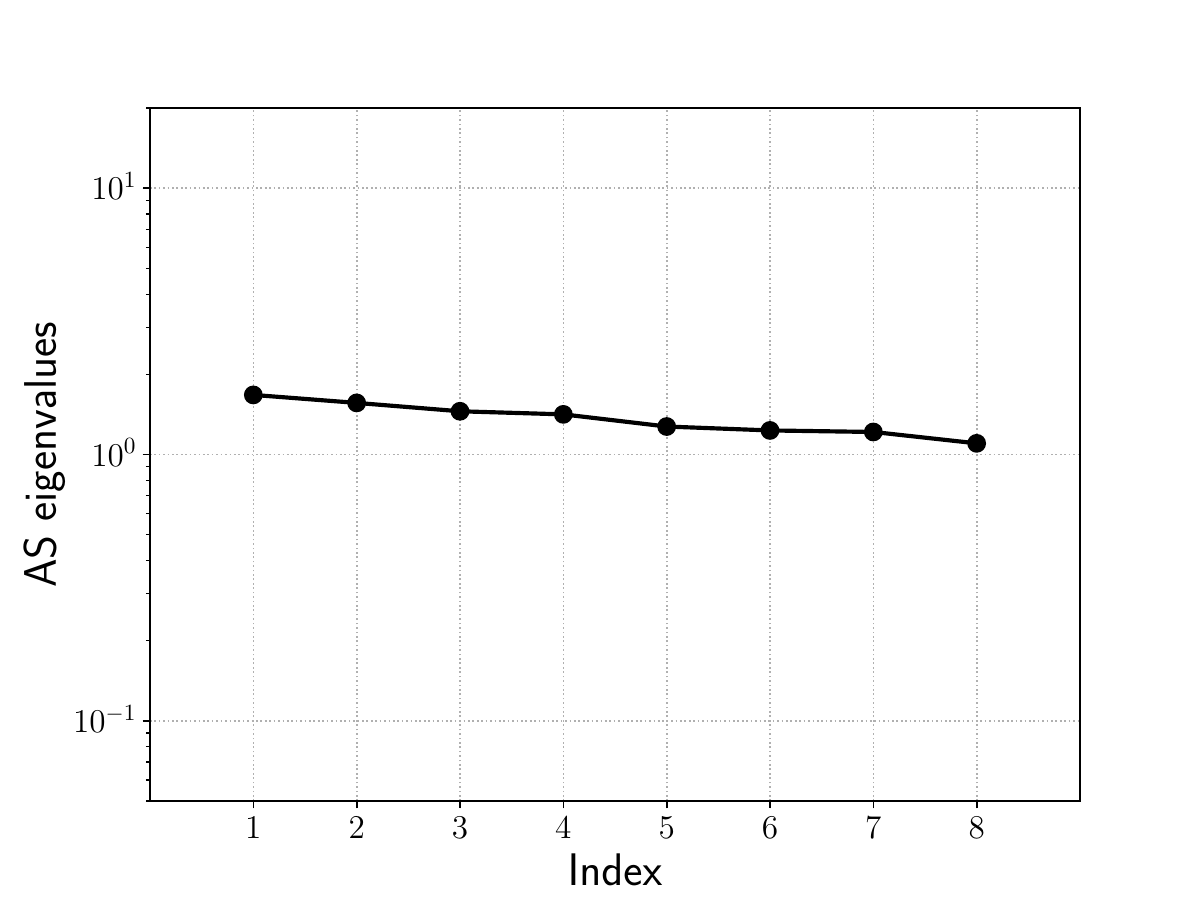}\hfill
\includegraphics[width=.49\textwidth]{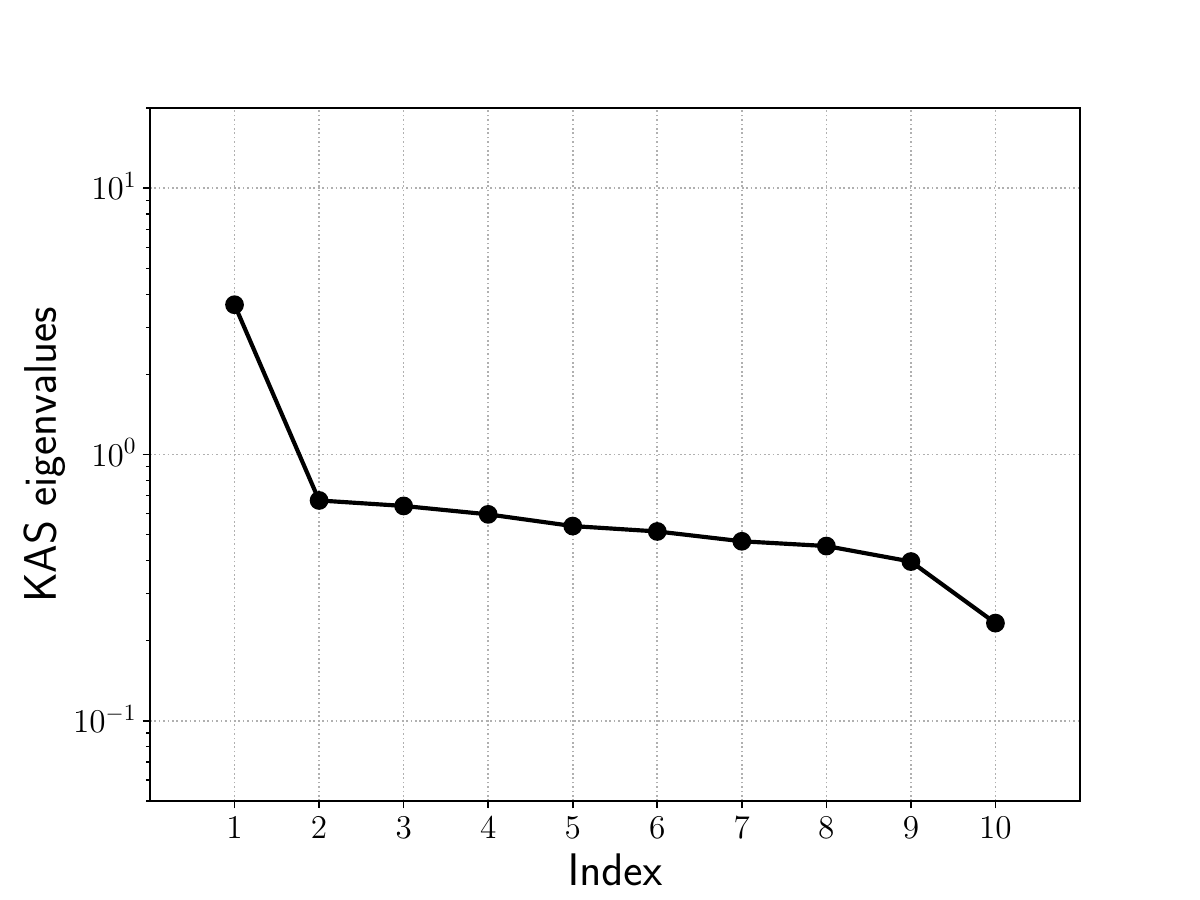}\\
\caption{\RB{Eigenvalues of the covariance matrix
  $\Tilde{H}\in\mathbb{R}^{8\times 8}$ applied to
  the hyperparaboloid case for the AS procedure on the left, and the first $10$ eigenvalues of the
  covariance matrix $\Tilde{H}\in\mathbb{R}^{1000\times 1000}$ for the
  KAS procedure applied to the same case on the right.}}
\label{fig:evals_hyperparaboloid}
\end{figure}

The one-dimensional sufficient summary plots, which are $f(\x)$
against $W_1^T \x$ --- in the AS case --- or against $W_1^T \phi
(\x)$ --- in the KAS case ---, are shown in \autoref{pics:hyperparaboloid} and
\autoref{pics:sin}, respectively. On the left panels we present the
Gaussian process
response surfaces obtained from the active subspaces reduction, while
on the right panels the ones obtained with the kernel-based AS
extension. As we can see AS fails to properly reduce the parameter
spaces, since there are no preferred directions over which the model
functions vary the most. The KAS approach, on the contrary, is able to
unveil the corresponding generatrices. This results in a reduction of the RMS by
a factor of at least $3$ (see \autoref{tab:res_tests}).

\begin{figure}[ht!]
\centering
\includegraphics[width=.49\textwidth]{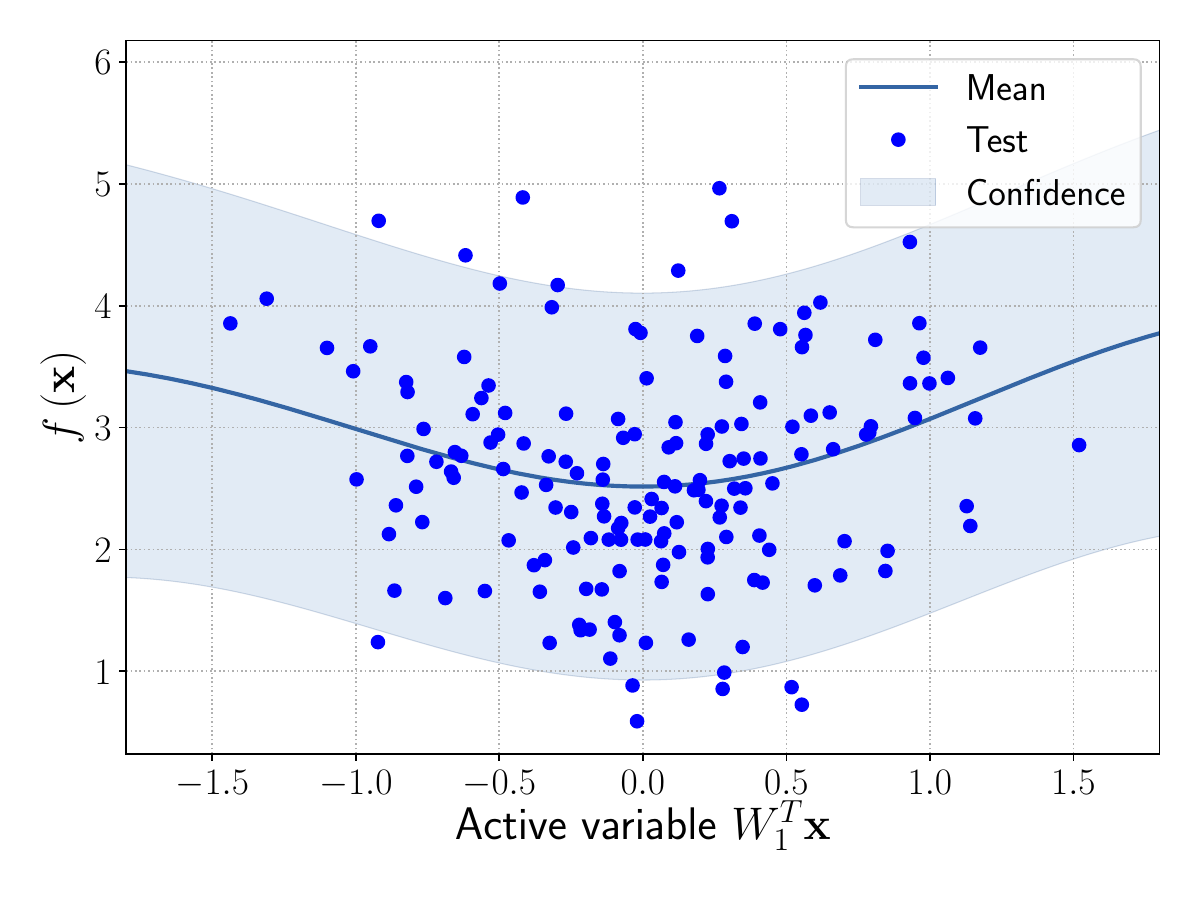}\hfill
\includegraphics[width=.49\textwidth]{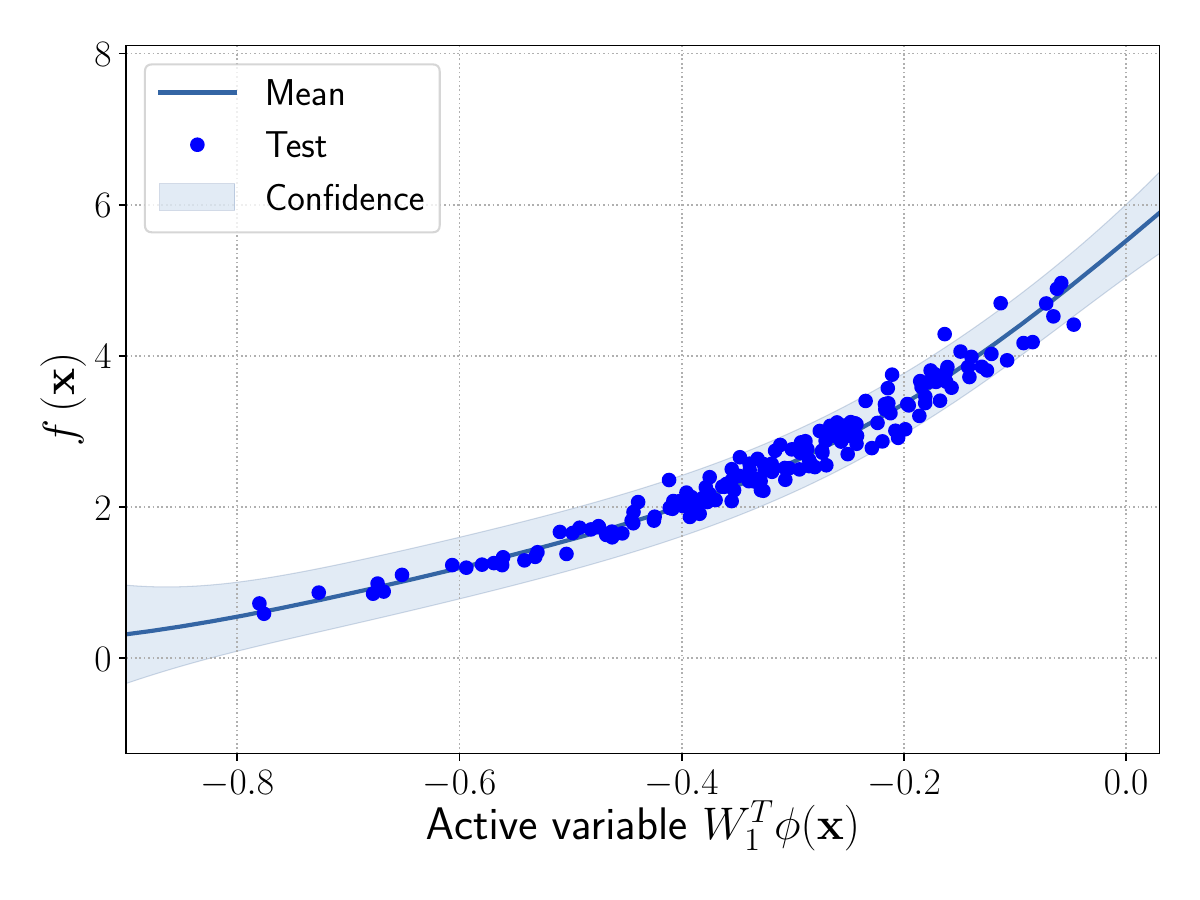}\\
\caption{Comparison between the sufficiency summary plots obtained from
  the application of AS and KAS methods for the hyperparaboloid model
  function with domain $[-1,1]^8$, defined in \autoref{eq:paraboloid}. The left plot refers to AS, the
  right plot to KAS. With the blue solid line we depict the posterior mean of
  the GP, with the shadow area the $68\%$ confidence intervals,
  and with the blue dots the testing points.}
\label{pics:hyperparaboloid}
\end{figure}

\begin{figure}[ht!]
\centering
\includegraphics[width=.49\textwidth]{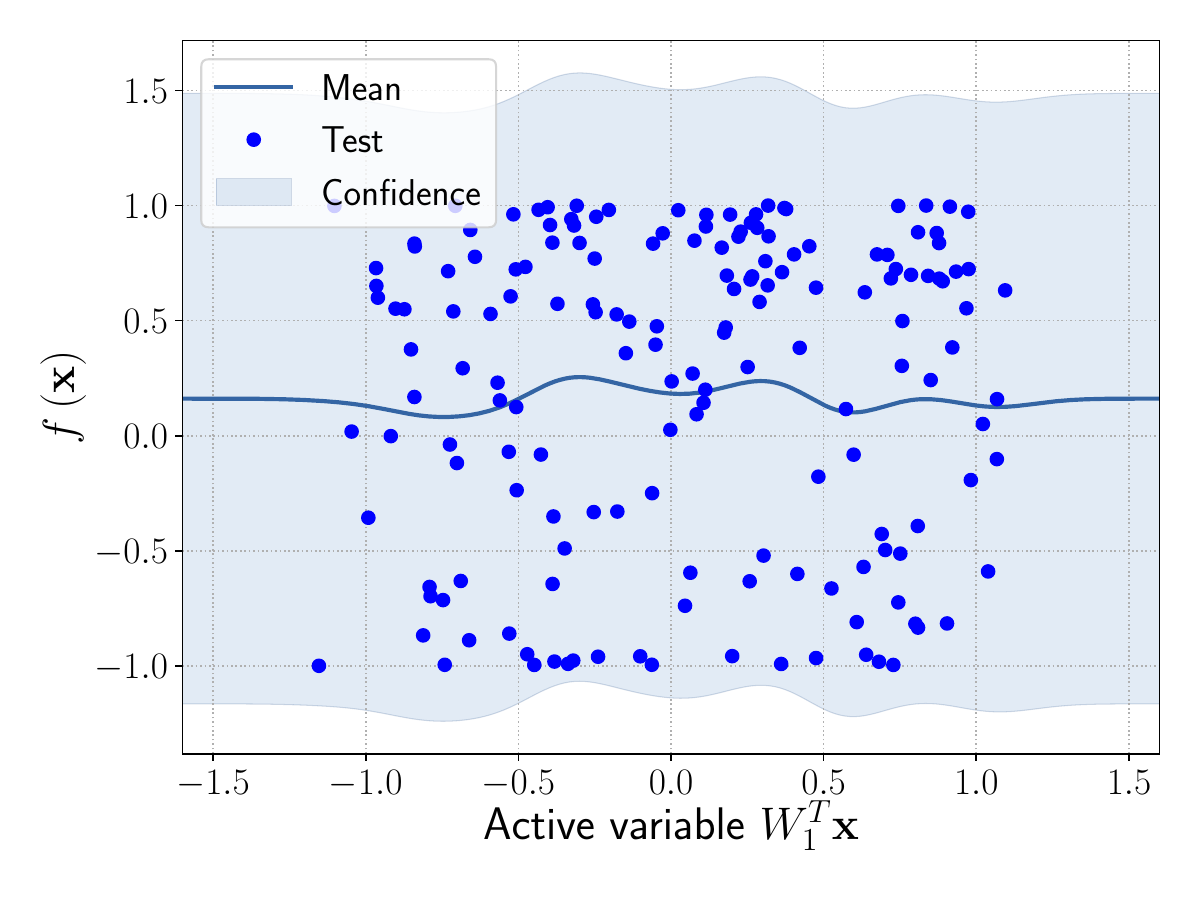}\hfill
\includegraphics[width=.49\textwidth]{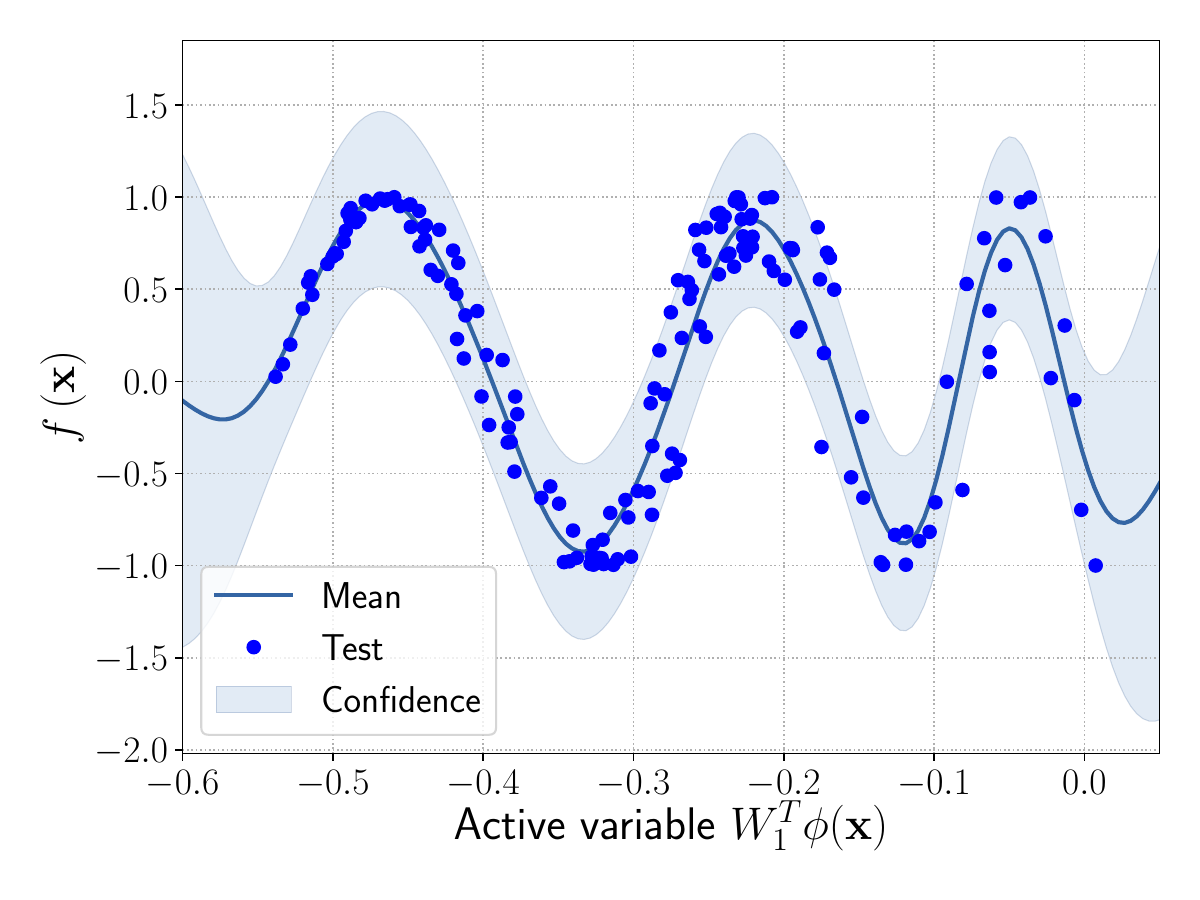}\\
\caption{Comparison between the sufficiency summary plots obtained from
  the application of AS and KAS methods for the surface
  of revolution model function with domain $[-3,3]^2$, defined
  in \autoref{eq:sine}. The left plot refers to AS, the
  right plot to KAS. With the blue solid line we depict the posterior mean of
  the GP, with the shadow area the $68\%$ confidence intervals,
  and with the blue dots the testing points.}
\label{pics:sin}
\end{figure}

\subsection{SEIR model for the spread of Ebola}
In most engineering applications the output of interest presents a
monotonic behaviour with respect to the parameters. This means that,
for example, the increment in the inputs produces a proportional response
in the outputs. Rarely the model function has a radial symmetry, and
in such cases the parameter space can be divided in subdomains, which
are analyzed separately. In this section we are going to present a
test case where there is no radial symmetry, showing that, even in
this case the kernel-based AS presents better performance with respect
to AS.

For the Ebola test case\footnote{The dataset was taken from
  \url{https://github.com/paulcon/as-data-sets}.}, the output of
interest is the basic reproduction
number $R_0$ of the SEIR model, described in~\cite{diaz2018modified},
which reads
\begin{equation}
\label{eq:ebola}
  R_0 =\frac{\beta_1 +\frac{\beta_2\rho_1 \gamma_1}{\omega} +
  \frac{\beta_3}{\gamma_2} \psi}{\gamma_1+ \psi},
\end{equation}
with parameters distributed uniformly in $\Omega \subset \R^8$. The
parameter space $\Omega$ is an hypercube defined by the lower and upper bounds
summarized in \autoref{tab:ebola}.

\begin{table}[htp!]
\centering
\caption{Parameter ranges for the Ebola model. Data taken
  from~\cite{diaz2018modified}.\label{tab:ebola}}
\begin{tabular}{ l c c c c c c c c }
\hline
\hline
  & $\beta_1$ & $\beta_2$ & $\beta_3$ & $\rho_1$ & $\gamma_1$
  & $\gamma_2$ & $\omega$ & $\psi$ \\
\hline
\hline
\rowcolor{Gray}
Lower bound  & 0.1 & 0.1 & 0.05 & 0.41 & 0.0276 & 0.081 & 0.25 & 0.0833 \\
Upper bound & 0.4 & 0.4 & 0.2 & 1 & 0.1702 & 0.21 & 0.5 & 0.7 \\
\hline
\hline
\end{tabular}
\end{table}

We can compare the two one-dimensional response surfaces obtained with
Gaussian process regression. The training samples are $N_s = 800$, and
we use $1000$ features. As spectral measure we use again the multivariate
gaussian distribution $\mathcal{N}(\mathbf{0}, \Sigma)$ with
hyperparameters the elements of the diagonal of the covariance
matrix. The tuning is carried out with $5$-fold cross validation. Even
in this case, the KAS approach results in smaller RMS with respect to
the use of AS (around $60$\% less), as reported in \autoref{tab:res_tests}.
In \autoref{fig:ebola} we report the comparison of the two approaches
over an active subspace of dimension $1$.

\begin{figure}[ht!]
\centering
\includegraphics[width=.49\textwidth]{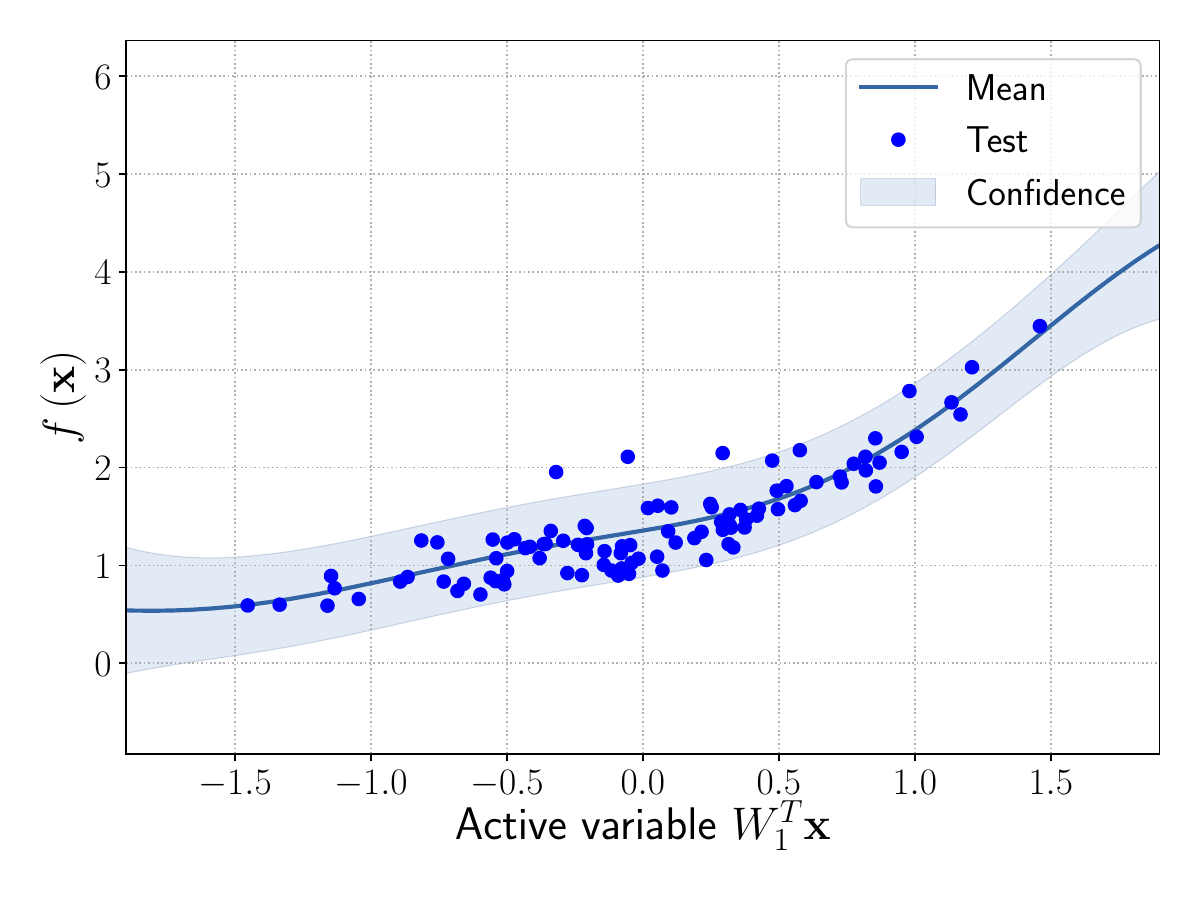}\hfill
\includegraphics[width=.49\textwidth]{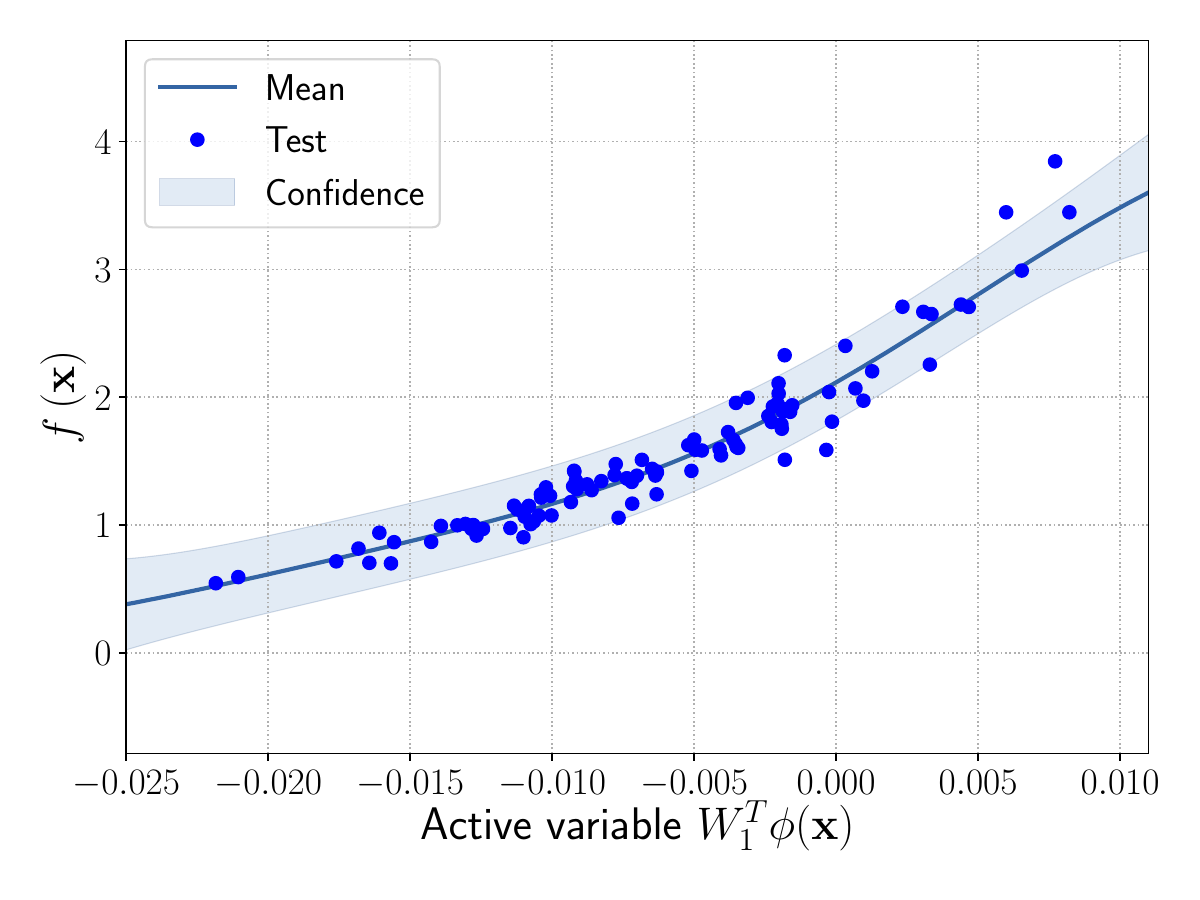}\\
\caption{Comparison between the sufficiency summary plots obtained from
  the application of AS and KAS methods for the $R_0$ model function
  with domain $\Omega$, defined in \autoref{eq:ebola}. The left plot
  refers to AS, the right plot to KAS. With the blue solid line we
  depict the posterior mean of the GP, with the shadow area the
  $68\%$ confidence intervals, and with the blue dots the testing points.}
\label{fig:ebola}
\end{figure}

\subsection{Elliptic Partial Differential Equation with random coefficients}
\label{ssec:spde}
In our last benchmark we apply the kernel-based AS to a
vectorial model function, that is the solution of a Poisson
problem with heterogeneous diffusion coefficient. We refer
to~\cite{zahm2020gradient} for an application, on the
same problem, of the AS approach.

We consider the following stochastic Poisson problem on the square $\mathbf{x}=(x, y)\in\Omega := [0,
1]^2$:
\begin{equation}
\label{eq:sPDE}
\begin{cases}
-\nabla\cdot (\kappa\ \nabla u)=1, & \x \in\Omega, \\
u = 0, & \x\in\partial\Omega_{\text{top}}\cup\partial\Omega_{\text{bottom}},\\
u = 10 y(1-y), &\x \in\partial\Omega_{\text{left}},\\
\mathbf{n}\cdot\nabla u = 0, & \x \in\partial\Omega_{\text{right}},
\end{cases}
\end{equation}
with homogeneous Neumann boundary condition on the right side of the domain, that is $\partial\Omega_{\text{right}}$, Neumann boundary conditions on the left side of the domain, that is $\partial\Omega_{\text{left}}$, and Dirichlet boundary conditions on the remaining part of $\partial\Omega$.
The diffusion coefficient $\kappa:(\Omega, \mathcal{A}, P)\times\Omega\rightarrow \R$, where
$\mathcal{A}$ is a $\sigma$-algebra, is such that $\log(\kappa)$ is a Gaussian random field, with covariance function $C(\x,\y)$ defined by
\begin{equation}
C(\x, \y) = \exp\left(-\frac{\lVert \x - \y \rVert^{2}}{\beta^{2}}
\right),\quad \forall \x,\y\in\Omega,
\end{equation}
where $\beta=0.03$ is the correlation length. This random field is approximated with the truncated Karhunen-Lo\`eve decomposition
\begin{equation}
\kappa(s, \x) \approx \exp\left(\sum_{i=0}^m X_i(s)
  \gamma_{i} \boldsymbol{\psi}_i (\x) \right)\qquad\forall (s,
\x) \in \Omega\times\Omega,
\end{equation}
 where $(X_{i})_{i\in 1,\dots, m}$ are independent standard normal
 distributed random variables, and $(\gamma_{i},
 \boldsymbol{\psi}_{i})_{i\in 1,\dots, d}$ are the eigenpairs of the
 Karhunen-Lo\`eve decomposition of the zero-mean random field $\kappa$.

In our simulation the domain $\Omega$ is discretized with a triangular
unstructured mesh $\mathcal{T}$ with $3194$ triangles. The parameter
space has dimension $m=10$. The simulations are carried out with the finite
element method (FEM) with polynomial order one, and for each
simulation the parameters $(X_i)_{i=1,\dots m}$ are sampled from a
standard normal distribution.  The solution $u$ is evaluated at
$d=1668$ degrees of freedom, thus $(V, R_{V})\approx(\mathbb{R}^{d},
S+M)$ where the metric $R_{V}$ is approximated with the sum of the
stiffness matrix $S\in\mathbb{R}^{d}\times\mathbb{R}^{d}$ and the
mass matrix $M\in\mathbb{R}^{d}\times\mathbb{R}^{d}$. This sum is a
discretization of the norm of the Sobolev space $H^{1}(\Omega)$. The
number of features used in the KAS procedure is $D=1500$, the number
of different independent simulations is $M=1000$.

Three outputs of interest are considered. The first target function
$f:\R^m \rightarrow \R$ is the mean value of
the solution at the right boundary $\partial\Omega_{\text{right}}$, which reads
\begin{equation}
\label{eq:spde_1}
f(\mathbf{X})=
\frac{1}{|\partial\Omega_{\text{right}}|}\int_{\partial\Omega_{\text{right}}}u(s)\,ds,
\end{equation}
and it is used to tune the feature map minimizing the RRMSE of the
Gaussian process regression, as described in \RC{Algorithm~\autoref{algo: tuning}}. A
summary of the results for the first output is reported in
\autoref{tab:res_tests}. The plots of the regression are reported in
\autoref{fig:GPR Poisson}. Even in this case both from a qualitative
and a quantitative point of view, the kernel-based approach achieves
the best results.

\begin{figure}[ht!]
\centering
\includegraphics[width=.49\textwidth]{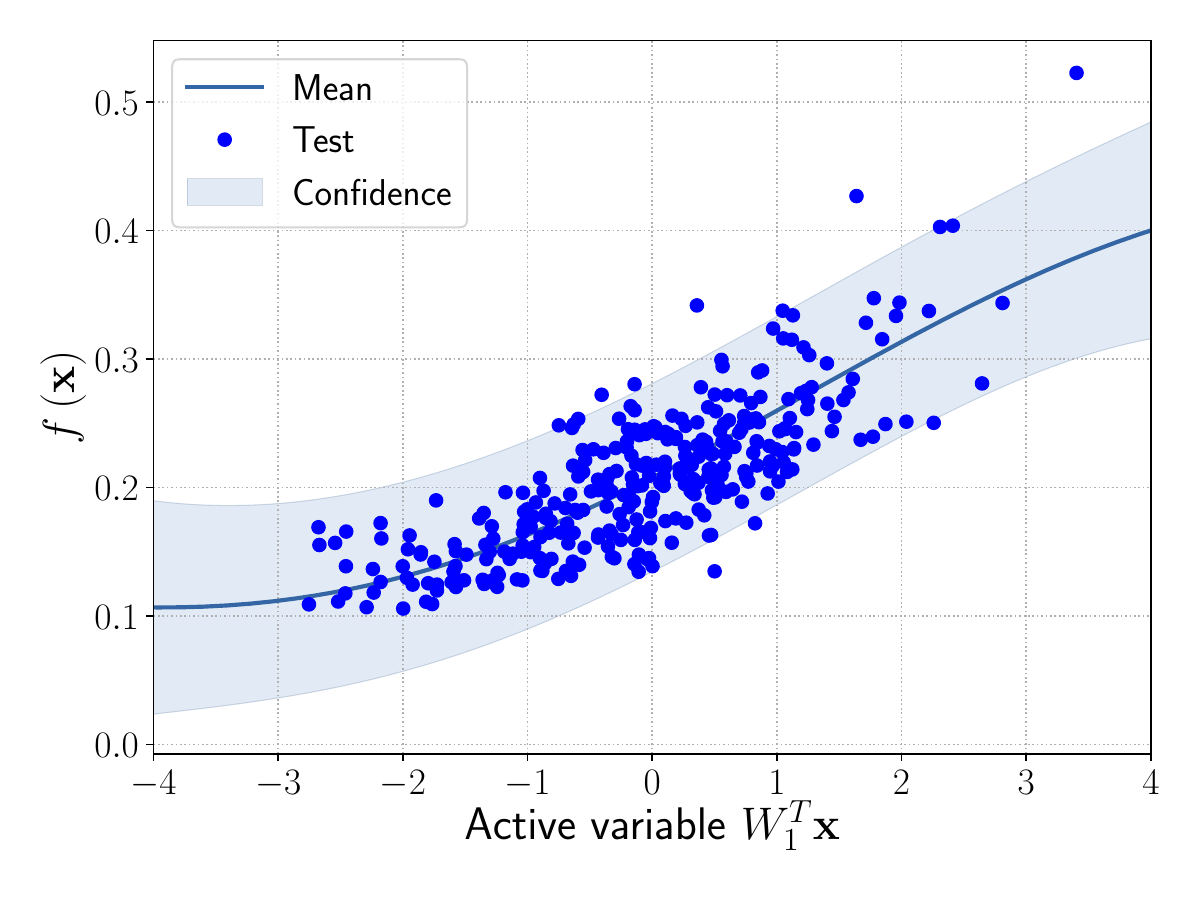}\hfill
\includegraphics[width=.49\textwidth]{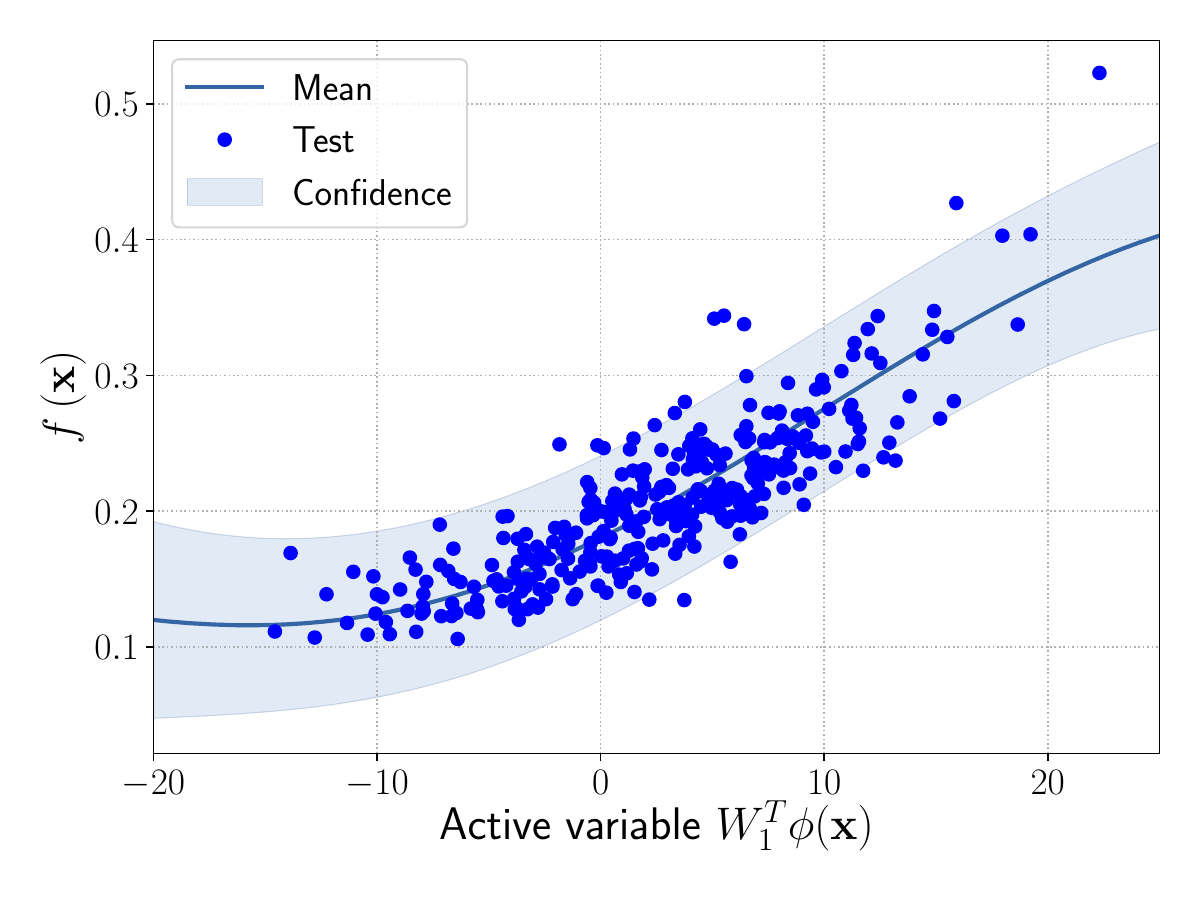}
\caption{Comparison between the sufficiency summary plots obtained from
  the application of AS and KAS methods for the stochastic PDE model, defined
  in Equations~\eqref{eq:sPDE} and~\eqref{eq:spde_1}. The left plot refers to AS, the
  right plot to KAS. With the blue solid line we depict the posterior mean of
  the GP, with the shadow area the $68\%$ confidence intervals,
  and with the blue dots the testing points.}
\label{fig:GPR Poisson}
\end{figure}

The second output we consider is the solution function
\begin{equation}
\label{eq:spde_2}
f:\R^m\rightarrow(V, R_{V})\approx(\R^d, S), \qquad
f(\X)=u\in\R^d.
\end{equation}
This output can be employed as a
surrogate model to predict the solution $u$ given the parameters
$\mathbf{X}$ that define the diffusion coefficient instead of carrying
out the numerical simulation. \RB{The surrogate model should be
  constructed over the span of the modes identified by the chosen
  reduction strategy, after projecting the data.} AS and KAS modes are
distinguished but can detect some common regions of interest as shown in
\autoref{tab:modes}.

The third output is the evaluation of the solution at a specific degree of freedom
with index $\hat{i}$, that is
\begin{equation}
\label{eq:spde_3}
f:\mathbb{R}^{m}\rightarrow\mathbb{R},\qquad
f(\X)=u_{\hat{i}}\in\R,
\end{equation}
in this case the dimension of the input space is
$m=100$. Since we use a Lagrangian basis in the finite element formulation and the polinomial order is 1, the node of the mesh associated to the chosen degree of freedom has coordinates $[0.27, 0.427]\in\Omega$. Qualitatively we can see from \autoref{tab:modes} that the AS
modes locate features in the domain which are relatively more regular
with respect to the KAS modes. To obtain this result we increased the dimension of the input space, otherwise not even the AS modes could locate properly the position in the domain $\Omega$ of the degree of freedom.

In the second and third case the diffusion coefficient is given by
\begin{equation}
\kappa(\x) = \exp\left(\sum_{i=1}^D \mathbf{v}_j[i]
  \Tilde{\boldsymbol{\psi}}_{j}(\x) \right)\qquad\forall (s,
\x) \in \Omega\times\Omega,
\end{equation}
where $\mathbf{v}_j\in\mathbb{R}^{D}$, $j\in\{1, \dots,D\}$, is the
$j$-th active eigenvector from the KAS procedure and the functions
$\Tilde{\Psi}:=(\Tilde{\boldsymbol{\psi}}_{1},\dots,
\Tilde{\boldsymbol{\psi}}_{D})$ are defined by
\begin{equation}
\Tilde{\Psi} = \phi(\Psi),
\end{equation}
where $\phi$ is the feature map defined in
Equation~\eqref{eq:feature_map} with the projection matrix $W$ and bias
$b$, and $\Psi:=(\gamma_{1}\boldsymbol{\psi}_{1},\dots,
\gamma_{m}\boldsymbol{\psi}_{m})$.

The gradients of the three outputs of interest considered are evaluated with the adjoint method.

\begin{table}
\centering
  \caption{First $3$ modes using Karhunen-Lo\`eve (K-L)
    decomposition, AS, and KAS, for the outputs defined in
    Equations~\eqref{eq:spde_1},~\eqref{eq:spde_2},
    and~\eqref{eq:spde_3}.\label{tab:modes}}
\begin{tabular}{c C{.22\textwidth} C{.22\textwidth} C{.22\textwidth}}
\toprule
Case & Mode $1$ & Mode $2$ & Mode $3$  \\
\midrule
K-L & \includegraphics[width=0.2\textwidth]{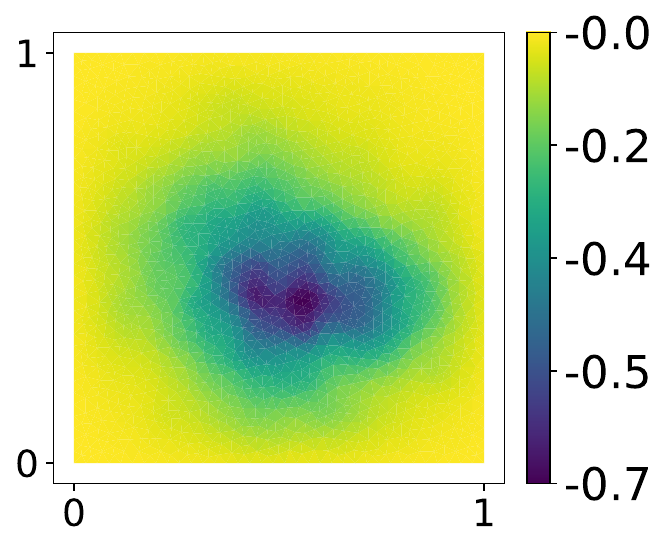}
& \includegraphics[width=0.2\textwidth]{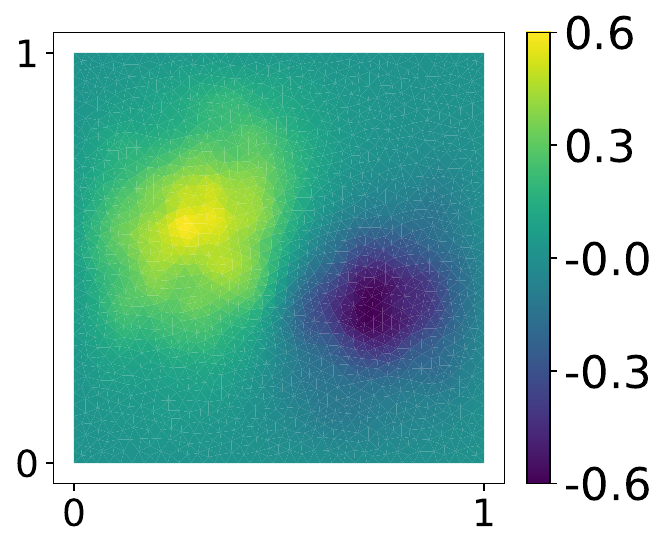}
& \includegraphics[width=0.2\textwidth]{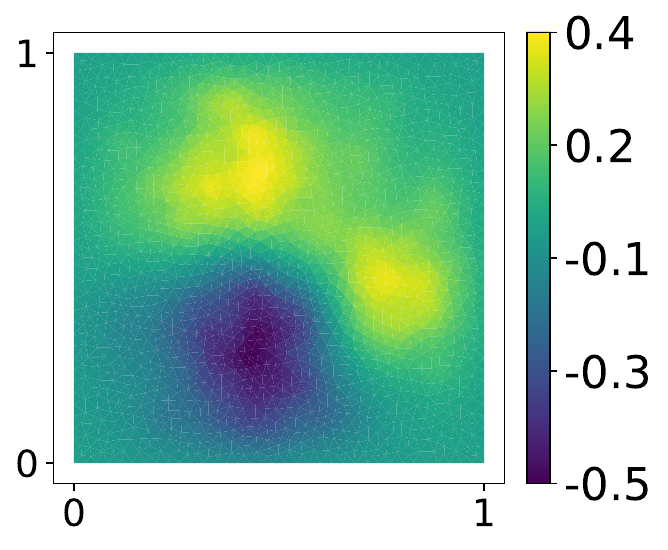}\\
AS \eqref{eq:spde_1} & \includegraphics[width=0.2\textwidth]{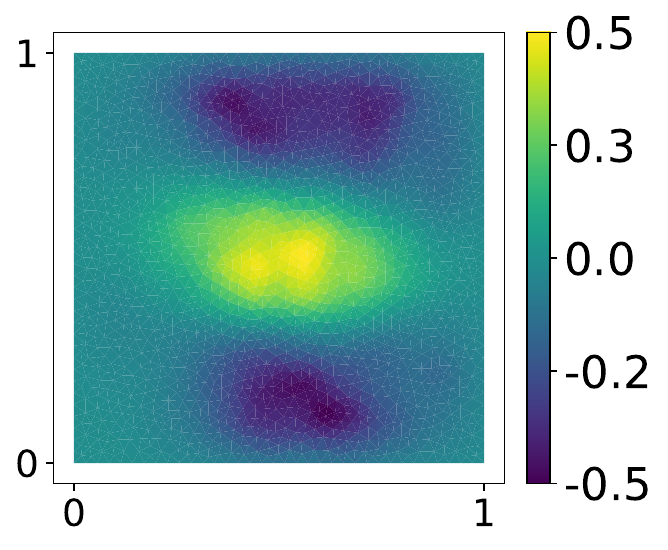}
& \includegraphics[width=0.2\textwidth]{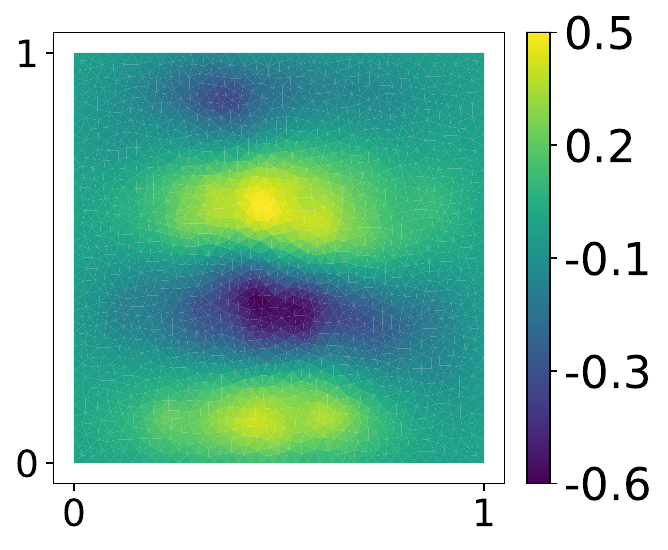}
& \includegraphics[width=0.2\textwidth]{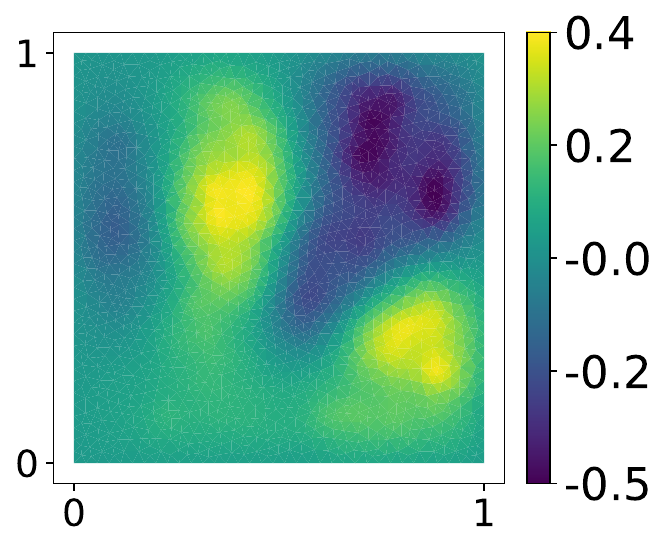}\\
KAS \eqref{eq:spde_1} & \includegraphics[width=0.2\textwidth]{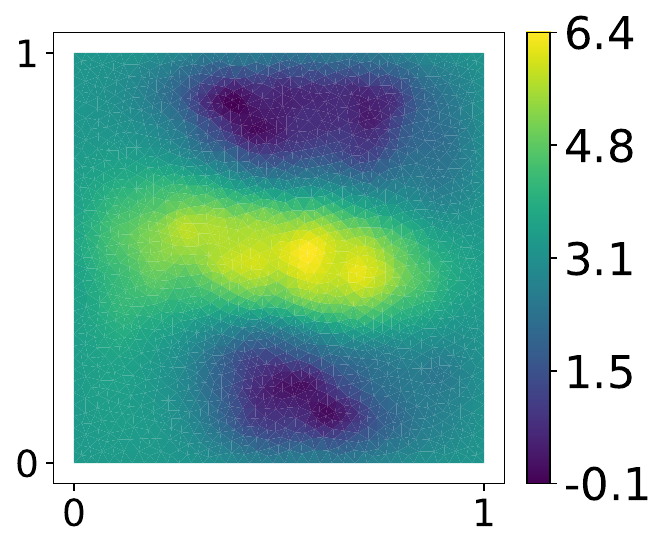}
& \includegraphics[width=0.2\textwidth]{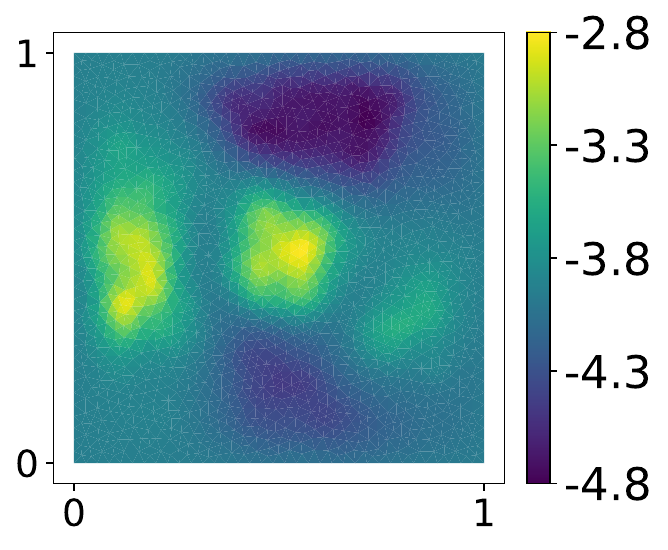}
& \includegraphics[width=0.2\textwidth]{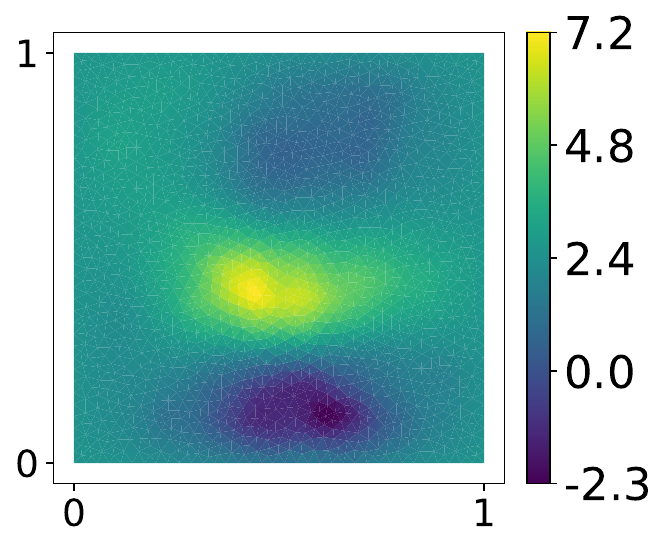}\\
AS \eqref{eq:spde_2} & \includegraphics[width=0.2\textwidth]{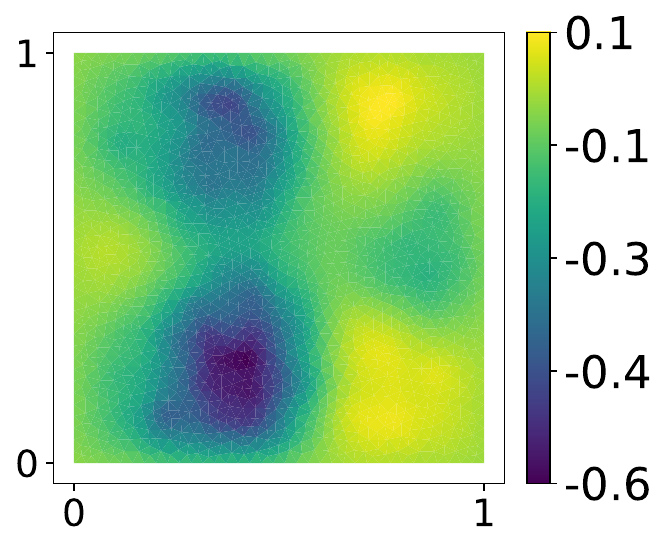}
& \includegraphics[width=0.2\textwidth]{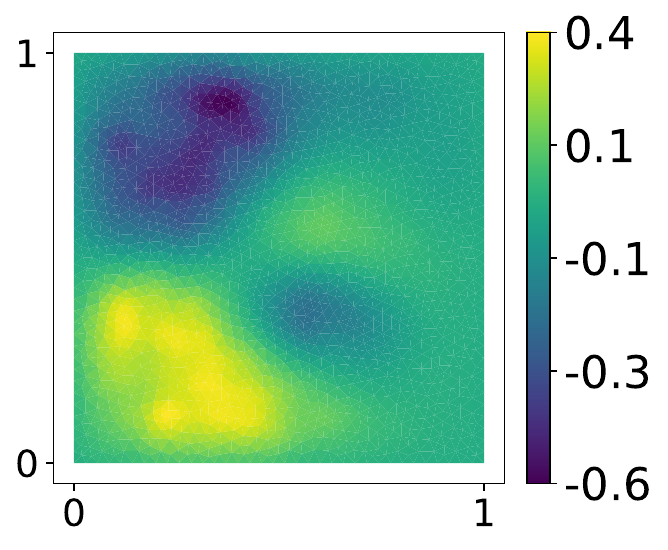}
& \includegraphics[width=0.2\textwidth]{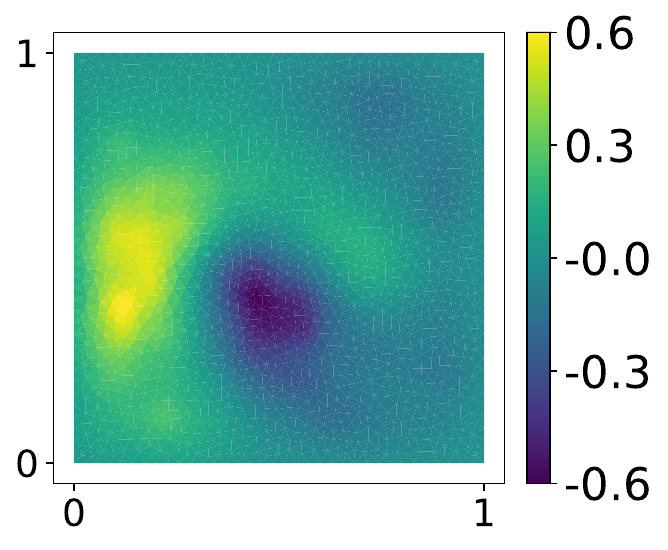}\\
KAS \eqref{eq:spde_2} & \includegraphics[width=0.2\textwidth]{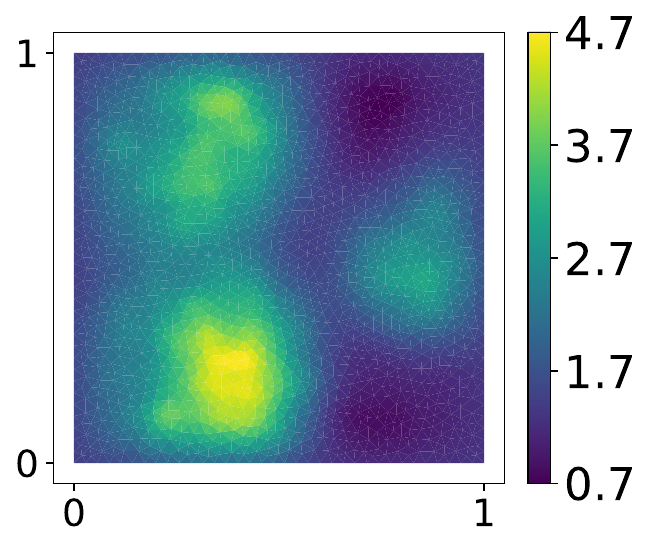}
& \includegraphics[width=0.2\textwidth]{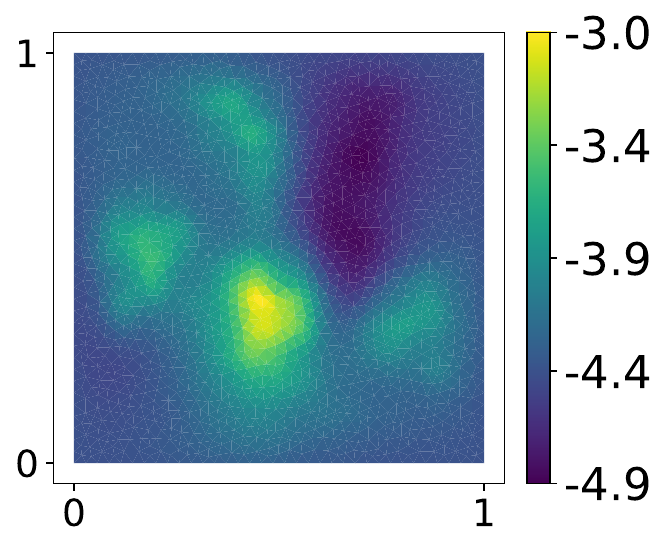}
& \includegraphics[width=0.2\textwidth]{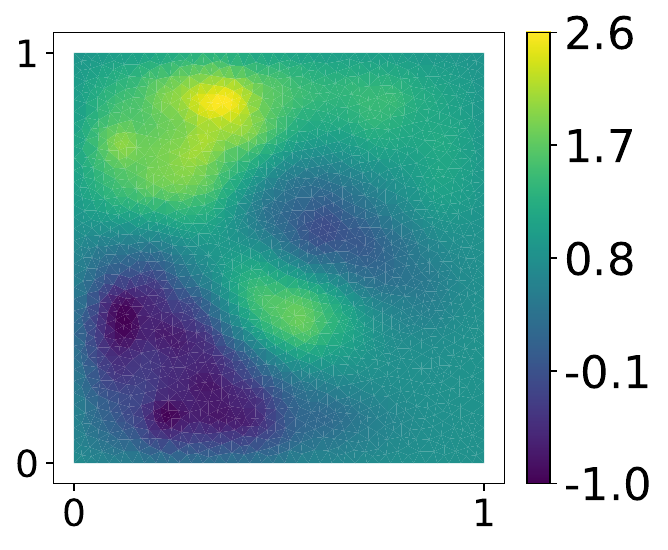}\\
AS \eqref{eq:spde_3} & \includegraphics[width=0.2\textwidth]{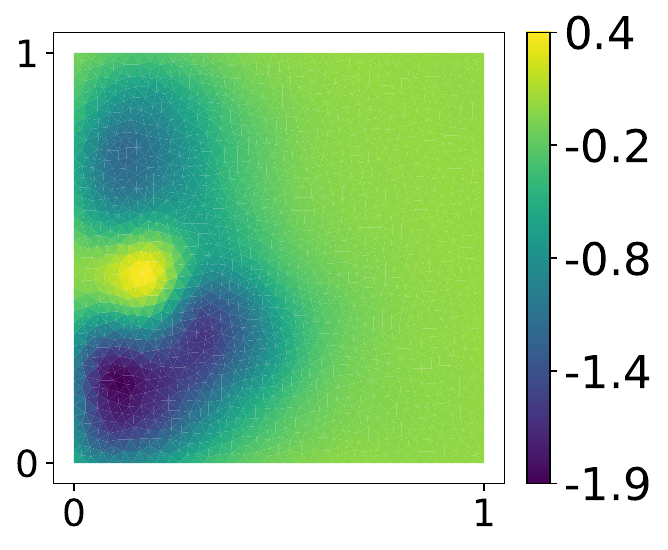}
& \includegraphics[width=0.2\textwidth]{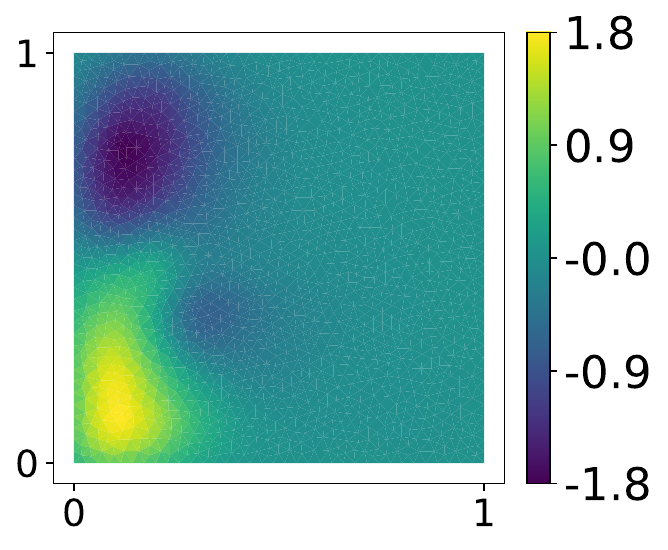}
& \includegraphics[width=0.2\textwidth]{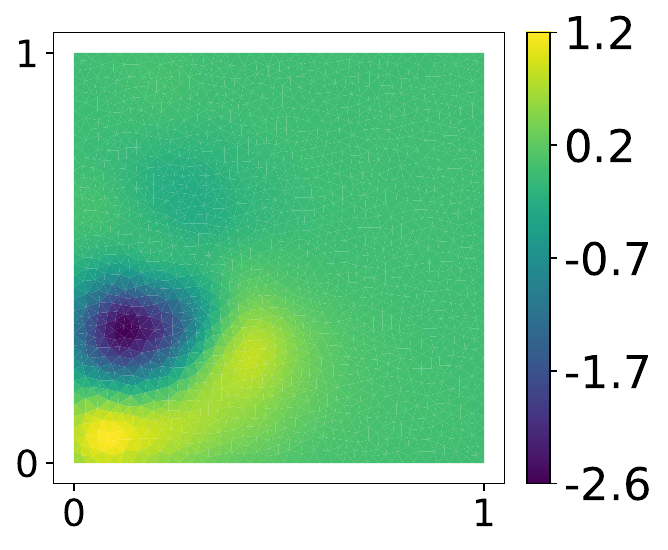}\\
KAS \eqref{eq:spde_3} & \includegraphics[width=0.2\textwidth]{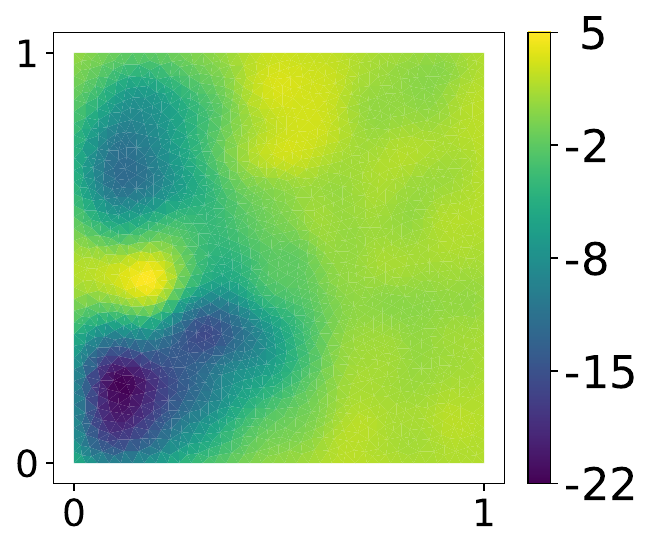}
& \includegraphics[width=0.2\textwidth]{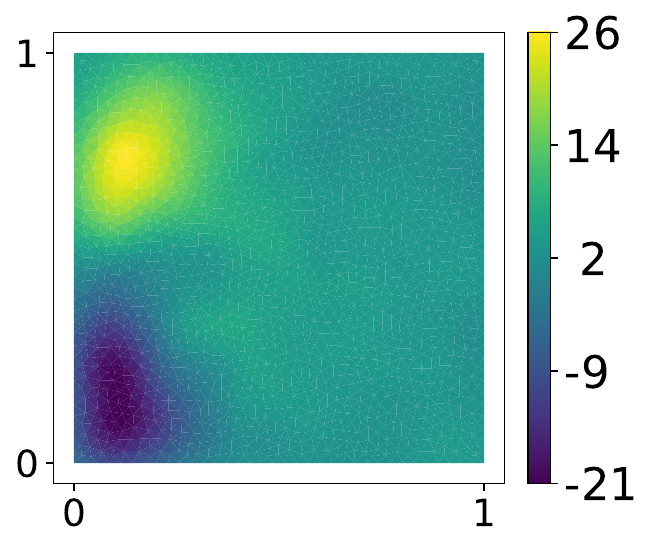}
& \includegraphics[width=0.2\textwidth]{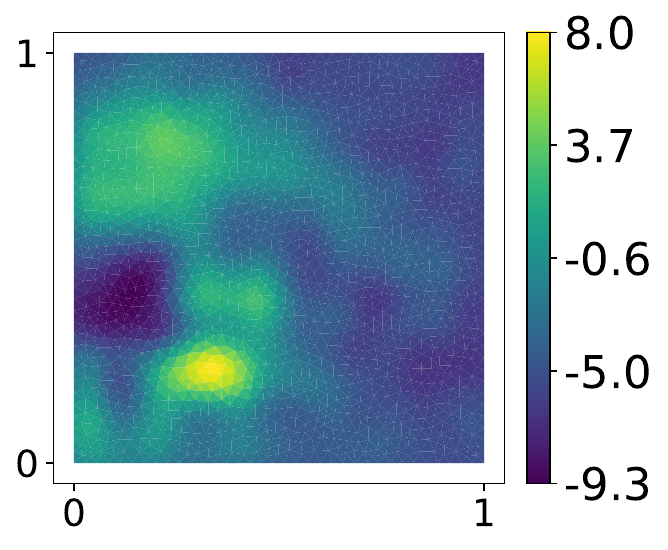}\\
\bottomrule
\end{tabular}
\end{table}

\section{A CFD parametric application of KAS solved with the DG method}
\label{sec:results}

We want to test the kernel-based extension of the active subspaces
in a computational fluid dynamics context. The lift and drag
coefficients of a NACA~0012 airfoil are considered as model
functions.
Numerical simulations are carried out with different input parameters for
quantities that describe the geometry and the physical conditions of
the problem. The evolution of the model is protracted until a periodic
regime is reached.  Once the simulation data have been
collected, sensitivity analysis is performed searching for an active
subspace and response surfaces with GPR are then built from the
application of AS and KAS techniques.

The fluid motion is modelled through the unsteady
incompressible Navier-Stokes equations approximated through the
Chorin-Temam operator-splitting method implemented in HopeFOAM~\cite{HopeFOAM}.
HopeFOAM is an extension of OpenFOAM~\cite{OpenFOAM, weller1998}, an open source
software for the solution of complex fluid flows problems, to variable
higher order  element method and it adopts a Discontinuous Galerkin Method,
based on  the formulation proposed by Hesthaven and
Warburton~\cite{hesthaven2007nodal}.

The Discontinuous Galerkin method (DG) is a high-order method, which
has appealing features such as the low artificial viscosity and a
convergence rate which is optimal also on unstructured grids, commonly
used in industrial frameworks.
In addition to this, DG is naturally suited for the solution of problems
described by conservative governing equations (Navier Stokes equations,
Maxwell's equations and so on) and for parallel computing.
All these properties are due to the fact that, differently from
formulations based on standard finite elements, no continuity is imposed
on the  cell boundaries and neighboring elements only exchange a common
flux. The major drawback of DG is its high computational cost with respect
to continuous Galerkin methods, due to the need of evaluating fluxes
during each time step and the presence of extra degree of freedoms in
correspondence of the elemental edges.

Nowadays efforts are aimed at applying the DG in problems which involve
deformable domains~\cite{zahr2016adjoint} and at improving the
computational efficiency of the DG adopting techniques based on
hybridization methods, matrix-free implementations, and massive
parallelization~\cite{nguyen2009implicit,pazner2017stage}.

\subsection{Domain and mesh description}
\label{ssec:mesh}
The domain $\Omega$ of the fluid dynamic simulation is a
two-dimensional duct with a sudden area expansion and a NACA~0012
airfoil is placed in the largest section.
The inflow $\partial\Omega_{I}$ is placed at the beginning of the
narrowest part of  the duct, and here the fluid velocity is set constant
along all the inlet boundary. The outlet is placed on the
right hand side and it is denoted with  $\partial\Omega_{O}$.
We refer with
$\partial\Omega_{W} := \partial\Omega\backslash
\{\partial\Omega_{O}\cup\partial\Omega_{I} \}$
to the boundaries of the airfoil and to the walls of the duct, where
no slip boundary conditions are applied.
The horizontal lengths of the sections of the channels are $0.6$~\si{m} and
$1.35$~\si{m}, respectively. The vertical length of the duct after the area
expansion is $0.4$~\si{m}, while the width of the first one depends on two
distinct parameters.
The airfoil has a chord-length equal to $0.1$~\si{m} but its position with
respect to the duct and its angle of attack are described by
geometric parameters. Further details about the geometric
parameterization of the geometry are provided in the following
section. A proper triangulation is designed with the aid of the
gmsh~\cite{gmsh} tool and the domain is discretized with $4445$
unstructured elements.

\RC{The evaluation of adimensional
magnitudes, commonly used for characterizing the fluid flow field,
requires the definition of some reference magnitudes.} For
the problem at hand we consider the equivalent diameter of the
channel in correspondence of the inlet as the reference lengthscale,
while the reference velocity is the one imposed at the inlet.

\subsection{Parameter space description}
We chose $7$ heterogeneous parameters for the model: $2$ physical, and
$5$ geometrical which describe the width of the channel and the
position of the airfoil. In \autoref{tab:naca_pars} are reported the
ranges for the geometrical and physical parameters of the
simulation. $U$ is the first component of the initial velocity, $\nu$
is the kinematic viscosity, $x_{0}$ and $y_{0}$ are the horizontal and
vertical components of the translation of the airfoil with respect to its reference position (see \autoref{fig:naca_pars}), $\alpha$ is the
angle of the counterclockwise rotation and the center of rotation is
located right in the middle of the airfoil, $y^+$ and $y^-$ are the
module of the vertical displacements of the upper and lower side of
the initial conduct from a prescribed position.
\begin{table}[htp!]
\centering
\caption{Parameter ranges for the NACA problem.\label{tab:naca_pars}}
\begin{tabular}{ l c c c c c c c }
\hline
\hline
  & $\nu$ & $U$ & $x_0$ & $y_0$ & $\alpha$ & $y^+$ & $y^-$ \\
\hline
\hline
\rowcolor{Gray}
Lower bound & 0.00036 & 0.5& -0.099& -0.035& 0& -0.02& -0.02 \\
Upper bound & 0.00060 & 2& 0.099& 0.035& 0.0698& 0.02& 0.02 \\
\hline
\hline
\end{tabular}
\end{table}

In \autoref{fig:naca_pars} are reported different configurations
of the domain for the minimum and maximum values of the parameters
$\alpha$, $x_{0}$, $y_{0}$, and the minimum opening of the
channel.
\begin{figure}
\centering
\includegraphics[width=0.32\textwidth]{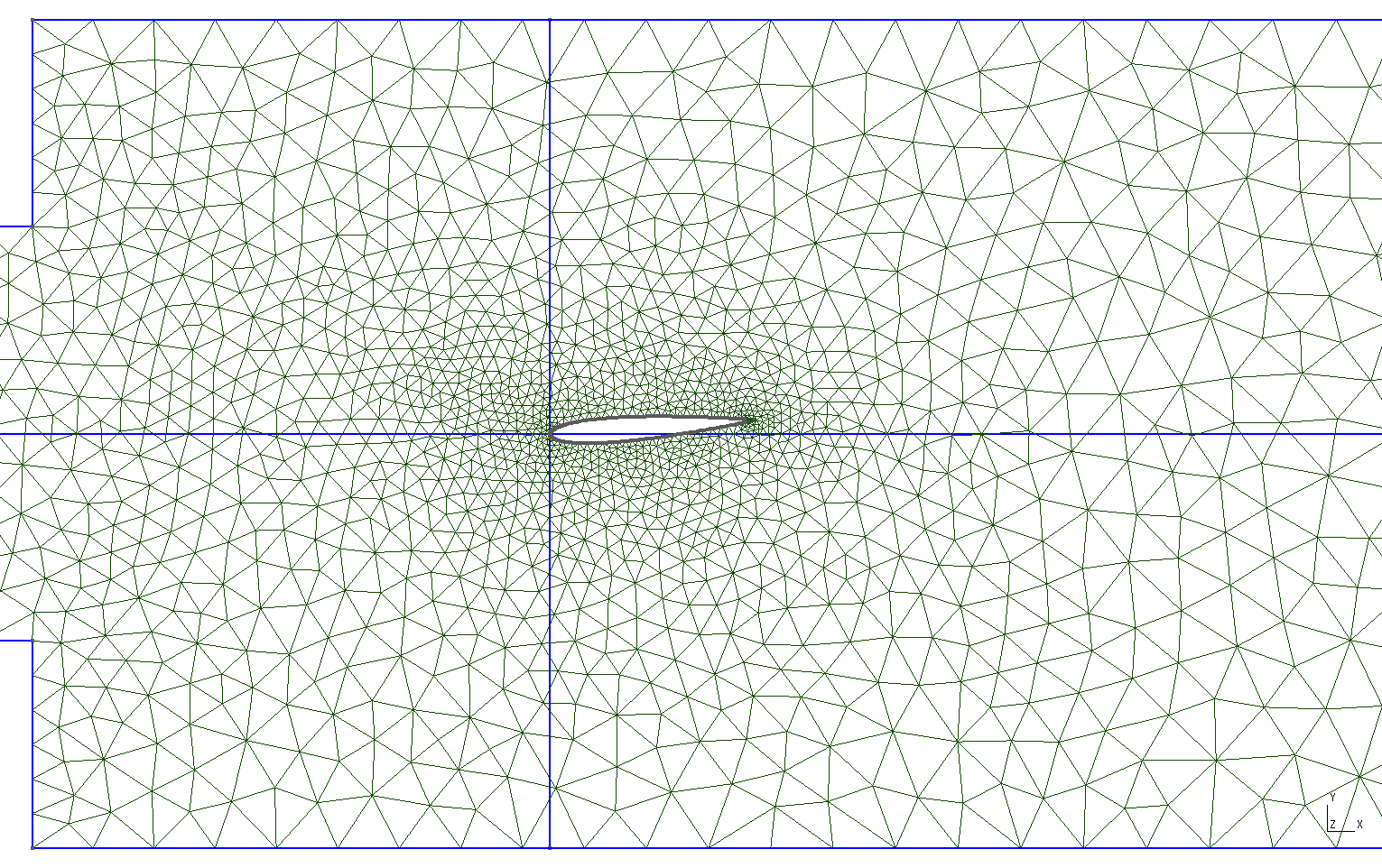}
\includegraphics[width=0.32\textwidth]{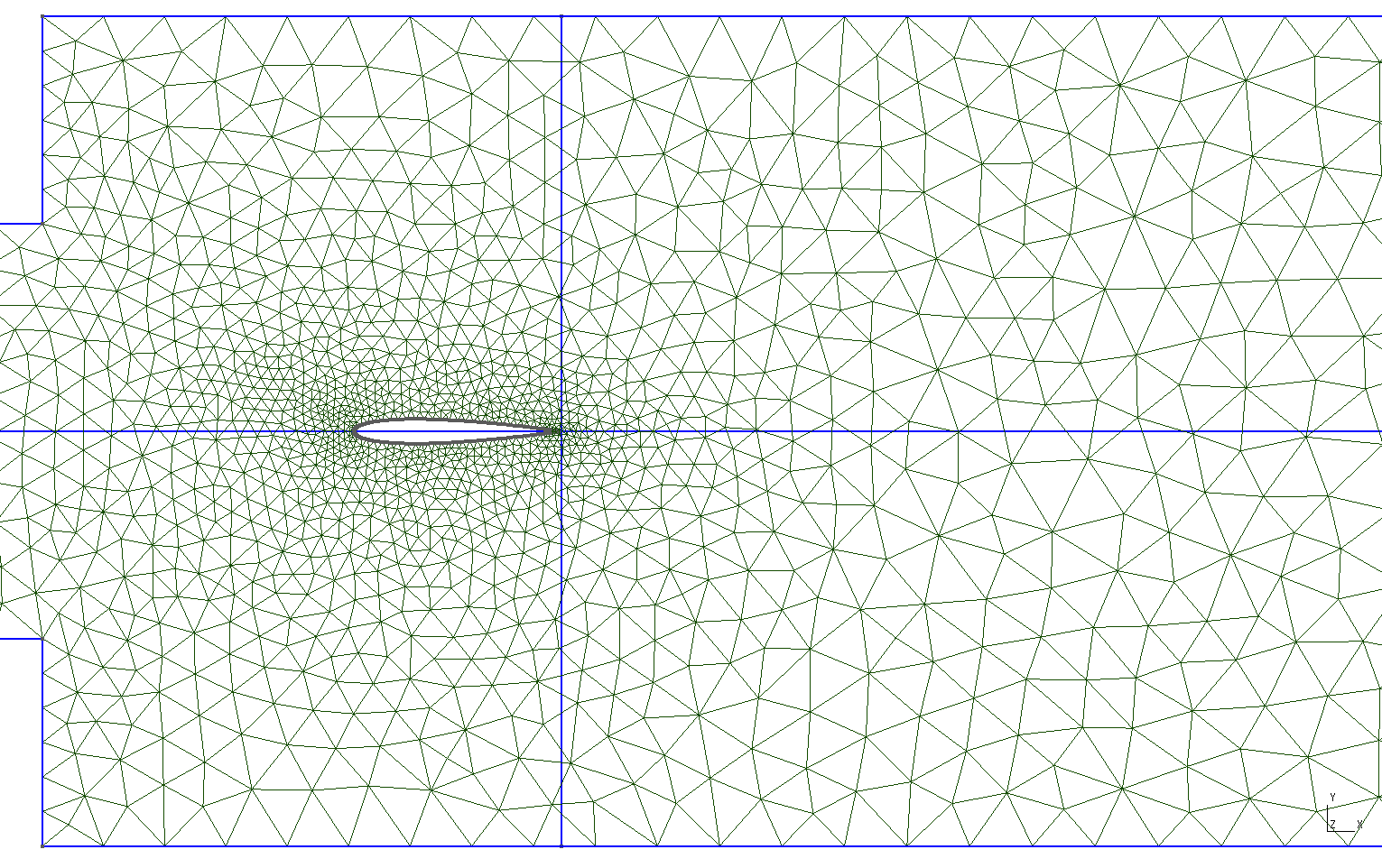}
\includegraphics[width=0.32\textwidth]{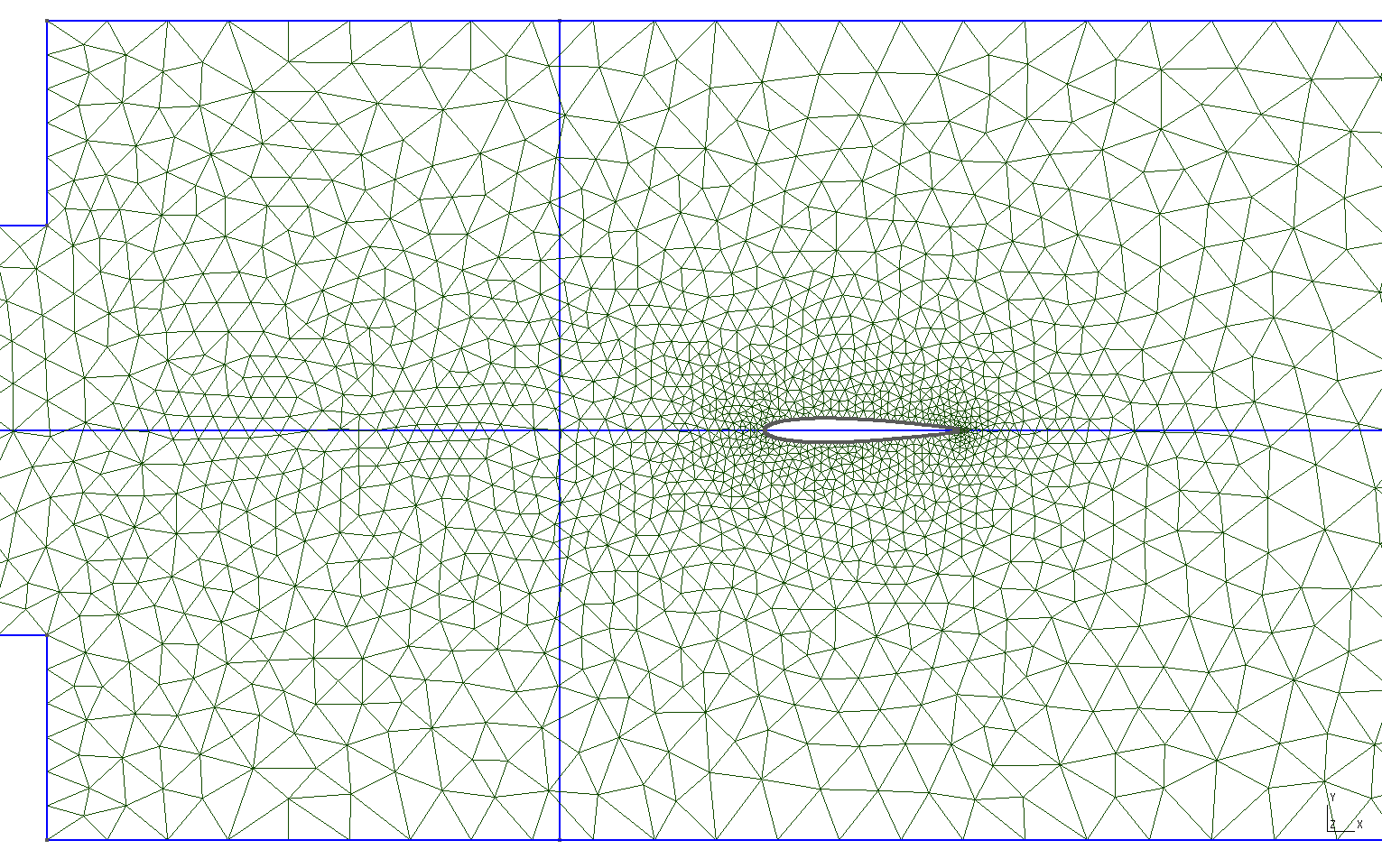}\\
\includegraphics[width=0.32\textwidth]{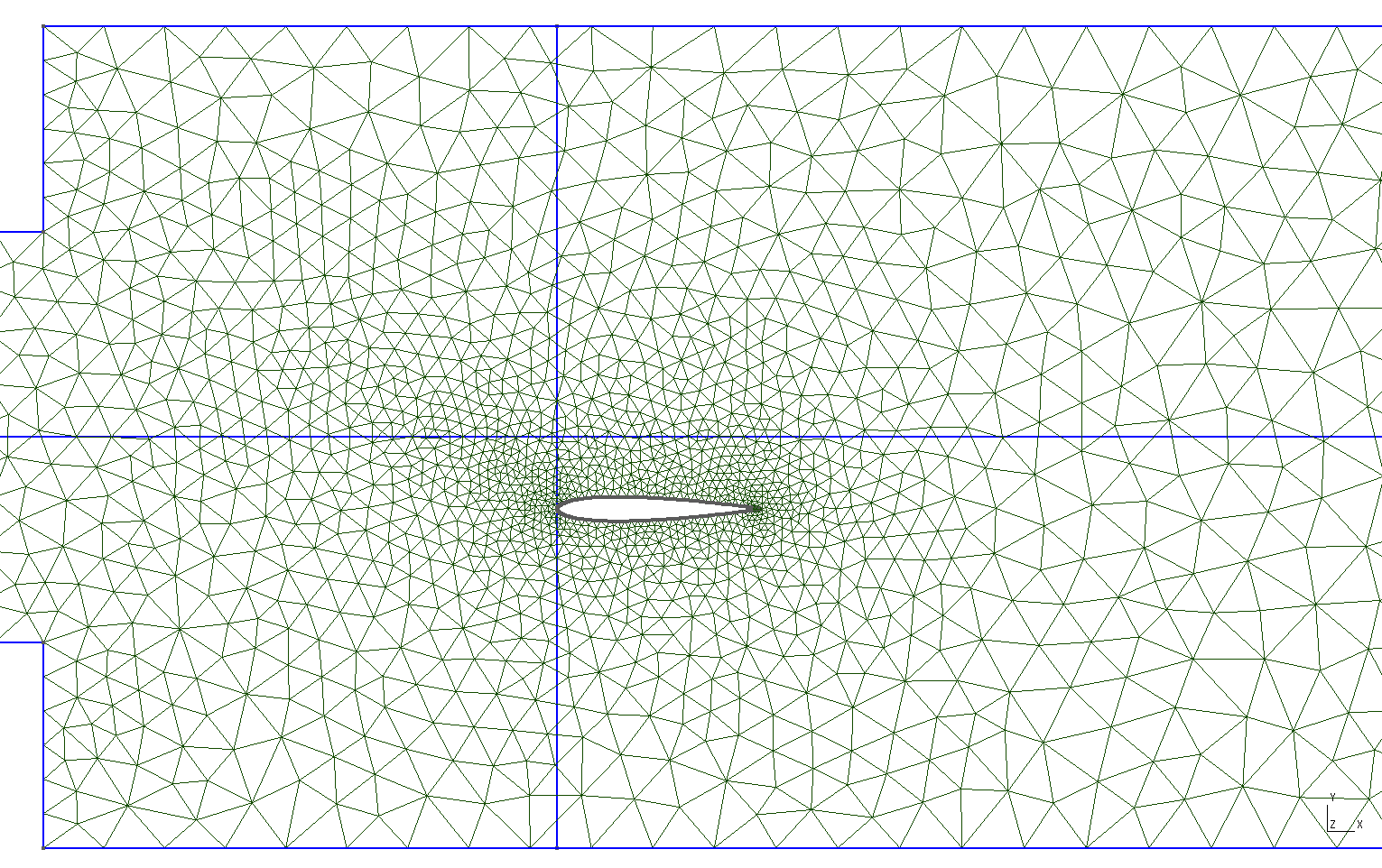}
\includegraphics[width=0.32\textwidth]{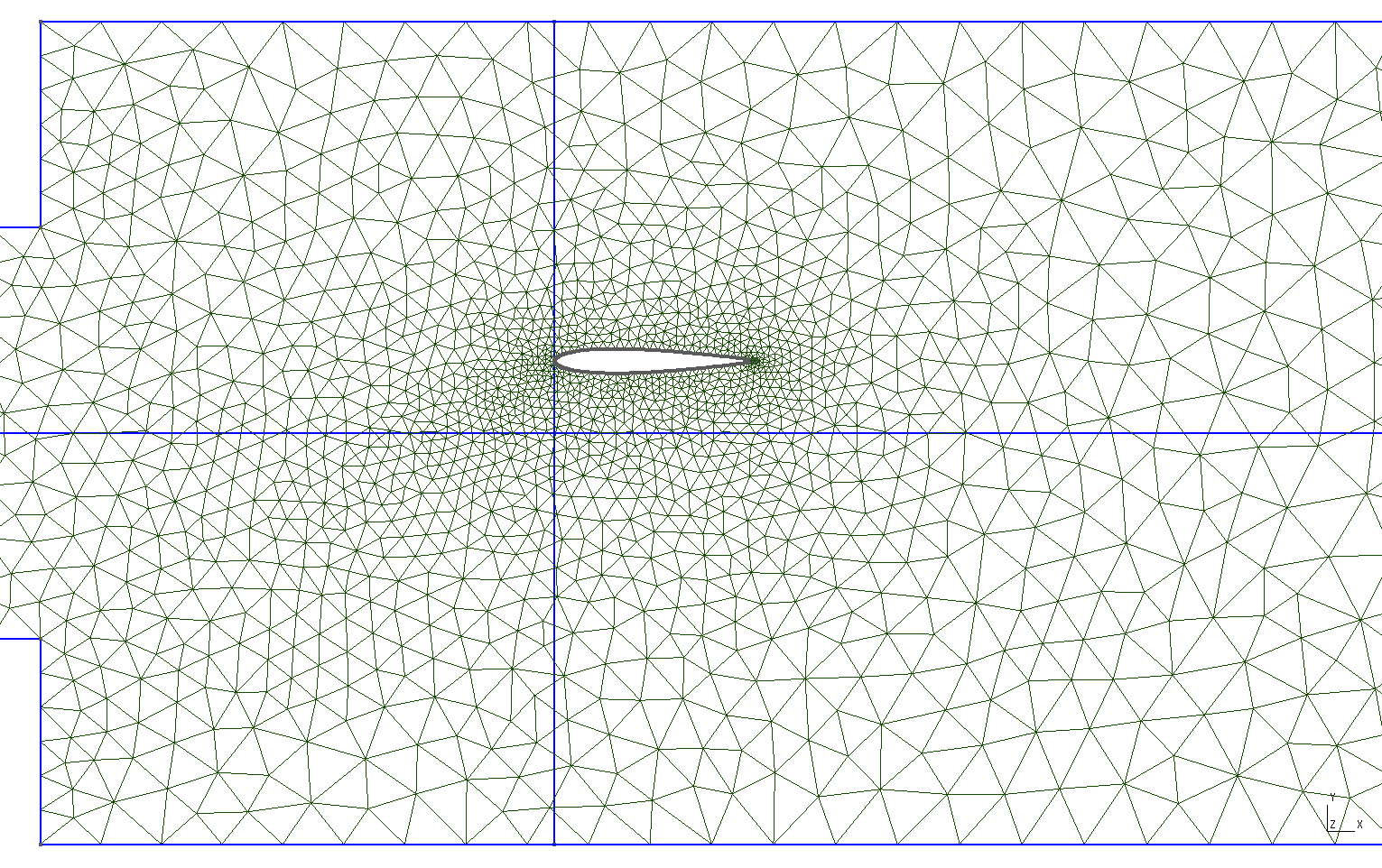}
\includegraphics[width=0.32\textwidth]{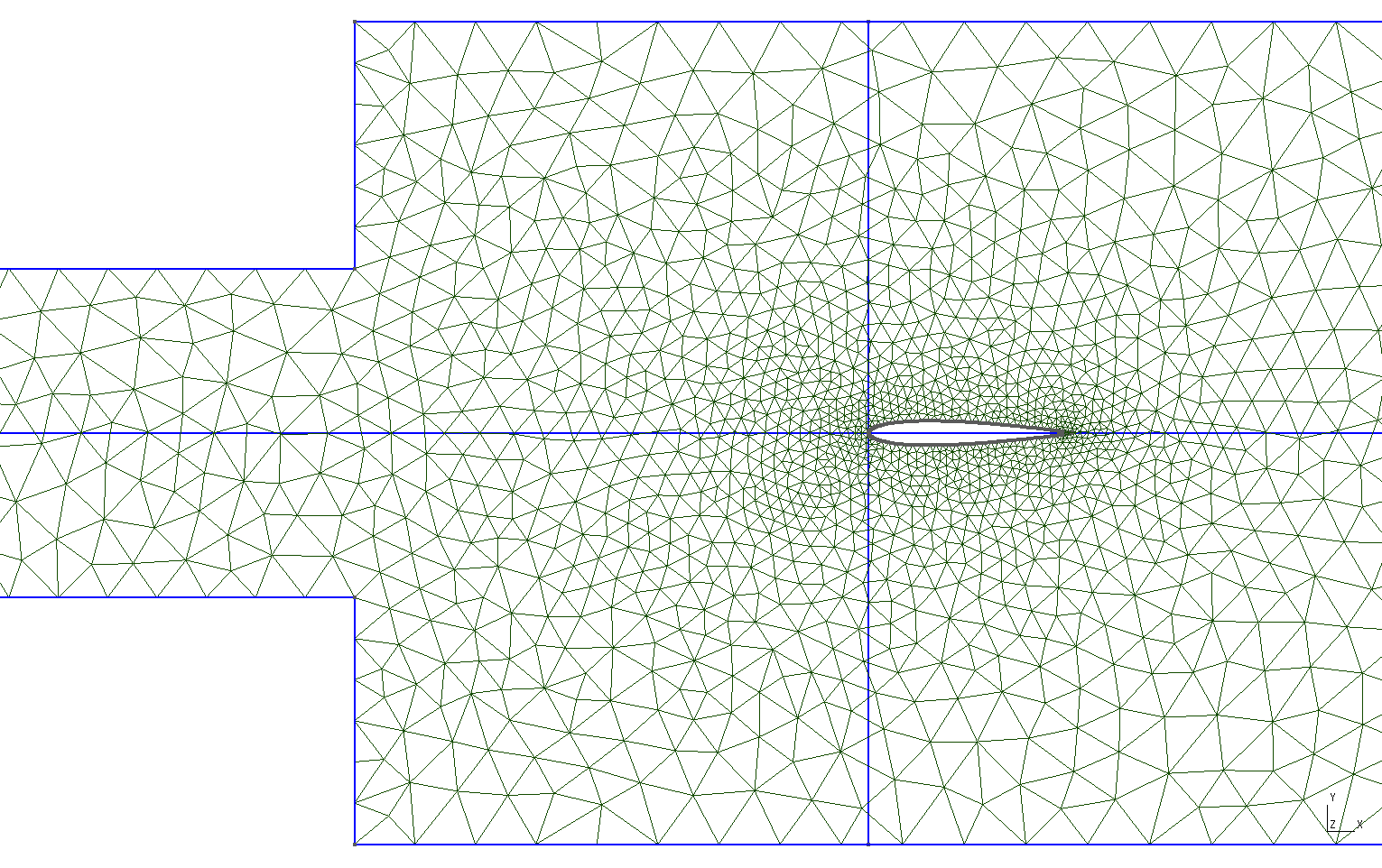}
\caption{Domain configuration for minimum and maximum values
  of some geometric parameters. In order are represented the maximum
  angle of attack $\alpha$, the ranges for the horizontal translation $x_{0}$,
  the ranges for the vertical translation $y_{0}$, and the minimum opening
  of the channel which depends on the parameters $y^+$ and $y^-$ in
  \autoref{tab:naca_pars}.}
\label{fig:naca_pars}
\end{figure}

We have considered only the counterclockwise rotation of the airfoil
for symmetrical reasons. The range of the Reynolds number varies from
$400$ to $2000$, still under the regime of laminar flow.

\subsection{Governing equations}
The CFD problem is modeled through the incompressible Navier-Stokes
and the open source solver HopeFOAM~\cite{HopeFOAM} has been employed
for solving this set of equations~\cite{hesthaven2007nodal}.

Let $\Omega\subset \R^2$ be the two-dimensional domain introduced in
\autoref{ssec:mesh}, and let us consider the incompressible
Navier-Stokes equations. Omitting the dependence on
$(\x, t) \in \Omega \times \R^+$ in the first two
equations for the sake of compactness, the governing equations are
\begin{align}
\label{eqincompressibleNS}
    \begin{cases}
        \partial_t \ub  + (\ub \cdot\nabla) \ub =-\nabla p+\nu \Delta
        \ub \qquad \qquad \qquad  & \x \in\Omega, \\
        \nabla\cdot\ub  = 0  & \x \in\Omega, \\
        \ub (\x ,0) = \ub _{0}, \quad p(\x ,0)=0  &\x \in\Omega, \\
        \ub (\x ,t) = \ub _{0}, \quad \mathbf{n}\cdot\nabla p(\x ,t)=0
        &\x \in\partial\Omega_{I}, \\
        \ub (\x ,t) = 0, \quad \mathbf{n}\cdot\nabla p(\x ,t)=0 &\x
        \in\partial\Omega_{W}, \\
        \mathbf{n} \cdot \nabla\ub (\x ,t)=0, \quad p(\x ,t)=1 &\x \in\partial\Omega_{O},
    \end{cases}
\end{align}
where $p$ is the scalar pressure field, $\ub=(u,v)$ is the velocity
field, $\nu$ is the viscosity constant and $\ub_0$ is the initial
velocity. In conservative form, the previous equations can be rewritten
as
\begin{equation}
\begin{cases}{}
    \partial_t \mathbf{u}+\nabla\cdot \mathcal{F}=-\nabla p + \nu\Delta\mathbf{u},\\
    \nabla\cdot \mathbf{u}=0,
\end{cases}
\end{equation}
with the flux $\mathcal{F}$ given by
\begin{equation}
    \mathcal{F}=\left[\mathbf{F}_{1},\mathbf{F}_{2}\right]=\left[\begin{array}{cc}
    u^{2}   &  uv\\
    uv     & v^{2}
    \end{array}
    \right].
\end{equation}
From now on, in order to have a more compact notation, the
 advection term is  written as $\mathcal{N}(\mathbf{u})=\nabla\cdot
 \mathcal{F}(\mathbf{u})$.

For each timestep the procedure is broken into three stages
accordingly to the algorithm proposed by Chorin and adapted for a DG
framework by Hesthaven et al.~\cite{hesthaven2007nodal}: the solution
of the advection dominated conservation law component, the pressure
correction weak divergence-free velocity projection, and the viscosity
update. The non-linear advection term is treated explicitly in time
through a second order Adams-Bashforth method~\cite{gazdag1976time}, while the diffusion
term implicitly. The Chorin algorithm is reported in
Algorithm~\ref{algo:chorin}.

In order to recover the Discontinuos Galerkin formulation, the
equations introduced by the Chorin method are projected onto the
solution space by introducing a proper set of test functions and then
the variables are approximated over each element as a linear
combination of local shape functions.
The DG does not impose the continuity of the solution between
neighboring elements and therefore it requires the adoption of methods
for the evaluation of the flux exchange between neighboring
elements. In the present work the convective fluxes are treated
accordingly to the Lax-Friedrichs scheme, while the viscous
ones are solved through the Interior Penalty
method~\cite{arnold1982interior, shahbazi2005explicit}.
\begin{algorithm}
\caption{Chorin Algorithm.}\label{algo:chorin}

\begin{algorithmic}[1]

\Require state variables $\mathbf{u}$ and $p$ at $t=0$, mesh, and boundary conditions

    \While{$t< t_{\text{final}}$}
    \State Update state variables $\mathbf{u}^{n-1} = \mathbf{u}^n$, $\mathbf{u}^n = \mathbf{u}^{n+1}$.
    \State Find a guess value for the velocity $\tilde{\mathbf{u}}$ by solving:
                \begin{equation*}
                \frac{\gamma_0 \Tilde{\mathbf{u}}-\alpha_0
                  \mathbf{u}^n - \alpha_1 \mathbf{u}^{n-1}}{\Delta t}
                = -\beta_0 \mathcal{N} (\mathbf{u}^n) - \beta_1
                \mathcal{N} (\mathbf{u}^{n-1}) .
                \end{equation*}
    \State Find the pressure at $n+1$ solving:
                $-\Delta\Bar{p}^{n+1} = -\frac{\gamma_0}{\Delta t}
                \nabla \cdot\Tilde{\mathbf{u}}$.
    \State Find the intermediate velocity $\Tilde{\Tilde{\mathbf{u}}} $ solving:
                $
                \gamma_0
                \frac{\Tilde{\Tilde{\mathbf{u}}}-\Tilde{\mathbf{u}}}{\Delta
                  t} = \nabla\Bar{p}^{n+1} $.
    \State Find the velocity at the $n+1$ time instant solving:
                $\gamma_0 \left(\frac{\mathbf{u}^{n+1} -
                     \Tilde{\Tilde{\mathbf{u}}}}{\Delta t}\right) = \nu \Delta \mathbf{u}^{n+1}$.
    \State Update $t^n$.
    \EndWhile

\State \Return state variables $\mathbf{u}$ and $p$ at $t=t_{\text{final}}$
\end{algorithmic}
\end{algorithm}

The aerodynamic quantities we are interested in are the lift and drag
coefficients in the incompressible case computed from the quantities
$\mathbf{u}$, $p$, $\nu$, $A_{\text{ref}}$, and $\mathbf{u}_0$
with a contour integral along the airfoil $\Gamma$ as
\begin{equation}
\label{eq:force lift and drag}
    f=
    \oint_{\Gamma}p\mathbf{n} -
    \nu\left(\nabla\mathbf{u}+\nabla\mathbf{u}^{T}
    \right)\mathbf{n}\,d\mathbf{s}.
\end{equation}
The vector $\mathbf{n}$ is
the outward normal along the airfoil surface. The circulation in
$\Gamma$ is affected by both the pressure and stress distributions
around the airfoil. The
projection of the force along the horizontal and vertical directions
gives the drag and lift coefficients respectively
\begin{equation}
\label{eq:Drag}
    C_D=\frac{f\cdot \mathbf{e}_{1}}{\frac{1}{2}|\mathbf{u}_0|^{2}A_{\text{ref}}},
\end{equation}
\begin{equation}
\label{eq:Lift}
    C_L=\frac{f\cdot \mathbf{e}_{2}}{\frac{1}{2}|\mathbf{u}_0|^{2}A_{\text{ref}}},
\end{equation}
where the reference area $A_{\text{ref}}$ is the chord of the
airfoil times a length of $1$~\si{m}. For the aerodynamic analysis of the fluid flow past an airfoil see~\cite{kundu2012fluid}.

\subsection{Numerical results}
In this section a brief review of the procedure and some details about
the numerical method and the computational domain will be presented
along the results obtained. For what concerns the DG the polynomial
order chosen is $3$. The total number of degrees of freedom is
$133350$. Small variations on the mesh are present in each of the $285$
simulations due to the different configurations of the domain.
Each simulation is carried out until a periodic behaviour is reached
and for this reason the final times range between $3.5$ and $5$~\si{s},
depending on the specific configuration.
The integration time intervals are variable and they are updated at
the end of each step in order to satisfy the CFL condition.
The $7$ physical and geometrical parameters of the simulation are
sampled uniformly from the intervals in \autoref{tab:naca_pars}. In
total we consider a dataset of $285$ samples.

With the purpose of qualitatively visualizing the results, $4$ different
simulations are reported in \autoref{fig:naca_press_vel} for the
module of the velocity field and the scalar
pressure field, respectively, both evaluated at the last time instant.
These simulations were chosen from the $285$ collected in order to show
significant differences in the evolution of the fluid flow. In
\autoref{table:sim_params} are reported the corresponding
parameters. Depending on the position of the airfoil and the other
physical parameters, different fluid flow patterns can be qualitatively
observed.

\begin{table}[htb!]
\centering
\caption{Parameters associated to the simulations plotted in
  \autoref{fig:naca_press_vel}. \label{table:sim_params}}
\begin{tabular}{ c c c c c c c c }
\hline
\hline
\# & $\nu$ & $U$ & $x_0$ & $y_0$ & $\alpha$ & $y^+$ & $y^-$ \\
\hline
\hline
\rowcolor{Gray}
1 & 0.000405 & 1.99 & -0.096 & -0.00207 & 0.00282 & 0.00784 & 0.0188\\
2 & 0.000541 & 0.763 & -0.084 & 0.00279 & 0.0260 & -0.0108 & 0.0195\\
\rowcolor{Gray}
3 & 0.000406 & 0.533 & -0.0503 & -0.0327 & 0.0604 & -0.0193 & 0.0068\\
4 & 0.000430 & 1.11 & -0.0897 & -0.0279 & 0.0278 & -0.00624 & 0.0197\\
\hline
\hline
\end{tabular}
\end{table}

\begin{figure}[htb!]
\includegraphics[width=.48\textwidth]{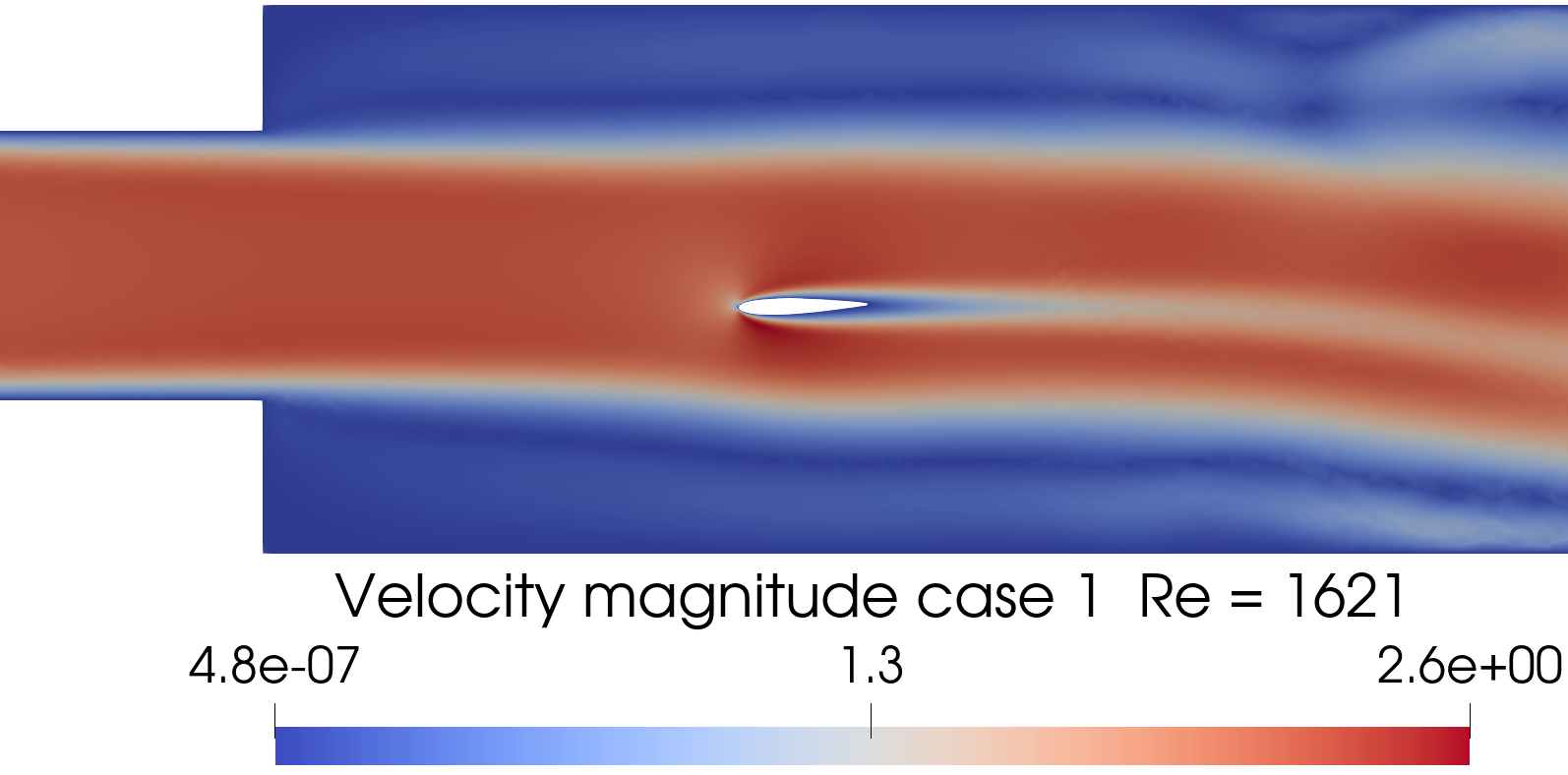}\hfill
\includegraphics[width=.48\textwidth]{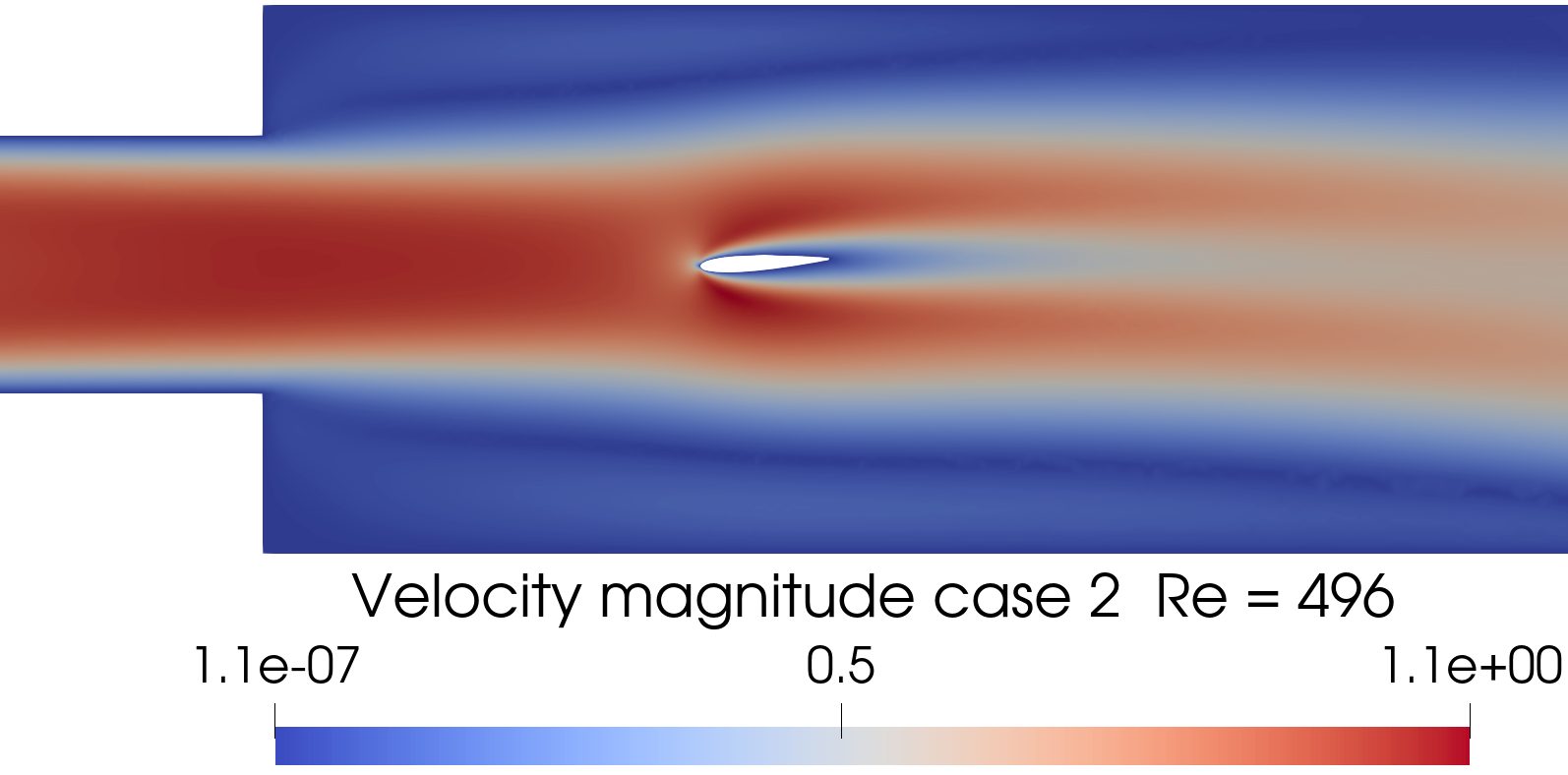}\\

\includegraphics[width=.48\textwidth]{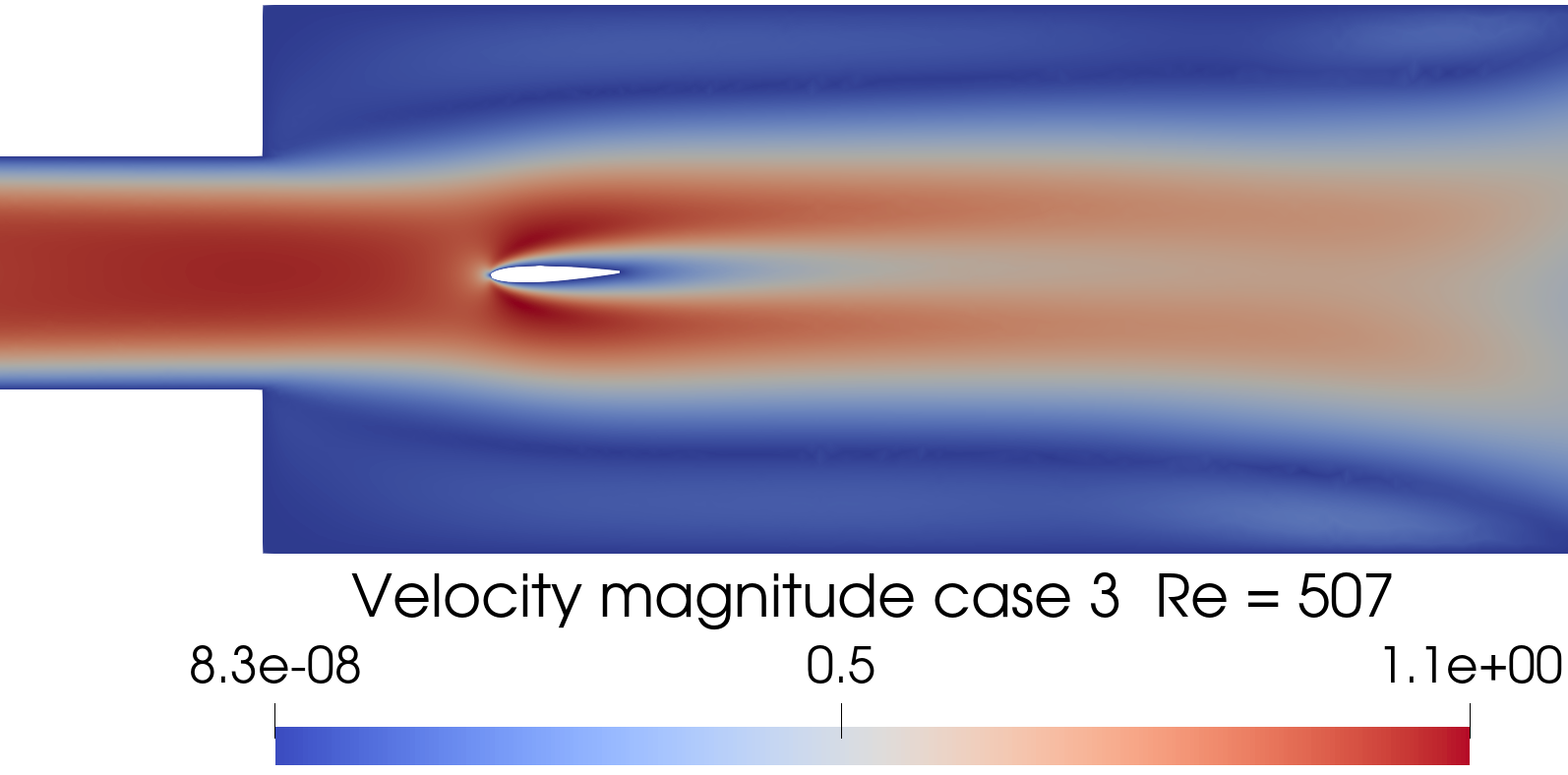}\hfill
\includegraphics[width=.48\textwidth]{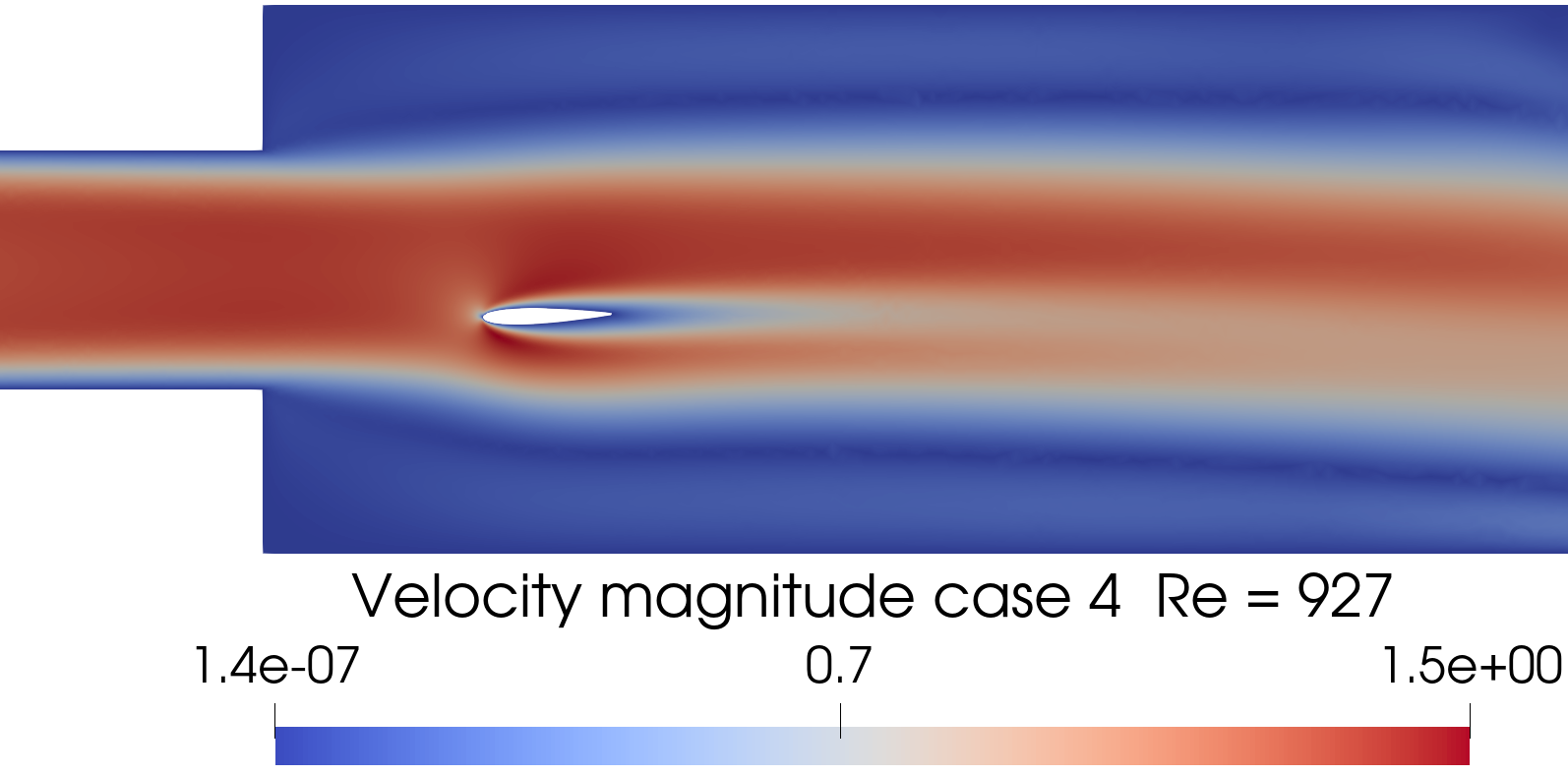}\\

\includegraphics[width=.48\textwidth]{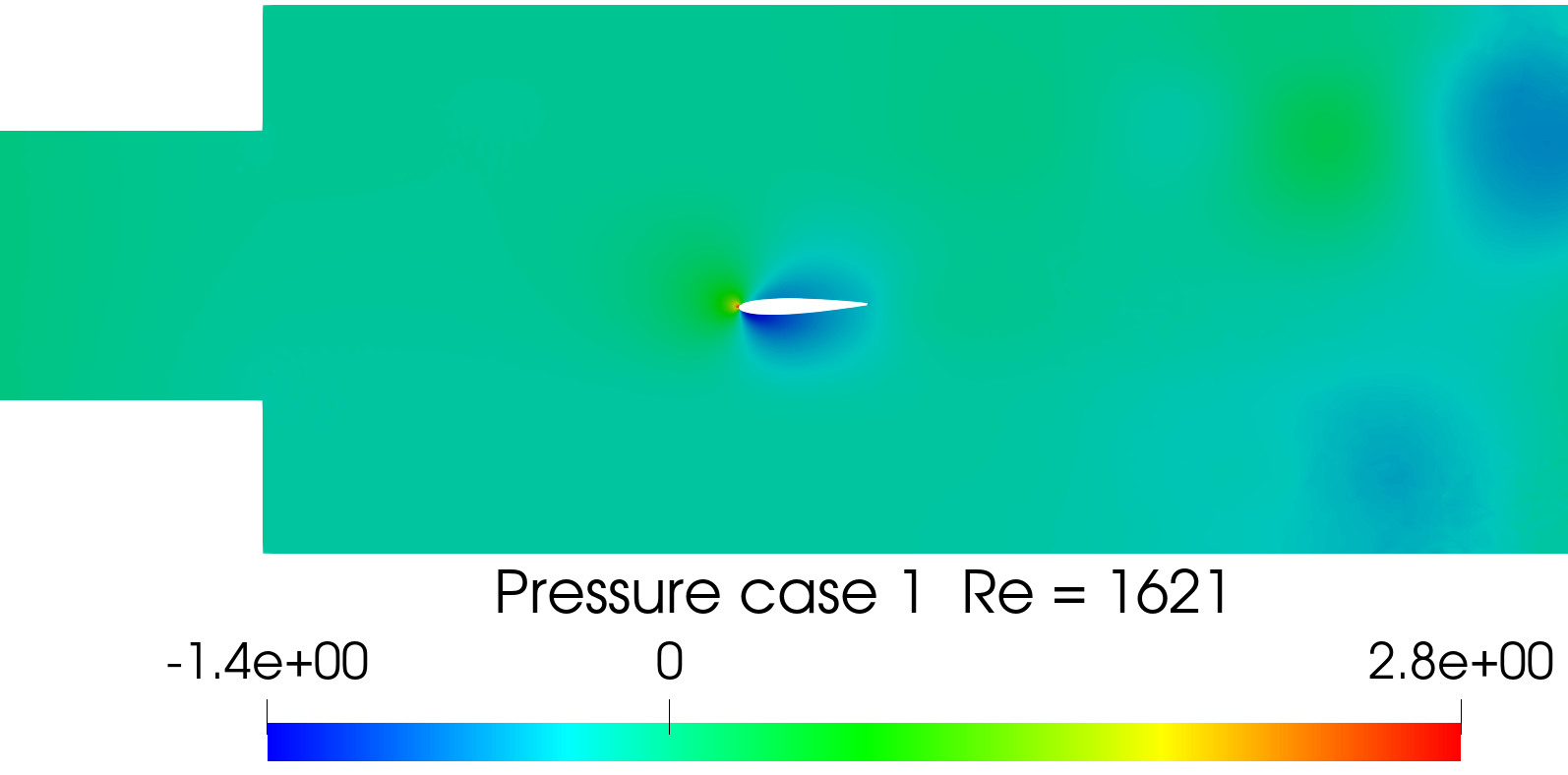}\hfill
\includegraphics[width=.48\textwidth]{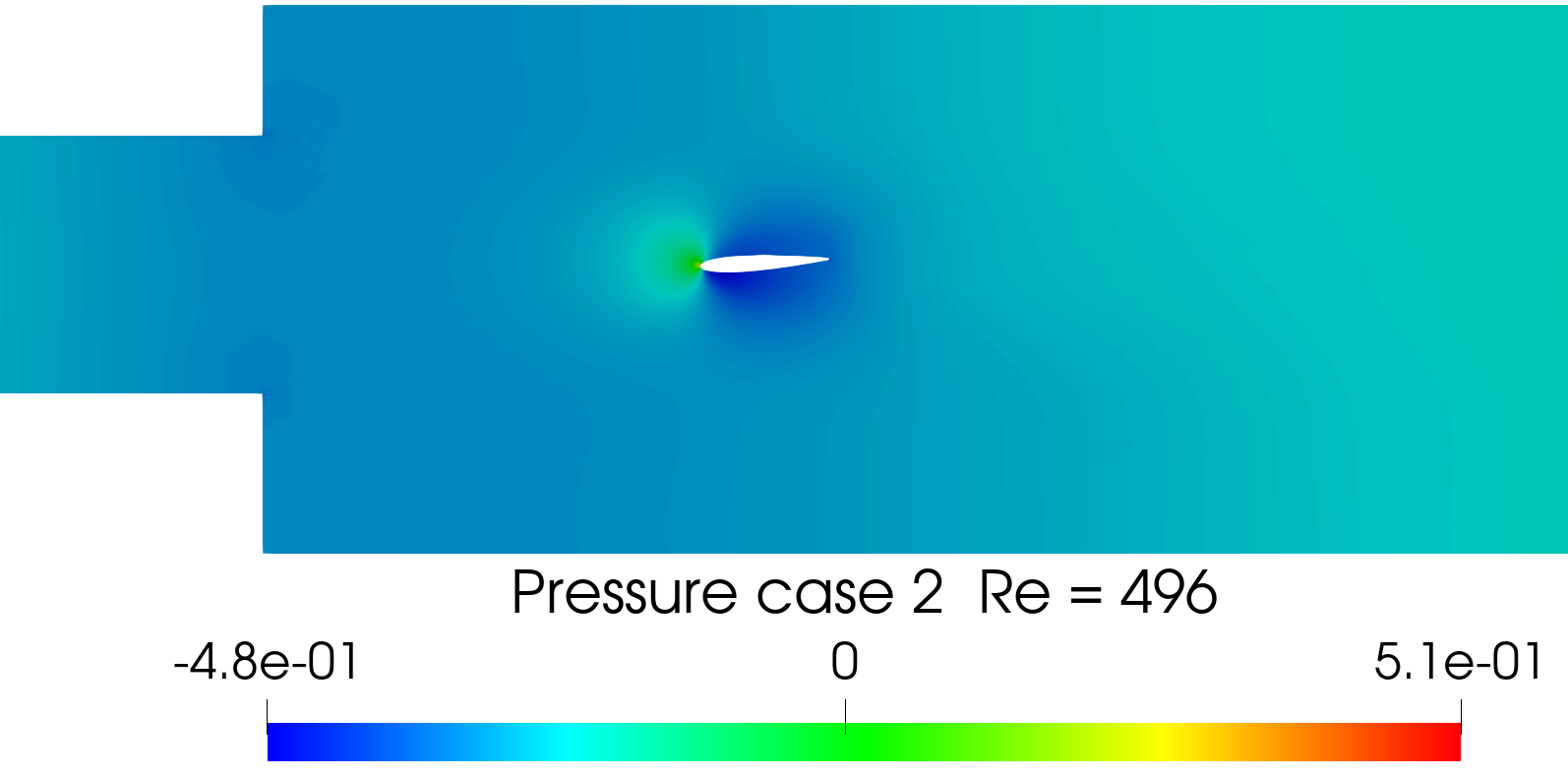}\\

\includegraphics[width=.48\textwidth]{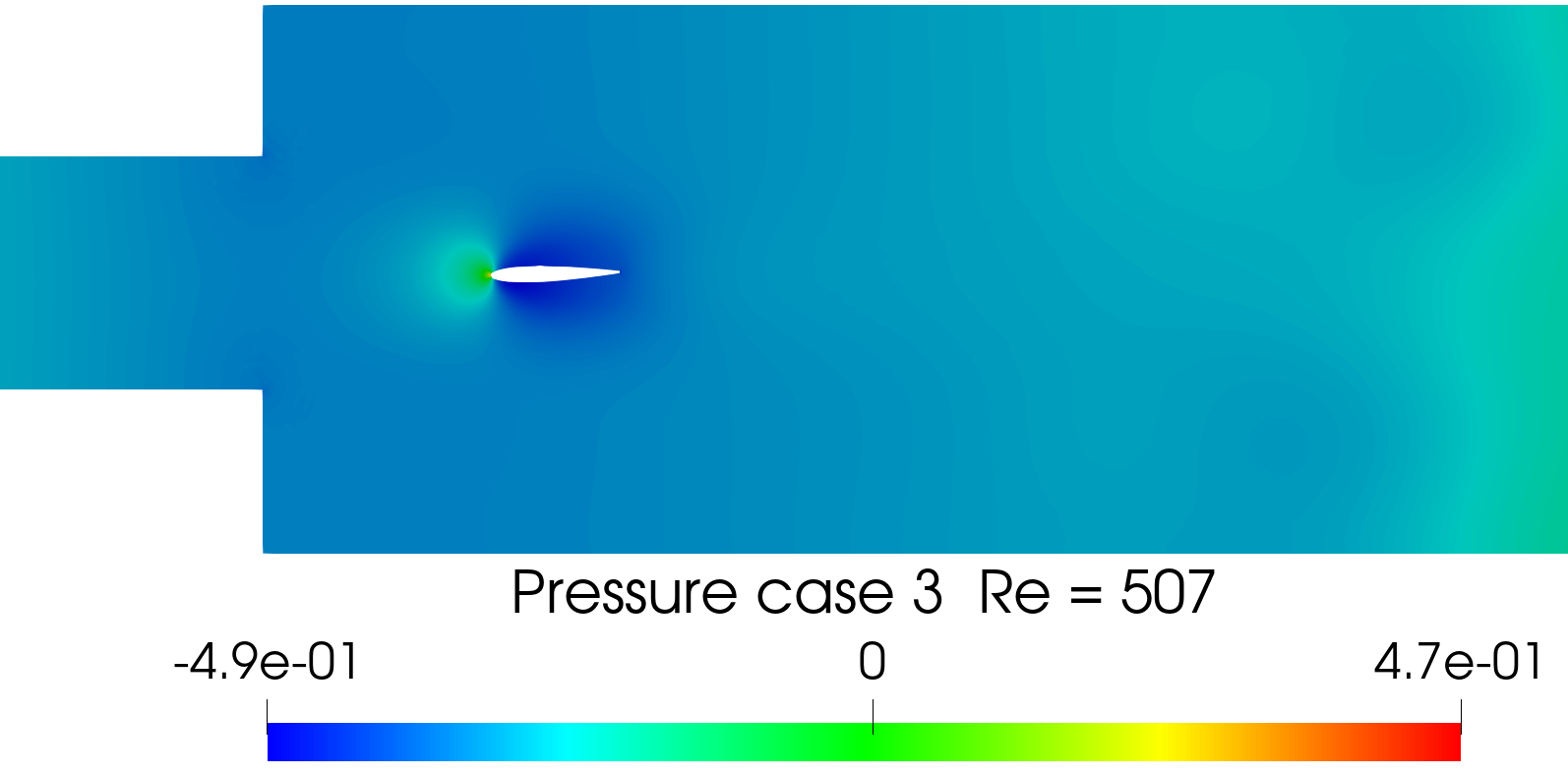}\hfill
\includegraphics[width=.48\textwidth]{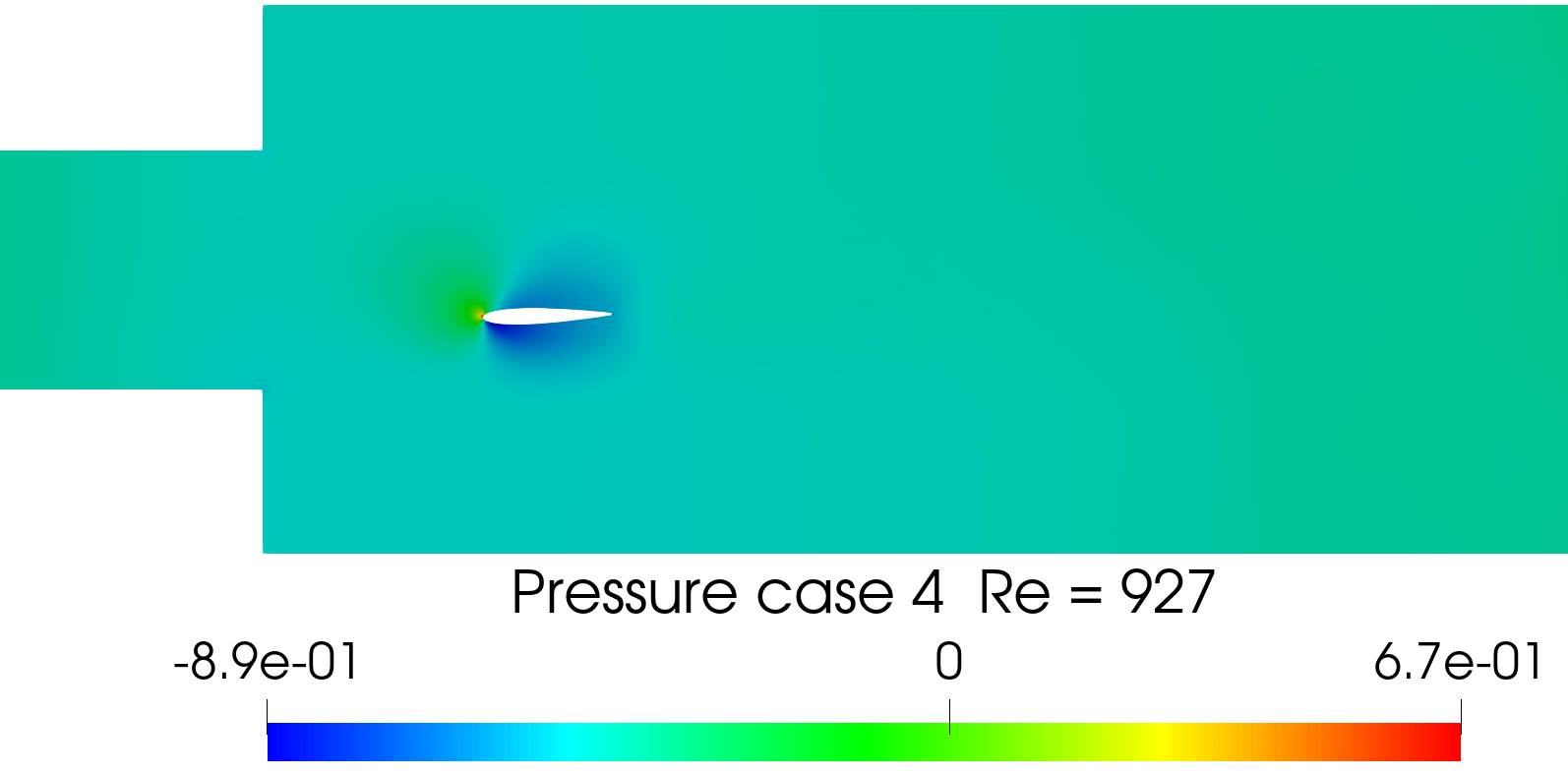}\\
\caption{Module of the velocity fields (on the left) and pressure fields (on the right)
evaluated at the last time instant of $4$ different simulations. The corresponding parameters
are reported in \autoref{table:sim_params}.}
\label{fig:naca_press_vel}
\end{figure}

The lift ($C_{L}$) and drag ($C_{D}$) coefficients are evaluated when
stationary or periodic regimes are reached, starting from the values
of pressure and viscous stresses evaluated on the nodes close to the
airfoil. After this sensitivity analysis is carried out. First the AS
method is applied. The gradients necessary for the application of the
AS method are obtained from the Gaussian process regression of the
model functions $C_{L}$ and $C_{D}$ on the whole parameters'
domain. The eigenvalues of the uncentered covariance matrix for the
lift and drag coefficients suggest the presence of a one-dimensional
active subspace in both cases.

The plots of the first active eigenvector components are useful as
sensitivity measures, see \autoref{fig:evecs_lift_drag}. The greater
the absolute value of a component is, the greater is its influence on
the model function. We observe that the lift coefficient is influenced
mainly by the vertical position of the airfoil and the angle of
attack, while the drag coefficient depends mainly on the initial
velocity, and secondarily on the viscosity and on the angle of
attack.

As one could expect from physical considerations, the angle of attack
affects both drag and lift coefficients, while the viscosity, which
governs the wall stresses, is relevant for the evaluation of the $C_D$.
The vertical position of the airfoil with respect to the symmetric
axis of the section of the duct after the area expansion also greatly
affects both coefficients, and this is mainly due to the fact that the
fluid flow conditions change drastically between the core, where the
speed is higher, and the one close to the wall of the duct, where the
speed tends to zero. On the other hand, the horizontal translation has
almost no impact on the results, given the regularity of the fluid
flow along the $x$-axis for the considered range of $x_0$.
Moreover, the non-symmetric behaviour of the upper and lower
parameters which determine the opening of the channel is due to the
non-symmetric choice of the range considered for the angle of attack.

\begin{figure}[htb!]
\centering
\includegraphics[width=0.49\textwidth]{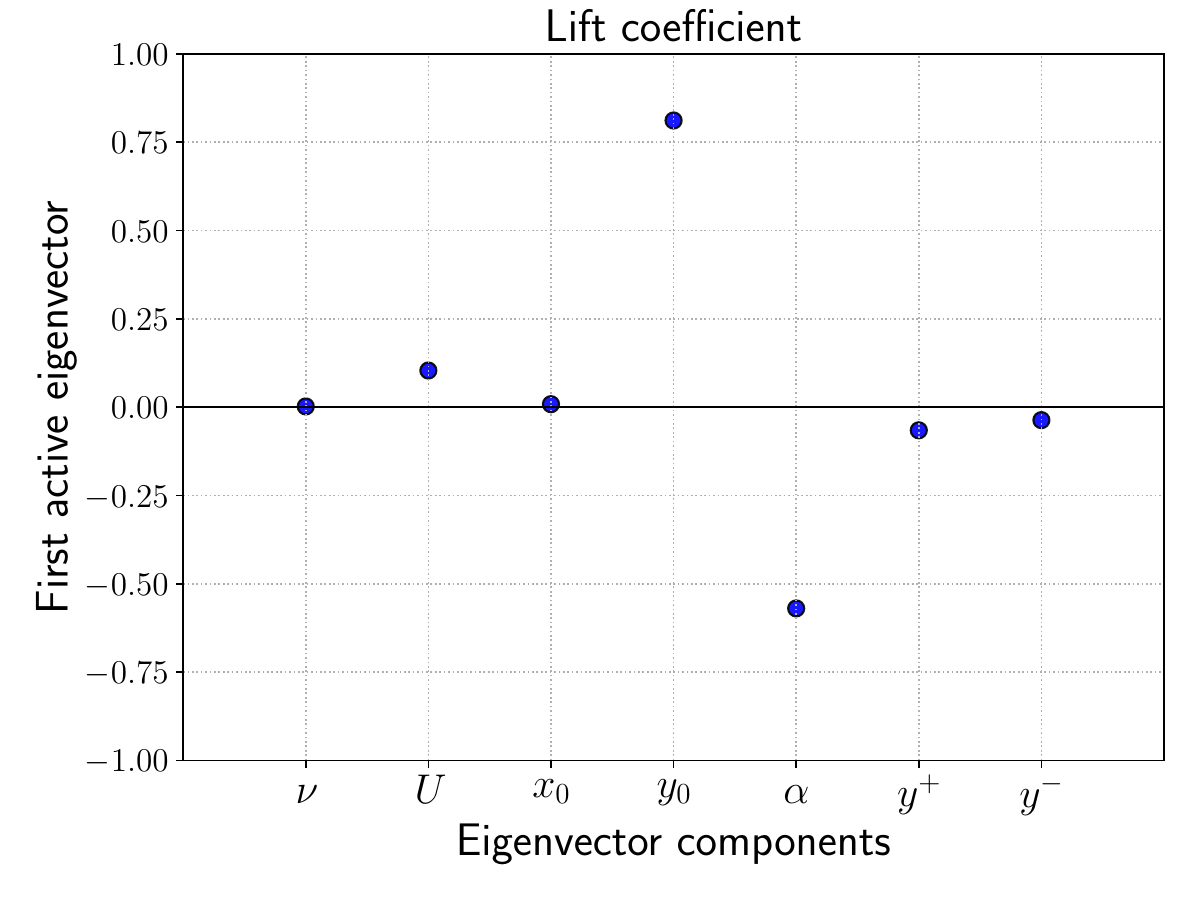}\hfill
\includegraphics[width=0.49\textwidth]{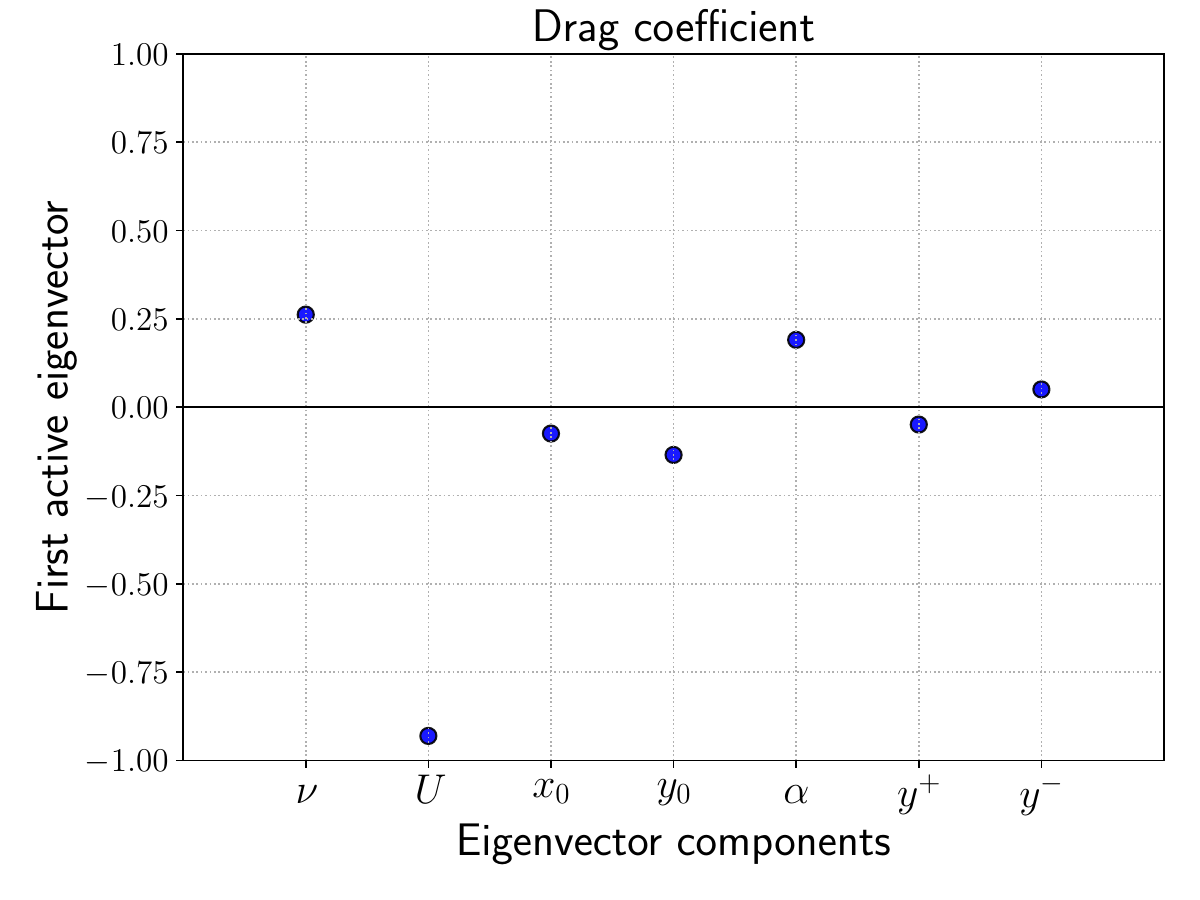}
\caption{Components of the first active eigenvector for the lift
  coefficient (on the left), and for the drag
  coefficient (on the right). Values near $0$ suggest little sensitivity for the
  target function.}
\label{fig:evecs_lift_drag}
\end{figure}

The KAS method was applied with $1500$ features. In order to compare
the AS and KAS methods $5$-fold cross validation was implemented. The
score of cross validation is the relative root mean square error
(RRMSE) defined in \autoref{eq:RRMSE_discrete}.

\begin{figure}[htb!]
\centering
\includegraphics[width=0.49\textwidth]{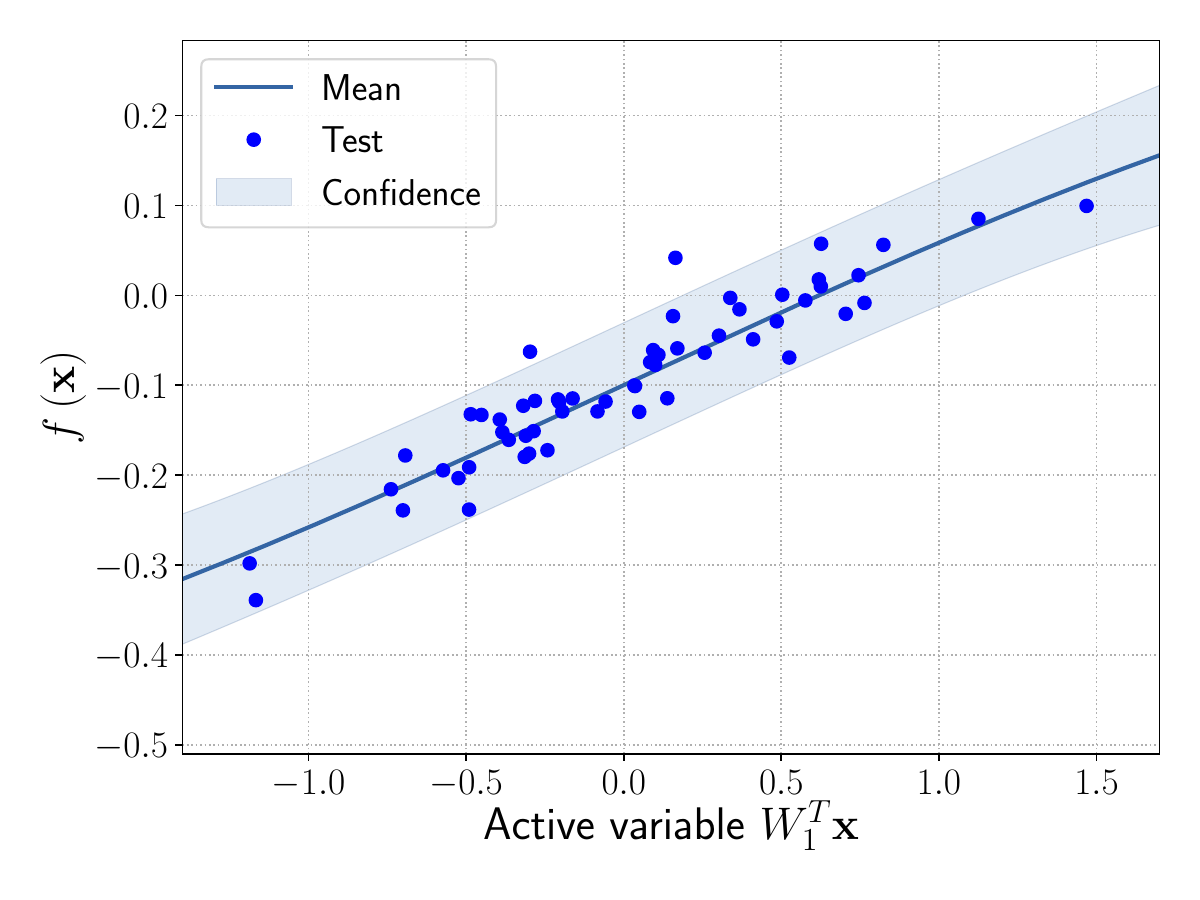}\hfill
\includegraphics[width=0.49\textwidth]{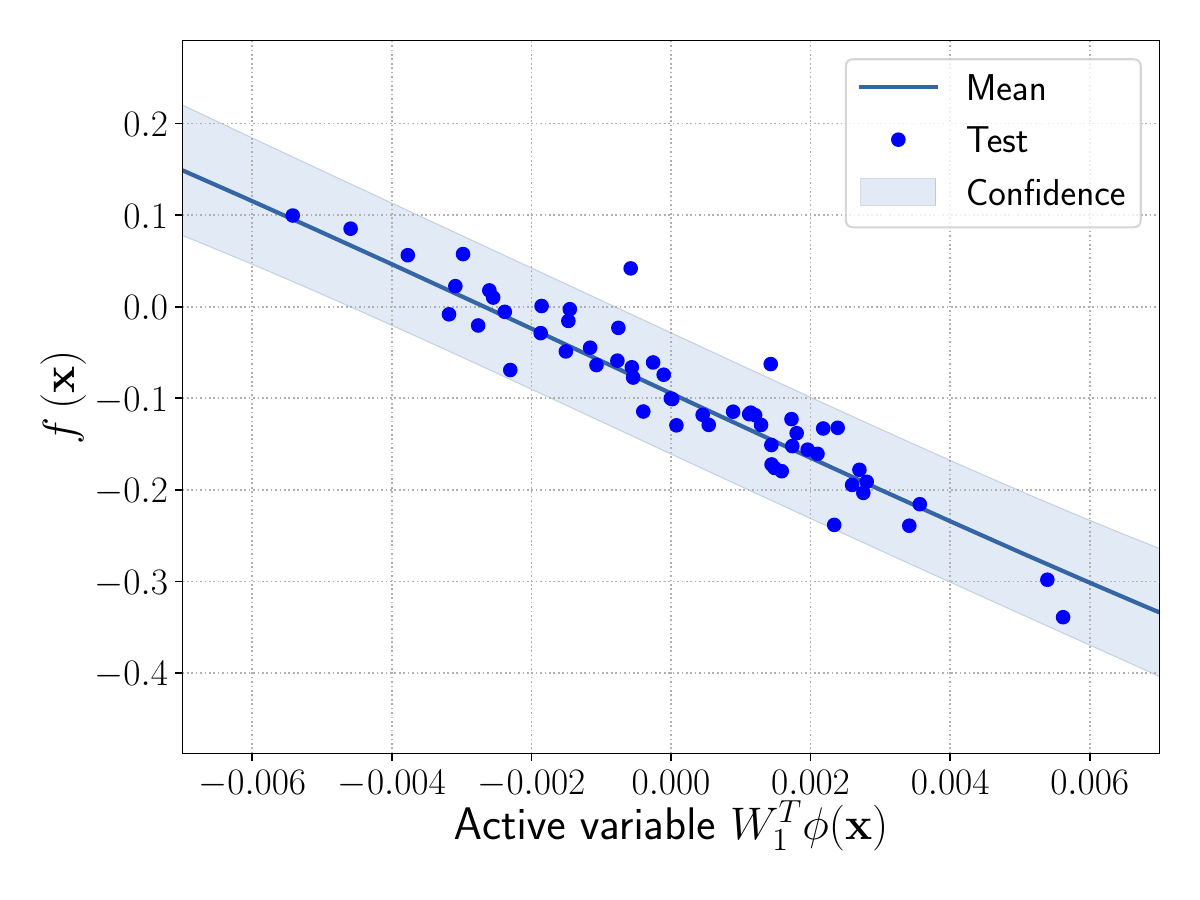}
\caption{Comparison between the sufficiency summary plots obtained
  from the application of AS and KAS methods for the lift coefficient
  $C_L$ defined in \autoref{eq:Lift}. The left plot refers to AS,
  the right plot to KAS. With the blue solid line we depict the
  posterior mean of the GP, with the shadow area the c$68\%$ confidence intervals, and with the blue dots the testing points.}
\label{fig:ssp_lift}
\end{figure}

\begin{figure}[htb!]
\centering
\includegraphics[width=0.49\textwidth]{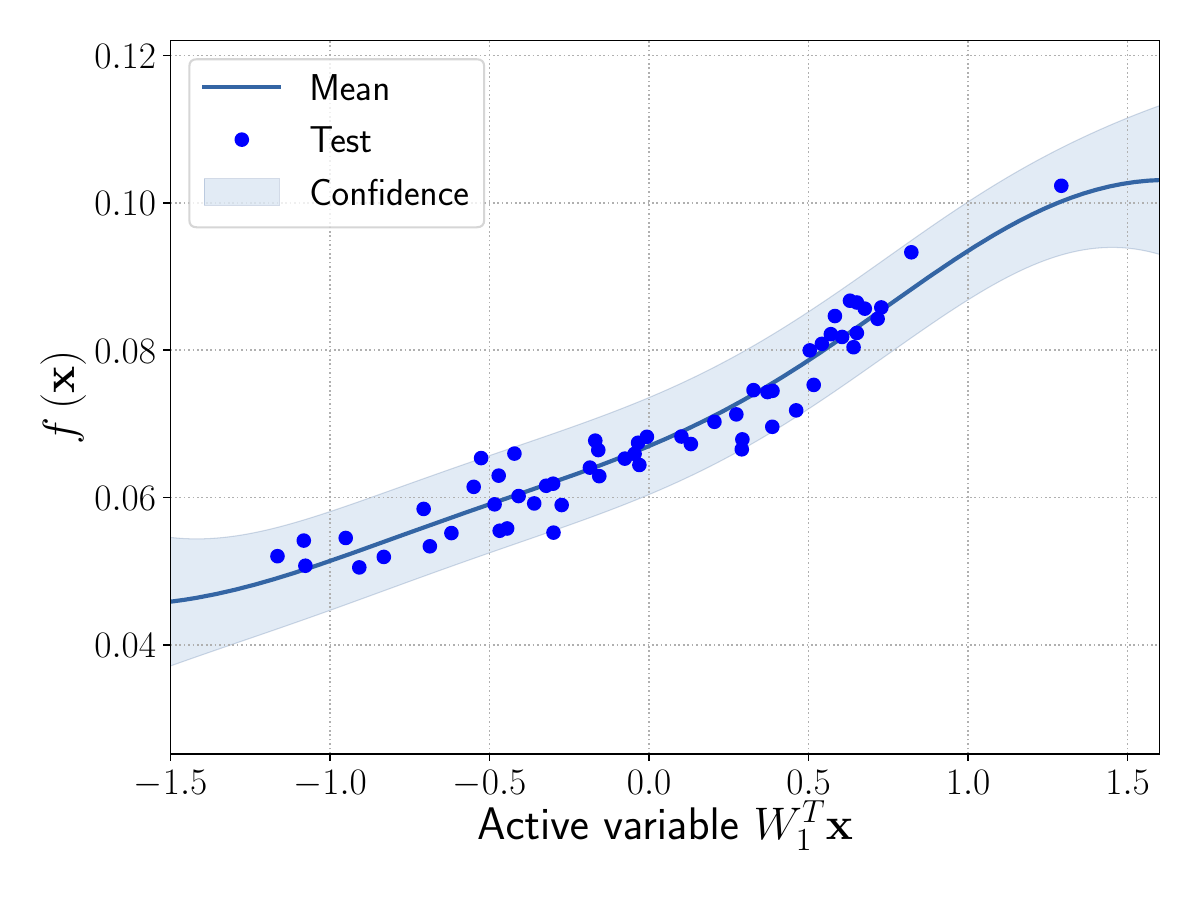}\hfill
\includegraphics[width=0.49\textwidth]{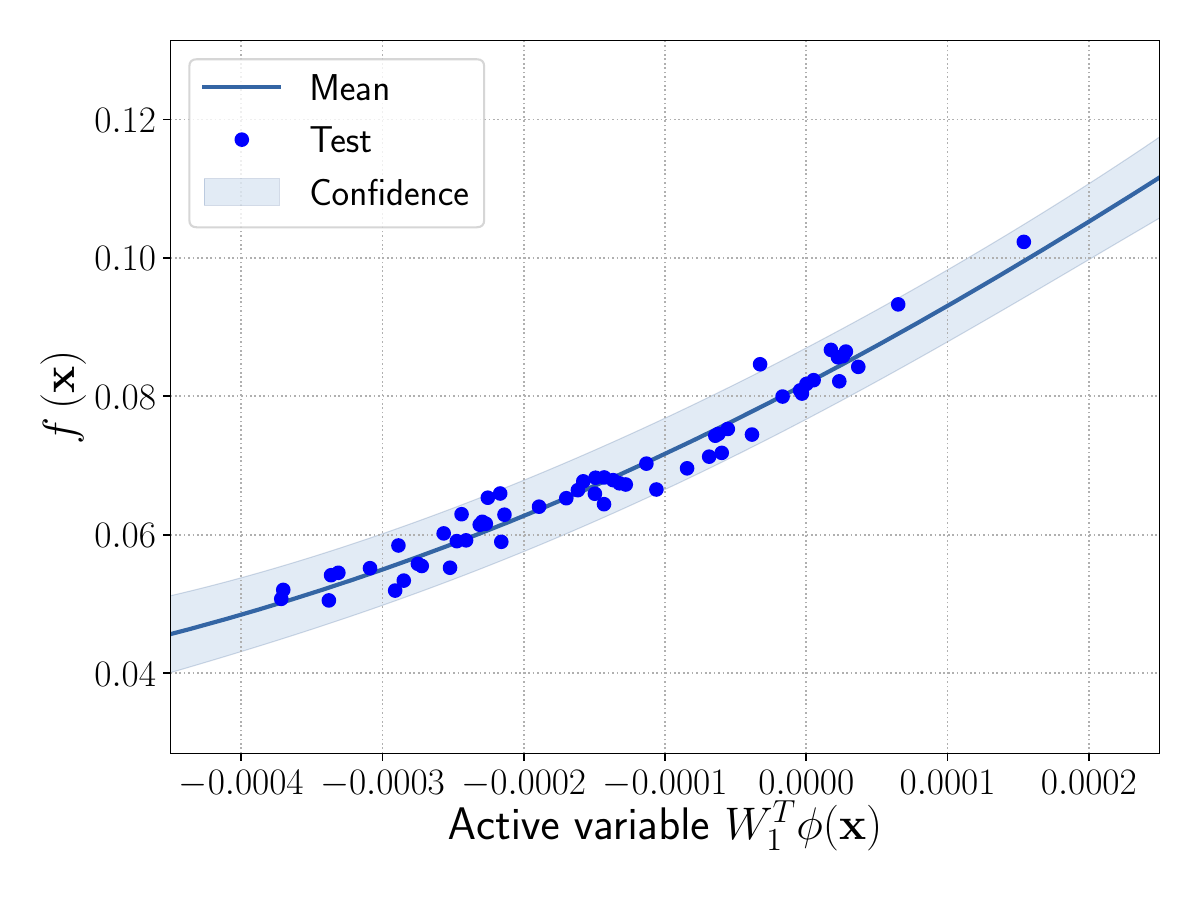}
\caption{Comparison between the sufficiency summary plots obtained
  from the application of AS and KAS methods for the drag coefficient
  $C_D$ defined in \autoref{eq:Drag}. The left plot refers to AS,
  the right plot to KAS. With the blue solid line we depict the
  posterior mean of the GP, with the shadow area the c$68\%$ confidence intervals, and with the blue dots the testing points.}
\label{fig:ssp_drag}
\end{figure}

The Gaussian process regressions for the two methods are shown in
\autoref{fig:ssp_lift} for the lift coefficient, and in
\autoref{fig:ssp_drag} for the drag coefficient. They were obtained as
a single step of $5$-fold cross validation with one fifth of the $285$
samples used as test set. The spectral distribution of the feature map
is the Gaussian distribution for the lift, and the Beta for the drag, respectively.
The RRMSE mean and standard deviation from $5$-fold cross validation,
are reported for different active dimensions in
\autoref{tab:rms_naca}. The feature map from Equation~\eqref{eq:feature_map} was
adopted. The hyperparameters of the spectral distributions were tuned
with logarithmic grid-search with $5$-fold cross validation as
described in Algorithm~\ref{algo: tuning}.

Regarding the drag coefficient, the relative gain using the KAS method reaches
the $19.2$\% on average when employing the Beta spectral measure for the
definition of the feature map. The relative gain of the one
dimensional response surface built with GPR from the KAS method is
$7$\% on average for the lift coefficient. This result could be
due to the higher noise in the evaluation of the $C_{L}$. In this case
the relative gain increases when the dimension of the
response surface increases to $2$ with a gain of $14.6$\%. A
slight reduction of the AS RRMSE relative to the drag
coefficient is ascertained when increasing the dimension of the response surface.

\begin{table}[htb!]
\centering
\caption{Summary of the results for AS and KAS procedures. In bold the best results.\label{tab:rms_naca}}
\begin{tabular}{ c c c c c c c }
\hline
\hline
\multirow{2}{*}{Method} &  \multirow{2}{*}{Dim} & Feature & Lift spectral & \multirow{2}{*}{RRMSE Lift} & Drag spectral &\multirow{2}{*}{RRMSE Drag} \\
  & & space dim & distribution & & distribution & \\
\hline
\hline
\rowcolor{Gray}
AS  & 1 & - & - & 0.37 $\pm$ 0.09 & - & 0.268 $\pm$ 0.032 \\
KAS & 1 & 1500 & $\mathcal{N}(0, \lambda I_{d})$& \textbf{0.344} $\pm$ 0.048 & $\text{Beta}(\alpha, \beta)$ & \textbf{0.218} $\pm$ 0.045 \\
\hline
\rowcolor{Gray}
AS  & 2 & - & - & 0.384 $\pm$ 0.073 & - & 0.183 $\pm$ 0.027 \\
KAS & 2  & 1500 & $\mathcal{N}(0, \lambda I_{d})$ & \textbf{0.328} $\pm$ 0.071  & $\text{Beta}(\alpha, \beta)$ & \textbf{0.17} $\pm$ 0.02 \\
\hline
\hline
\end{tabular}
\end{table}

\section{Conclusions and perspectives}
\label{sec:the_end}
In this work we presented a new nonlinear extension of the active
subspaces property that introduces Kernel-based Active Subspaces (KAS). The
method exploits random Fourier features to find active
subspaces on high-dimensional feature spaces. We
tested the new method over $5$ different benchmarks of increasing
complexity, and we provided pseudo-codes for every aspects of the
proposed kernel-extension. The
tested model functions range from scalar to vector-valued. We also
provide a CFD application discretized by the Discontinuous Galerkin
method. We compared the kernel-based active subspaces to the standard
linear active subspaces and we observed in all the cases an increment
of the accuracy of the Gaussian response surfaces built over the reduced
parameter spaces.
The most interesting results regard the possibility to apply the KAS method
when an active subspace does not exist. This was shown
for radial symmetric model functions.

Future developments will involve the study of more efficient
procedures for tuning the hyperparameters of the spectral
distribution. Other possible advances could be done finding an
effective back-mapping from the targets to the actual parameters in
the full original space. This could promote the implementation of
optimization algorithms or other parameter studies enhanced by the
kernel-based active subspaces extension.

\appendix
\section{Appendix - Proof details}
\label{sec:appendix}

In this section we provide an expanded version of the proof of \autoref{theo:existence}.

\begin{proof}[Proof of \autoref{theo:existence}: existence of an active subspace]
The proof is remodeled from~\cite{parente2020generalized,
  zahm2020gradient}, and it is developed in five steps:
\begin{enumerate}
\item Since $R_{V}\in\mathcal{M}(d,d)$ is symmetric positive definite
  there exists a basis of eigenvectors
  $(\mathbf{w}_{i})_{i\in\{1,\dots, d\}}$ and a corresponding set of positive
  eigenvalues $(\beta_{i})_{i\in\{1,\dots, d\}}$ such that
\label{eq:metric_decomposition}
\begin{equation}
R_{V}=\sum^{d}_{i=1}\beta_{i}\,\mathbf{w}_{i}\otimes\mathbf{w}_{i}.
\end{equation}
\item Let us define the ridge approximation error as
\begin{gather}
    \mathbf{e}=\lVert \mathbf{f}-\mathbf{h}\circ
P_{r}\rVert_{L^{2}(\mathbb{R}^{m}, \mathcal{B}(\mathbb{R}^{m}),
  \rho;V)} = \RA{ \RA{\mathbb{E}_{P}}\left[ \lVert(
  \mathbf{f}(\mathbf{X})-\mathbf{h}(P_{r}(\mathbf{X}))\rVert^{2}_{R_V}\right].
}
\end{gather}
Then we can decompose the error analysis for each component employing
the spectral decomposition~\eqref{eq:metric_decomposition}
\begin{align}
  \mathbb{E}_{P} \left[
\RA{  \lVert\mathbf{e}(\mathbf{X})\rVert^{2}_{R_V} }
  \right] &= \RA{\mathbb{E}_{P}}\left[\text{tr}(\left(R_{V}
\mathbf{e}(\mathbf{X})\right) \otimes\mathbf{e} (\mathbf{X})) \right]=
 \nonumber \\
&=\sum^{d}_{i=1}\beta_{i}\,\RA{\mathbb{E}_{P}}\left[\text{tr}(((\mathbf{w}_{i}
\otimes\mathbf{w}_{i})\mathbf{e}(\mathbf{X}))\otimes\mathbf{e}(\mathbf{X}))\right]=
  \nonumber \\
&=\sum^{d}_{i=1}\beta_{i}\,\RA{\mathbb{E}_{P}}\left[(\mathbf{w}_{i}\cdot\mathbf{e}
(\mathbf{X}))\,\text{tr}(\mathbf{w}_{i}\otimes\mathbf{e}(\mathbf{X}))\right]=
  \nonumber \\
&=\sum^{d}_{i=1}\beta_{i}\,\RA{\mathbb{E}_{P}}\left[(\mathbf{w}_{i}\cdot\mathbf{e}(\mathbf{X}))^{2}\right],
\end{align}
so we can define
$e_{i}(\mathbf{X})=\mathbf{w}_{i}\cdot\mathbf{e}(\mathbf{X})=f_{i}(\mathbf{X})-h_{i}(P_{r}(\mathbf{X})),\,\forall
i\in\{1, \dots, d\}$ and treat each component separately.
\item The next step involves the application of
  Lemma~\autoref{lemma:subspace_P_inequality} to the scalar functions
  $f_{i}(\mathbf{X})-h_{i}(P_{r}(\mathbf{X})),\,\forall i\in\{1,
  \dots, d\}$
\begin{align}
  \label{eq:poincarè_step}
\RA{\mathbb{E}_{P}}\left[\left(f_{i}(\mathbf{X})-h_{i}(P_{r}(\X))\right)^{2}\right]
  & =
    \RA{\mathbb{E}_{P}}\left[\RA{\mathbb{E}_{P}}\left[\left(f_i
    (\X)-h_{i}(P_{r} (\X))\right)^{2}|\sigma(P_{r})\right]\right]
  \nonumber \\
&\leq\RA{\mathbb{E}_{P}}\left[C_{p}(\rho, P_{r}(\mathbf{X}))
                 \RA{\mathbb{E}_{P}}\left[\lVert(I-P_{r}^{T})\nabla
                 f_{i}(\mathbf{X})\rVert^{2}_{2}|\sigma(P_{r})\right]\right]
  \nonumber \\
&\leq \RA{\mathbb{E}_{P}}\left[C_{p}(\rho,
                 P_{r}(\mathbf{X}))\right]^{\frac{1}{p}}\,
                 \RA{\mathbb{E}_{P}}\left[\lVert(I - P_{r}^{T})\nabla
                 f_{i}(\mathbf{X})\rVert^{2}_{2}\right]^{\frac{1}{q}} ,
\end{align}
where we used the H{\"o}lder inequality with indexes $(p, q)=(\infty,
1)$ when $\rho$ belongs to the first and second classes of
Assumption~\autoref{ass:pdf}, and $(p, q)=(\frac{\tau+1}{\tau}, 1+\tau)$
when $\rho$ belongs to the third class.

Then we can bound
$\RA{\mathbb{E}_{P}}\left[(C_{p},
  P_{r}(\mathbf{X}))\right]^{\frac{1}{p}}$ with a constant
$C(C_{p}(\rho, P_{r}(\mathbf{X})))$ which depends on the class of
$\rho$ (see Lemma 3.1, Lemma 4.2, Lemma 4.3, Lemma 4.4 and Theorem 4.5
of~\cite{parente2020generalized}) as follows
\begin{align}
&\RA{\mathbb{E}_{P}} \left[C_{p}(\rho,
                P_{r}(\mathbf{X}))\right]^{\frac{1}{p}} \,
                \RA{\mathbb{E}_{P}}\left[\lVert(I - P_{r}^{T}) \nabla
                f_{i}(\mathbf{X})\rVert^{2}_{2}
                \right]^{\frac{1}{q}}\leq \nonumber \\
&\leq C(C_{p}(\rho, P_{r}(\mathbf{X})))\, \text{tr}(\RA{\mathbb{E}_{P}}
\left[(I - P_r^T)\nabla f_i (\X)(\nabla f_i(\X))^T (I -P_r) \right]^{\frac{1}{q}})= \nonumber \\
&= C(C_{p}(\rho, P_{r}(\mathbf{X})))\, \text{tr}((I - P_{r}^{T})
\RA{\mathbb{E}_{P}}\left[\nabla f_{i}(\X)(\nabla f_{i}(\X))^{T} \right](I - P_{r}))^{\frac{1}{q}}.
\end{align}

\item The spectral decomposition~\eqref{eq:metric_decomposition} is
  employed again and the covariance matrix $H$ is introduced in the
  last equation
\begin{align}
&\RA{\mathbb{E}_{P}} \left[
\RA{                \lVert\mathbf{e}(\mathbf{X})\rVert^{2}_{R_V} } \right]\leq  \nonumber \\
&\leq\sum^{d}_{i=1}\beta_{i}\,C(C_{p}(\rho, P_{r}(\mathbf{X})))\, \text{tr}((I - P_{r}^{T})\RA{\mathbb{E}_{P}}\left[((\nabla \mathbf{f}(\mathbf{X}))^{T}\mathbf{w}_{i})\otimes((\nabla\mathbf{f}(\mathbf{X}))^{T}\mathbf{w}_{i})\right](I - P_{r}))^{\frac{1}{q}}= \nonumber \\
&=C(C_{p}(\rho, P_{r}(\mathbf{X})))\, \text{tr}((I - P_{r}^{T})\RA{\mathbb{E}_{P}}\left[(\nabla \mathbf{f}(\mathbf{X}))^{T}\left(\sum^{d}_{i=1}\beta_{i}^{q}\,\mathbf{w}_{i}\otimes \mathbf{w}_{i}\right)\nabla \mathbf{f}(\mathbf{X})\right](I - P_{r}))^{\frac{1}{q}}= \nonumber \\
&=C(C_{p}(\rho, P_{r}(\mathbf{X})))\, \text{tr}((I - P_{r}^{T})\RA{\mathbb{E}_{P}}\left[(\nabla \mathbf{f}(\mathbf{X}))^{T}R_{V}(\rho)\nabla \mathbf{f}(\mathbf{X})\right](I - P_{r}))^{\frac{1}{q}}= \nonumber \\
&=C(C_{p}(\rho, P_{r}(\mathbf{X})))\, \text{tr}((I - P_{r}^{T})H(I - P_{r}))^{\frac{1}{q}},
\end{align}
where $R_{V}(\rho)$ is the original metric matrix if $\rho$ belongs to the first or second class of Assumption~\ref{ass:pdf} and is equal to
\begin{equation}
\sum^{d}_{i=1}\beta_{i}^{1+\tau}\,\mathbf{w}_{i}\otimes \mathbf{w}_{i} ,
\end{equation}
if $\rho$ belongs to the third class.
\item Finally the bound in the statement of the theorem is recovered
  solving the following minimization problem with classical model
  reduction arguments employing singular value decomposition (SVD)
\begin{equation}
\Tilde{P}_{r} = \argmin_{P_{r}\in\mathcal{O}(m, m)}\text{tr}((I - P_{r}^{T})H(I - P_{r})).
\end{equation}
\end{enumerate}
\end{proof}

\section*{Acknowledgements}
This work was partially supported by an industrial Ph.D. grant sponsored
by Fincantieri S.p.A. (IRONTH Project), by MIUR (Italian ministry for
university and research)
through FARE-X-AROMA-CFD project, and partially funded by European
Union Funding for Research and Innovation --- Horizon 2020 Program --- in
the framework of European Research Council Executive Agency: H2020 ERC
CoG 2015 AROMA-CFD project 681447 ``Advanced Reduced Order Methods with
Applications in Computational Fluid Dynamics'' P.I. Professor Gianluigi Rozza.

\bibliographystyle{abbrvurl}

\end{document}